\documentclass[11pt]{article}

\usepackage[utf8]{inputenc}
\usepackage[english]{babel}
\usepackage{hyperref}
\usepackage{fullpage}
\usepackage{url}
\usepackage{booktabs}
\usepackage{amsmath,amsthm,amsfonts}
\usepackage[usenames]{color}
\usepackage{xcolor}
\usepackage{multirow}
\usepackage{float}
\usepackage[usenames]{color}
\usepackage{rotating}
\usepackage{boxedminipage}
\usepackage{algorithm}
\usepackage{algorithmic}
\usepackage{natbib}
\usepackage{graphicx,psfrag,epsf}
\usepackage{enumerate}
\usepackage{authblk}
\usepackage{epstopdf}
\usepackage{subfigure}
\definecolor{DSgray}{cmyk}{0,1,0,0}

\newtheorem*{lemmax}{Proposition~\ref{lemma:x}}
\usepackage{thmtools}
\usepackage{thm-restate}

\newtheorem{innercustomthm}{Theorem}
\newenvironment{customthm}[1]
{\renewcommand\theinnercustomthm{#1}\innercustomthm}
{\endinnercustomthm}

\usepackage{rotating}
\usepackage{subfigure}
\newcommand{\blind}{1}

\renewcommand{\baselinestretch}{1.7}
\hypersetup{
colorlinks=true,
citecolor=blue,
anchorcolor = red,
}
\newtheorem{theorem}{Theorem}[section]
\newtheorem{lemma}[theorem]{Lemma}
\newtheorem{proposition}[theorem]{Proposition}
\newtheorem{assumption}{Assumption}

\newtheorem{remark}[theorem]{Remark}
\newtheorem{example}[theorem]{Example}

\makeatletter 
\let\c@table\c@figure
\let\c@lstlisting\c@figure
\makeatother

\newcommand{\tr}{\mathrm{tr}}
\newcommand{\E}{\mathbb{E}}
\renewcommand{\P}{\mathbb{P}}
\renewcommand{\O}{\mathcal{O}}
\newcommand{\R}{\mathbb{R}}
\newcommand{\N}{\mathcal{N}}
\newcommand{\diag}{\mathrm{diag}}
\newcommand{\Fro}{\mathrm{Fro}}
\newcommand{\intd}{\,\mathrm{d}}
\newcommand{\Qg}{Q^{\tg}}
\newcommand{\Qs}{Q^{\ts}}
\newcommand{\Qi}{Q^{\ti}}
\newcommand{\Qu}{Q^{\tu}}
\newcommand{\Qp}{Q^{\tp}}
\newcommand{\tg}{\tt{(G)}}
\newcommand{\ts}{\tt{(S)}}
\newcommand{\ti}{\tt{(I)}}
\newcommand{\tu}{\tt{(U)}}
\newcommand{\tp}{\tt{(P)}}
\newcommand{\tnr}{\tt{WOR}}
\newcommand{\kw}{{\tt{(KW)}}}
\newcommand{\akw}{{\tt{(AKW)}}}
\renewcommand{\rm}{{\tt{(RM)}}}

\newcommand{\bxi}{\boldsymbol{\xi}}
\newcommand{\bgamma}{\boldsymbol{\gamma}}
\newcommand{\bdelta}{\boldsymbol{\delta}}
\newcommand{\bvarepsilon}{\boldsymbol{\varepsilon}}
\renewcommand{\a}{\boldsymbol{a}}
\renewcommand{\b}{\boldsymbol{b}}
\newcommand{\Pv}{\mathcal{P}_{\v}}
\newcommand{\Pz}{\mathcal{P}_{\zet}}

\newcommand{\e}{\boldsymbol{e}}
\renewcommand{\u}{\boldsymbol{u}}
\renewcommand{\v}{\boldsymbol{v}}
\newcommand{\w}{\boldsymbol{w}}
\newcommand{\tth}{\boldsymbol{\theta}}
\newcommand{\zet}{\boldsymbol{\zeta}}
\newcommand{\X}{\boldsymbol{X}}
\newcommand{\x}{\boldsymbol{x}}
\newcommand{\z}{\boldsymbol{z}}
\newcommand{\y}{\boldsymbol{\theta'}}

\newcommand{\0}{\mathbf{0}}
\newcommand{\1}{\mathbf{1}}

\newcommand{\red}{}

 \definecolor{DSgray}{cmyk}{0,1,0,0}

\let\hat\widehat

\begin{document}

\if1\blind
{
\title{Online Statistical Inference for Stochastic Optimization via Kiefer-Wolfowitz Methods}
\author{Xi Chen\footnote{Stern School of Business, New York University, Email: xc13@stern.nyu.edu} ~ Zehua Lai\footnote{Committee on Computational and Applied Mathematics, University of Chicago, Email: laizehua@uchicago.edu} ~ He Li\footnote{Stern School of Business, New York Universitwy, Email: hli@stern.nyu.edu} ~ Yichen Zhang\footnote{Krannert School of Management, Purdue University, Email: zhang@purdue.edu}}
\date{}
  \maketitle
} \fi
\if0\blind
{
\title{Online Statistical Inference for Stochastic Optimization via Kiefer-Wolfowitz Methods}
\author{}
\date{}
  \maketitle
} \fi
\vspace{-1.5em}

\begin{abstract}

This paper investigates the problem of online statistical inference of model parameters in stochastic optimization problems via the Kiefer-Wolfowitz algorithm with random search directions. We first present the asymptotic distribution for the Polyak-Ruppert-averaging type Kiefer-Wolfowitz (AKW) estimators, whose asymptotic covariance matrices depend on the distribution of search directions and the function-value query complexity. The distributional result reflects the trade-off between statistical efficiency and function query complexity. We further analyze the choice of random search directions to minimize certain summary statistics of the asymptotic covariance matrix. Based on the asymptotic distribution, we conduct online statistical inference by providing two construction procedures of valid confidence intervals.
\end{abstract}


\section{Introduction}
\label{sec:intro}

Stochastic optimization algorithms, introduced by \cite{robbins1951stochastic,Kiefer1952}, have been widely used in statistical estimation, especially for large-scale datasets and online learning where the sample arrives sequentially (e.g., web search queries, transactional data). The Robbins-Monro algorithm \citep{robbins1951stochastic}, often known as the stochastic gradient descent, is perhaps the most popular algorithm in stochastic optimization and has found a wide range of applications in statistics and machine learning. 
Nevertheless, in many modern applications, the gradient information is not available. For example, the objective function may be embedded in a black box and the user can only access the noisy objective value for a given input.  In such cases,  the Kiefer-Wolfowitz algorithm \citep{Kiefer1952} becomes a natural choice as it is completely free of gradient computation. Despite being equipped with an evident computational advantage to avoid gradient measurements,  the Kiefer-Wolfowitz algorithm  has been historically out of practice as compared to the Robbins-Monro counterpart. 
Nonetheless, heralded by the big data era, there has been a restoration of the interest of gradient-free optimization in a wide range of applications in recent years \citep{conn2009introduction,nesterov2017random}. We briefly highlight a few of them to motivate our paper. 
\begin{itemize} 
  \item  In some bandit problems, one may only have black-box access to individual objective values but not to their gradients \citep{flaxman2005online,shamir2017optimal}. Other examples include graphical models and variational inference problems, where the objective is defined variationally \citep{wainwright2008graphical}, and the explicit differentiation can be difficult.
  
  \item In some scenarios, the computation of gradient information is possible but very expensive. For example, in the online sensor selection problem \citep{joshi2008sensor}, evaluating the stochastic gradient requires the inverse of matrices, which generates $\O(d^3)$ computation cost per iteration, where $d$ is the number of sensors in the network. In addition, the storage for gradient calculation also requires an $O(d^3)$ memory, which could be practically infeasible. 
\end{itemize}

This paper aims to study the asymptotic properties of the Kiefer-Wolfowitz stochastic optimization and conduct online statistical inference. 
In particular, we consider the problem,
\begin{align}\label{eq:model}
\tth^\star = \mathrm{argmin}~F(\tth), \quad \text{where} \; F(\tth) := \E_{\Pz} \left[f(\tth; \zet)\right]=\int f(\tth;\zet)\text{d}\Pz, \end{align}
where $f(\tth; \zet)$ is a convex \emph{individual loss function} for a data point $\zet$, $F(\tth)$ is the \emph{population loss} function, and $\tth^\star$ is the true underlying parameter of a fixed dimension $d$.
Let $\tth_0$ denote any given initial point. Given a sequentially arriving online sample $\{\zet_n\}$, the \cite{robbins1951stochastic} algorithm {\rm}, also known as the stochastic gradient descent (SGD),  iteratively updates,
\begin{itemize}
  \item[\rm\label{rm}] \hfill%
    $\tth_{n}^\rm=\tth_{n-1}^\rm-\eta_n g(\tth_{n-1};\zet_n)$,%
  \hfill\refstepcounter{equation}\label{eq:rm}\textup{(\theequation)}%
\end{itemize}
where $\{\eta_n\}$ is a positive non-increasing step-size sequence, and $g(\tth; \zet)$ denotes the stochastic gradient, i.e., $g(\tth;\zet)=\nabla f(\tth;\zet)$.
In the scenarios that direct gradient measurements are inaccessible to practitioners, the \cite{Kiefer1952} algorithm  {\kw} becomes the natural choice, as
\begin{itemize}
  \item[\kw\label{kw}] \hfill%
    $\tth_{n}^\kw=\tth_{n-1}^\kw-\eta_n \widehat{g}(\tth_{n-1};\zet_n)$,%
  \hfill\refstepcounter{equation}\label{eq:kw}\textup{(\theequation)}%
\end{itemize}
where $\widehat{g}(\tth_{n-1};\zet_n)$ is an estimator of $g(\tth_{n-1};\zet_n)$. Under the univariate framework ($d=1$), \cite{Kiefer1952} considered the finite-difference approximation
\begin{align}\label{eq:hatq}
\widehat{g}(\theta_{n-1};\zeta_n)=\frac{f(\theta_{n-1}+h_n;\zeta_n)-f(\theta_{n-1};\zeta_n)}{h_n},
\end{align}
where $h_n$ is be a positive deterministic sequence that goes to zero. 
\cite{Blum1954} later extended the algorithm to the multivariate case
and proved its almost sure convergence. This pioneering work extended in various directions of statistics and control theory (see, e.g., \cite{fabian1967stochastic,fabian1980stochastic,hall2014martingale,ruppert1982almost,chen1988lower,polyak1990optimal,spall1992multivariate,chen1999kiefer,spall2000adaptive,hall2003sequential,dippon2003accelerated,mokkadem2007companion,broadie2011general}). In the optimization literature, the Kiefer-Wolfowitz {\kw} algorithm is often referred to as the gradient-free stochastic optimization, or zeroth-order SGD \citep[among others]{agarwal2010optimal,agarwal2011stochastic,Jamieson2012,ghadimi2013stochastic,duchi2015optimal,shamir2017optimal,nesterov2017random,wang2017stochastic}. 

For the {\rm} algorithm in \eqref{eq:rm}, \cite{ruppert1988efficient} and \cite{polyak1992acceleration} characterize the limiting distribution and statistical efficiency of the \emph{averaged iterate} $\overline{\tth}_{n}^{\rm}=\frac{1}{n} \sum_{i=1}^n \tth_i^{\rm}$  by
\begin{align}
	\label{eqpolyak}
	\sqrt{n}\left(\overline{\tth}_{n}^{\rm} - \tth^\star\right) \Longrightarrow \N\left(\0, H^{-1} S H^{-1}\right),
\end{align}
where $H=\nabla^2 F(\tth^\star)$ is the Hessian matrix of $F(\tth)$ at $\tth=\tth^\star$, and $S=\E[\nabla f(\tth^\star;\zet)\nabla f(\tth^\star; \zet)^\top]$ is the Gram matrix of the stochastic gradient. 
Under a well-specified model, this asymptotic covariance matrix matches the inverse Fisher information and the averaged {\rm} estimator is asymptotically efficient. Based on the limiting distribution result \eqref{eqpolyak}, there are many recent research efforts devoted to statistical inference for {\rm}. A brief survey is conducted at the end of the introduction.

For the {\kw} scheme,  we can similarly construct the averaged Kiefer-Wolfowitz {\akw} estimator
\begin{itemize}
	\item[\akw\label{eq:akw}] \hfill%
	$\overline{\tth}_n^\kw=\frac{1}{n} \sum\limits_{i=1}^n \tth_i^\kw$.
	\hfill\refstepcounter{equation}\label{eq:akw}\textup{(\theequation)}
\end{itemize}

As compared to well-established asymptotic properties of {\rm}, study of the asymptotics of {\akw} is limited, particularly with a random sampling direction in multivariate {\kw}. 
In this paper, we study the {\kw} algorithm  \eqref{eq:kw} with random search directions $\{\v_i\}_{i=1}^n \overset{i.i.d.}{\sim} \Pv$, i.e., 
{at each iteration $i=1,2,\dots, n$, a random direction $\v_i$ is sampled independently from $\Pv$, and the {\kw} gradient }
\begin{align}\label{eq:v}
\widehat{g}_{h_n, \v_n}(\tth_{n-1};\zet_n)=\frac{f(\tth_{n-1}+h_n\v_n;\zet_n)-f(\tth_{n-1};\zet_n)}{h_n} \v_n.
\end{align}
Compared to the {\rm} scheme, {\kw} introduces additional randomness into the stochastic gradient estimator through $\{\v_n\}$.  Indeed, as one can see from our main result in Theorem \ref{thm:clt}, {\akw}  is no longer statistically efficient and its asymptotic covariance structure depends on the distribution $\Pv$.  It opens the room for the investigation on the impact of $\Pv$ (see Section \ref{sec:choice_dir} for details). We further extend the estimator to utilize multiple function-value queries per step and establish an online statistical inference framework. We summarize our main results and contributions as follows,

\begin{itemize}
\item First, we quantify the asymptotic covariance structure of \hyperref[eq:akw]{\akw} in Theorem \ref{thm:clt}. Since the asymptotic distribution depends on the choice of the direction variable $\v$, we provide an introductory analysis on the asymptotic performance for different choices of random directions for constructing {\akw} estimators (see Section \ref{sec:choice_dir}).

\item  { The efficiency loss of {\akw} is due to the information constraint as one evaluates only \emph{two} function values at each iteration.
We analyze the {\akw} estimators in which multiple function queries can be assessed at each iteration, and show that the asymptotic covariance matrix  decreases as the number of function queries $m+1$ increases (see Section \ref{sec:multiple}). Moreover, {\akw} achieves asymptotic statistical efficiency as $m\rightarrow \infty$. We further show that when $\v$ is sampled without replacement from $\Pv$ with a discrete uniform distribution of any orthonormal basis,  {\akw}  achieves asymptotic statistical efficiency with $d+1$ function queries per iteration.}

\item { \label{page:intro_inf}Based on the asymptotic distribution, we propose two online statistical inference procedures. The first one is using a plug-in estimator of the asymptotic covariance matrix, which separately estimates the Hessian matrix and Gram matrix of the {\kw} gradients (with additional function-value queries, see Theorem \ref{thm:plugin}). The second procedure is to characterize the distribution of intermediate {\kw} iterates as a stochastic process and construct an asymptotically pivotal statistic by normalizing the {\akw} estimator, without directly estimating the covariance matrix. 
This inference procedure follows the ``random scaling'' method proposed in \cite{lee2021fast} that considers the online inference for the {\rm} scheme.} These two procedures have their advantages and disadvantages: the plug-in approach leads to better empirical performance but requires additional function-value queries to estimate the Hessian matrix, while the other one is more efficient in both computation and storage, though its finite-sample performance is inferior in practice when the dimension is large. A practitioner may choose the approach suitable to her computational resources and requirement of the inference accuracy.

\end{itemize}
Lastly, we provide a brief literature survey on the recent works for statistical inference for the {\rm}-type SGD algorithms. \cite{chen2020batchmeans} developed a batch-means estimator of the limiting covariance matrix $H^{-1} S H^{-1}$ in \eqref{eqpolyak}, which only uses the stochastic gradient information (i.e., without estimating any Hessian matrices). \cite{Zhu:21online} further extended the batch-means method in \cite{chen2020batchmeans} to a fully online covariance estimator. \label{page:lee2}{\cite{lee2021fast} extended the results in \cite{polyak1992acceleration} to a functional central limit theorem and utilized it to propose a novel online inference procedure that allows for efficient implementation, followed by \cite{lee2022fast,chen2023sgmm} for application to quantile regression and generalized method of moments.} \cite{fang2017scalable} presented a perturbation-based resampling procedure for inference. \cite{su2018uncertainty} proposed a tree-structured inference scheme, which splits the SGD into several threads to construct confidence intervals. \cite{liang2019statistical} introduced a moment-adjusted method and its corresponding inference procedure. 
\cite{toulis2017asymptotic} considered the implicit SGD, and investigate the statistical inference problem under the variant. \cite{duchi2021asymptotic} studied the stochastic optimization problem with constraints and investigate its optimality properties. \cite{chao2019generalization} proposed a class of generalized regularized dual averaging (RDA) algorithms and make uncertainty quantification possible for online $\ell_1$-penalized problems. \cite{shi2020statistical} developed an online estimation procedure for high-dimensional statistical inference. \cite{chen2020statistical} studied statistical inference of online decision-making problems via SGD in a contextual bandit setting. 

\subsection{Notations and organization of the paper}
We write vectors in boldface letters (e.g., $\tth$ and $\v$) and scalers in lightface letters (e.g., $\eta$). For any positive integer $n$, we use $[n]$ as a shorthand for the discrete set $\{1, 2, \cdots, n\}$. Let $\{\e_k\}_{k=1}^d$ be the standard basis in $\R^d$ with the $k$-th coordinate as $1$ and the other coordinates as $0$. Denote $I_d$ as the identity matrix in $\R^{d \times d}$. Let $\| \cdot \|$ denote the standard Euclidean norm for vectors and the spectral norm for matrices. 
We use $A_{k\ell}$ and $ A_{n,k\ell}$ to denote the $(k, \ell)$-th element of matrices $A,A_n \in \R^{d \times d}$, respectively, for all $k, \ell \in [d]$. Furthermore, we denote by $\diag(\v)$ a matrix in $\R^{d \times d}$ whose main diagonal is the same as the vector $\v$ and off-diagonal elements are zero, for some vector $\v \in \R^d$. With a slight abuse of notation, for a matrix $M \in \R^{d \times d}$, we also let $\diag(M)$ denote a  $\R^{d \times d}$ diagonal matrix with same diagonal elements as matrix $M$. We use the standard Loewner order notation $A \succeq 0$ if a matrix $A$ is positive semi-definite.
We use $\tth^\rm$ and $\tth^\kw$ to denote the iterates generated by the {\tt{(RM)}} scheme and the {\tt{(KW)}} scheme, respectively. We use $\widehat\tth^{\mathtt{(ERM)}}$ for the offline empirical risk minimizer, i.e.,
$\widehat\tth^{\mathtt{(ERM)}} = \mathrm{argmin}_{\tth} \frac{1}{n} \sum_{i=1}^n f(\tth; \zet_i)$. As we focus on the {\kw} scheme in this paper, we sometimes omit the superscript $\kw$ in the estimator to make room for the other notations.
In derivations of the {\kw} estimator, we denote the finite difference of $f(\cdot)$ as,
\begin{align}\label{eq:delta}
\Delta_{h, \v}f(\tth; \zet) =f(\tth + h \v; \zet) - f(\tth; \zet),
\end{align}
for some spacing parameter $h \in \R_{+}$ and search vector $\v \in \R^d$. We use $\E_{n}$ to denote the conditional expectation with respect to the natural filtration, i.e.,
\begin{align*}
\E_n[\tth_{n+1}] := \E [\tth_{n+1} | \mathcal{F}_n], \; \; \mathcal{F}_{n}:=\sigma \{\tth_k, \zet_k| k\leq n\}.
\end{align*}
We use the $\O(\cdot)$ notation to hide universal constants independent of the sample size $n$.

The remainder of the paper is organized as follows. In Section \ref{sec:KW}, we describe the Kiefer-Wolfowitz algorithm with random search directions along with three illustrative examples of the classical regression problems. We also provide a technical lemma to characterize the limiting behavior of the {\kw} gradient, which leads to the distributional constraint of the random direction vector. In Section \ref{sec:clt}, we first introduce the technical assumptions before we present the finite-sample rate of convergence of the {\kw} estimator. We further provide the asymptotic distribution of the {\akw} estimator, accompanied by discussions on the statistical (in)efficiency. We highlight a comparison of the choices of the direction distributions in Section \ref{sec:choice_dir}, and further extend the theoretical analysis to multi-query settings of the {\kw} algorithm in Section \ref{sec:multiple}. Section \ref{sec:nonsmooth} generalizes our analysis to some specific nonsmooth loss functions, such as the quantile regression. Based on the established asymptotic distribution results, we propose two types of online statistical inference procedures in Section \ref{sec:infer}. A functional extension of the distributional analysis of {\kw} as a stochastic process is also provided.  
Numerical experiments in Section \ref{sec:exp} lend empirical support to our theory. All proofs are relegated to the Appendix.

\section{Kiefer-Wolfowitz Algorithm}
\label{sec:KW}
In this section, we introduce the general form of the Kiefer-Wolfowitz {\kw} gradient estimator and the corresponding iterative algorithm 
$\tth_{n}=\tth_{n-1}-\eta_n \widehat{g}(\tth_{n-1};\zet_n)$.
In the seminal work  by \cite{Blum1954}, the {\kw} gradient estimator $\widehat{g}(\tth_{n-1};\zet_n)$ is constructed by approximating the stochastic gradient $g(\tth_{n-1};\zet_n)$ using the  canonical basis of $\mathbb{R}^d$, $\{\e_1,\e_2,\dots,\e_d\}$, as search directions.
In particular, given any $\tth\in\R^d$ and $\zet\sim \mathcal{P}_{\zet}$, the $k$-th coordinate of the {\kw} gradient estimator
\begin{align}\label{eq:hatge}
\big(\widehat{g}_{h,\e}(\tth;\zet)\big)_k=\frac{f(\tth+h\e_k;\zet)-f(\tth;\zet)}{h}, \qquad \text{for} \quad k=1,2,\ldots, d,
\end{align}
where $h$ is a spacing parameter for approximation. At each iteration, \eqref{eq:hatge} queries $d+1$ function values from $d$ fixed directions $\{\e_k\}_{k=1}^d$. To reduce the query complexity, a random difference becomes a natural choice. \cite{koronacki1975random} introduced a random version of the {\kw} algorithm using a sequence of random unit vectors that are independent and uniformly distributed on the unit sphere or unit cube. \cite{spall1992multivariate} also provided a random direction version of the {\kw} algorithm, named as the simultaneous perturbation stochastic approximation (SPSA) algorithm and later extended to several variants \cite{chen1999kiefer,spall2000adaptive,he2003convergence}. These random direction methods can reduce the bias in gradient estimates as compared to their non-random counterparts. In the following, we write the {\kw} algorithm with general random search directions, as in \eqref{eq:v},
\begin{align}
\notag \tth_n &=\tth_{n-1}-\eta_n\widehat{g}_{h_n,\v_n}(\tth_{n-1};\zet_n),\\
&\text{where } \widehat{g}_{h, \v}(\tth;\zet) := \frac{1}{h} \Delta_{h, \v}f(\tth; \zet)  \v =\frac{f(\tth+h\v;\zet)-f(\tth;\zet)}{h} \v.\label{eq:hatgv}
\end{align}
Here $\{\v_n\}$ is sampled from an underlying distribution $\Pv$ satisfying certain conditions (see Assumption \ref{assumption4} in Section \ref{sec:clt}). {\label{page:samplingvn}At each iteration $n$, the algorithm samples a direction vector $\v_n$ independently from $P_{\v}$,} and makes two solitary function-value queries, $f(\tth_{n-1}; \zet_n)$ and $f(\tth_{n-1}+h_n\v_n;\zet_n)$.
{\label{page:newpara_m}
We refer to the {\kw} gradient estimator $\widehat{g}_{h_n, \v_n}(\tth_{n-1},\zet_n)$ in \eqref{eq:hatgv} as a \emph{two-query} finite-difference approximation of the stochastic gradient. 
If one is allowed to make additional function-value queries, an averaging of the function values from multiple directions generates a \emph{multi-query} stochastic gradient estimator with reduced variance. In particular, at each iteration $n$, the practitioner makes $m+1$ queries $\{f(\tth_{n-1};\zet_n),f(\tth_{n-1}+h_n\v_n^{(j)};\zet_n)\}_{1\leq j\leq m}$ via $m$ random directions $\big\{\v_n^{(j)}\big\}$ sampled from $\Pv$. If $\Pv$ is a finite distribution, practitioners may choose to sample \emph{with} or \emph{without replacement}.
In summary, an ($m+1$)\emph{-query} {\kw} algorithm constructs a stochastic gradient estimator
\begin{equation}\label{eq:Avg-RandGradEst}
\overline{g}_{n}^{(m)}(\tth_{n-1}; \zet_n)=\frac1m\sum_{j=1}^m \widehat{g}_{h_n,\v_n^{(j)}}(\tth_{n-1}; \zet_n)= \frac{1}{m h_n}\sum_{j=1}^m \Delta_{h_n,\v_n^{(j)}}f(\tth_{n-1}; \zet_n) \v_n^{(j)},
\end{equation}
at each iteration $n$, and updates $\tth_{n}=\tth_{n-1}-\eta_n \overline{g}_{n}^{(m)}(\tth_{n-1}; \zet_n)$. Here we restrict the procedure to sampling from the same distribution $\Pv$ independently across different iterations.
We use $\tth_n^{(m)}$ to denote the final {\kw} estimator using the above $(m+1)$-query finite-difference approximation.

We now provide some illustrative examples of the two-query {\kw} estimator $\widehat{g}_{h_n, \v_n}$ in \eqref{eq:hatgv} used in popular statistical models, and we will refer to these examples throughout the paper. A multi-query extension of the examples can be constructed accordingly.
}

\begin{example}[Linear Regression]\label{ex1}
Consider a linear regression model
$y_i = \x_i^\top \tth^\star + \epsilon_i$ where $\{\zet_i=(\x_i,y_i),~i=1,2,\dots, n\}$ is an {i.i.d.} sample of $\zet=(\x,y)$ and the noise  $\epsilon_i \sim \N(0,\sigma^2)$. We use a quadratic loss function $f(\tth;\zet)=(y - \x^\top \tth)^2$. Therefore, the  stochastic gradient $\nabla f(\tth;\zet)=\left(\x^\top\tth-y\right)\x$, 
and the {\kw} gradient estimator $\widehat{g}_{h, \v}(\tth; \{\x,y\})$ in \eqref{eq:hatgv} becomes
\begin{align*}
\widehat{g}_{h, \v}(\tth; \{\x,y\})&=\frac{1}{h} \left[\big(y-\x^\top(\tth+h\v)\big)^2-\big(y-\x^\top\tth\big)^2\right] \v= 2\v\v^\top \big(\x^\top\tth-y\big)\x+h(\x^\top\v)^2\v.
\end{align*}
\end{example}

\begin{example}[Logistic Regression]\label{ex2}
Consider a logistic regression model with a binary response $y_i \in \{-1,1\}$ generated by 
$ \Pr(y_i|\x_i) =\left(1+\exp\left(-y_i \x_i^\top\tth^\star\right)\right)^{-1}$.
The individual loss function $f(\tth;\zet)=\log\left(1+\exp(-y\x^\top \tth)\right)$. The stochastic gradient $\nabla f(\tth;\zet)=-y \x\left(1+\exp(y\x^\top \tth)\right)^{-1}$, and the {\kw} gradient estimator $\widehat{g}_{h, \v}(\tth; \{\x,y\})$ in \eqref{eq:hatgv} becomes
\begin{align*}
\widehat{g}_{h, \v}(\tth; \{\x,y\})&=\frac{\v}{h} \left[\log\big(1+\exp(-y\x^\top (\tth+h\v))\big)-\log\big(1+\exp(-y\x^\top \tth)\big)\right]\\
&=\frac{-y\v\v^\top\x}{1+\exp(y\x^\top \tth)}+\frac{y^2(\x^\top\v)^2  \exp( y\x^\top\tth )h\v}{2(1+\exp(y\x^\top \tth ) )^2}+\O(h^2),\; \; \text{as }h\rightarrow 0_{+},
\end{align*}
under some regularity conditions on $\tth$ and the distribution of $\x$.
\end{example}

We note that for the {\rm} scheme with differentiable loss functions,  the stochastic gradient is an unbiased estimator of the population gradient under very mild assumption, i.e., $\E_{\zet} g(\tth; \zet) = \nabla F(\tth)$. In contrast, the {\kw} gradient estimator is no longer an unbiased estimator of $\nabla F(\tth)$. In the following lemma, we precisely quantifies the bias incurred by the {\kw} gradient estimator.

\begin{lemma}\label{lem:bv}
We assume that the population loss function $F(\cdot)$ is twice continuously differentiable and $L_f$-smooth, i.e., $\nabla^2 F(\tth) \preceq L_f I_d$ for any  $\tth \in \R^d$. Given any fixed parameter $\tth\in\mathbb{R}^d$, suppose the random direction vector $\v$ is independent from $\zet$, we have 
\begin{align*}
\left\| \E \; \widehat{g}_{h, \v}(\tth; \zet) - \nabla F(\tth) \right\| &\leq \left\| \E \big(\v \v^\top -I_d\big) \nabla F(\tth)\right\| + \frac{h}{2} L_f \E \| \v \|^3\notag,
\end{align*}
where the expectation in $\E \; \widehat{g}_{h, \v}(\tth; \zet)$ takes over both the randomness in $\v$ and $\zet$.
\end{lemma}
The proof of Lemma~\ref{lem:bv} is provided in Section~\ref{sec:app-a} of the Appendix. 
To reduce the bias in the {\kw} gradient, Lemma~\ref{lem:bv} indicates that one should choose the random direction $\v_n$ that satisfies the distributional constraint $\E[\v_n\v_n^\top]=I_d$ (see Assumption \ref{assumption4} in Section \ref{sec:clt}). We will further conduct a comprehensive analysis in Section \ref{sec:choice_dir} on different choices of distributions $\Pv$ satisfying the condition $\E[\v_n\v_n^\top]=I_d$. Despite the existence of the bias, as  the spacing parameter $h_n \rightarrow 0$, the bias  convergences to zero asymptotically.

\section{Theoretical Results}
\label{sec:clt}
We first introduce some regularity assumptions on the population loss $F(\tth)$ and the individual loss $f(\tth; \zet)$. 

{\red \begin{assumption}
\label{assumption1p}
The population loss function $F(\tth)$ is twice continuously differentiable, convex, and $L_f$-smooth, i.e., $0 \preceq \nabla^2 F(\tth) \preceq L_f I_d$ for any  $\tth \in \R^d$. In addition, there exist $\delta_1,\lambda > 0$, such that, $\nabla^2 F(\tth) \succeq \lambda I_d$ for any $\tth$ in the $\delta_1$-ball centered at $\tth^\star$.
\end{assumption}}
\begin{assumption}\label{assumption2}
Assume $\E \left[\nabla f(\tth;\zet_n)\right]=\nabla F(\tth)$ for any $\tth \in \R^d$. 
Moreover, for some $0 < \delta \leq 2$, there exists $M>0$ 
such that
$
\E \|\nabla f(\tth;\zet_n) - \nabla F(\tth)\|^{2+\delta} \leq M\big(\|\tth - \tth^\star\|^{2+\delta}+ 1\big).
$
\end{assumption}
\begin{assumption} \label{assumption3}
There are constants $L_h, L_p>0$ such that for any $\tth, \y \in \R^d$,
\begin{align*}
\E \left\| \nabla^2f(\tth; \zet_n) - \nabla^2 f(\y; \zet_n)\right\|^2 \leq L_h \|\tth-\y\|^{2},\qquad
\E\left\| [\nabla^2 f(\tth^\star; \zet_n)]^2  - H^2 \right\| \leq L_p,
\end{align*}
where $H$ is the Hessian matrix of the population loss function $F(\cdot)$, i.e.,
$H=\nabla^2 F(\tth^\star)$.
\end{assumption}
\begin{assumption} \label{assumption4}
We adopt {i.i.d.} random direction vectors $\{\v_n\}$ from some common distribution $\v \sim \Pv$ such that $\E[\v\v^\top]=I_d$. Moreover, assume that the $(6+3\delta)$-th moment of $\v$ is bounded.
\end{assumption}

We discuss the above assumptions and compare them with the standard conditions in the literature of {\rm}-type SGD inference. 
{\red Assumption~\ref{assumption1p} assumes the population loss function $F(\cdot)$ to be $\lambda$-strongly convex in a $\delta_1$ neighborhood of the true parameter $\tth^\star$, which is often referred to as \emph{local strong convexity} assumption widely used in many existing literature of statistical inference on {\rm}-type stochastic optimization (e.g., \citealp{polyak1992acceleration}). This condition is satisfied in the settings of linear and logistic regression (Examples \ref{ex1}--\ref{ex2}) with classical design assumptions on the covariates $\x$.}  Assumption~\ref{assumption2} introduces the unbiasedness condition on the stochastic gradient $\nabla f(\tth; \zet)$ when the individual loss function $f(\tth; \zet)$ is smooth. 
The $(2+\delta)$-th moment condition is the classical Lyapunov condition used in the derivation of asymptotic normality. 
The statements in Assumption~\ref{assumption3} introduce the Lipschitz continuity condition and the concentration condition on the Hessian matrix. Relaxation to Assumptions~\ref{assumption2} and \ref{assumption3} can be made to handle some nonsmooth loss functions $f(\tth; \zet)$, such as the quantile regression as described in Section \ref{sec:nonsmooth} below. 
Assumption~\ref{assumption4} guarantees that the {\kw} gradient $\widehat{g}_{h, \v}(\tth; \zet)$ is an asymptotically unbiased estimator of $\nabla F(\tth)$ as the spacing parameter $h_n\rightarrow0$, as suggested by Lemma~\ref{lem:bv}.  
We provide several examples of $\Pv$ in Section \ref{sec:choice_dir}.

Before we derive the asymptotic distribution for {\akw}, we first provide a consistency result and finite sample error bound for the final {\kw} iterate $\tth_n$: 
\begin{restatable}{proposition}{consist}
	\label{lemma:x}
	 Assume Assumptions \ref{assumption1p}, \ref{assumption2}, and \ref{assumption4} hold. Set the step size as $\eta_n = \eta_0 n^{-\alpha}$ for some constant $\eta_0>0$ and $\alpha\in \left(\frac{1}{2}, 1\right)$ and the spacing parameter as $h_n = h_0 n^{-\gamma}$ for  constant $h_0 > 0$, and $\gamma \in \left(\frac{1}{2}, 1\right)$. The {\kw} iterate $\tth_n$ converges to $\tth^\star$ almost surely. 
	
Additionally, assume Assumptions~\ref{assumption1p} holds with $\delta_1=+\infty$, i.e., the population loss function $F(\tth)$ is globally $\lambda$-strongly convex and $L_f$-smooth. For sufficiently large $n$, we have, $\E\| \tth_n - \tth^\star\|^{2+\delta} \leq C n^{-\alpha(2+\delta)/2}$,	 where 
  the constant $C$ depends on $d, \lambda,  L_f,\alpha, \gamma,$ $\eta_0, h_0$.
\end{restatable}
{The proof of Proposition \ref{lemma:x} and the explicit dependency of the constant $C$ on the parameters and the initial value $\tth_0$ are provided in Remark \ref{rmk:bach} of the Appendix.} A similar error bound on the parameter $\tth$ is given by \cite{duchi2015optimal} in terms of the function values for $\delta=0$. We provide an error bound for the $(2+\delta)$-moment under our assumption, where $\delta\in(0,2]$ is assumed in Assumption \ref{assumption2}. Proposition \ref{lemma:x} suggests that the asymptotic rate of the {\kw} estimator matches the best convergence rate of the {\rm} estimator \citep{bach2011non} when the spacing parameter $h_n = h_0 n^{-\gamma}$ is a decreasing sequence with $\gamma\in(\frac12,1)$.  

Recall that to characterize the asymptotic behavior of {\rm} iterates, we  denote by $S$, the Gram matrix of $\nabla f(\tth; \zet)$ at the true parameter $\tth^\star$, i.e., $S := \E \left[\nabla f(\tth^\star; \zet) \nabla f(\tth^\star; \zet)^\top \right]$. Analogously, we define the limiting Gram matrix of the {\kw} gradient estimator $\widehat{g}_{h,\v}$ at $\tth^\star$ as $h\rightarrow 0$ to be $Q$. The following lemma proves that the limiting Gram matrix takes the form of $Q=\E \left[\v \v^\top S \v \v^\top\right]$, and it quantifies the distance between $\widehat{g}_{h,\v}(\tth^\star; \zet) \widehat{g}_{h,\v}(\tth^\star; \zet)^\top$ and $Q$, as the spacing parameter $h\rightarrow 0$.

\begin{lemma}\label{lem:variance}
Under Assumptions \ref{assumption1p}, \ref{assumption2}, \ref{assumption3}, and \ref{assumption4}, we have
\begin{align*}
\left\|\E \big[\widehat{g}_{h,\v}(\tth^\star; \zet) \widehat{g}_{h,\v}(\tth^\star; \zet)^\top\big] - Q\right\| \leq C h(1+h^2), \; \;
Q=\E \big[\v \v^\top S \v \v^\top\big].
\end{align*}
where $S = \E \left[\nabla f(\tth^\star; \zet) \nabla f(\tth^\star; \zet)^\top \right]$ is defined in Assumption \ref{assumption2}.
\end{lemma}

With Lemma \ref{lem:variance} in place, we state our first main result that characterizes the limiting distribution of the averaged {\akw} iterates defined in \eqref{eq:akw}.

\begin{restatable}{theorem}{clt}
\label{thm:clt}
Let Assumptions \ref{assumption1p}, \ref{assumption2}, \ref{assumption3}, and \ref{assumption4} hold. Set the step size as $\eta_n = \eta_0 n^{-\alpha}$ for some constant $\eta_0>0$ and $\alpha\in \left(\frac{1}{2}, 1\right)$, and the spacing parameter as $h_n = h_0 n^{-\gamma}$ for some constant $h_0 > 0$, and $\gamma \in \left(\frac{1}{2}, 1\right)$.  The averaged {\kw} estimator $\overline\tth_n$ satisfies,
\begin{align}
\label{eq:clt}
\sqrt{n} \; \left(\overline{\tth}_n - \tth^\star\right) \Longrightarrow \N\left(\0, H^{-1} Q H^{-1}\right),  \qquad \text{as} \quad n\rightarrow \infty,
\end{align}
where $H=\nabla^2 F(\tth^\star)$ is the population Hessian matrix and $Q = \E \left[\v \v^\top S \v \v^\top\right]$ is defined in Lemma~\ref{lem:variance}. Here $\Longrightarrow$ represents the convergence in distribution.
\end{restatable}

We now compare the asymptotic covariance matrix of $\overline{\tth}_n$ with that of the {\rm} counterpart in \eqref{eqpolyak} \footnote{Note that the asymptotic covariance $H^{-1} S H^{-1}$ in \eqref{eqpolyak} is ``optimal'' in the sense that it matches the asymptotic covariance for the empirical risk minimizer $\widehat\tth^{(\mathtt{ERM})}$ without online computation and gradient information constraint.}.  As one can see, the asymptotic covariance matrix of {\akw} estimator $\overline\tth_n$ exhibits a similar sandwich form as the covariance matrix of {\rm}, but strictly dominates the latter, regardless of the choice of random direction vectors $\{\v_1,\v_2,\dots, \v_n\}$.  In fact, it is easy to check that
\begin{align}\label{eq:dominate}
	H^{-1}QH^{-1} - H^{-1}SH^{-1} = H^{-1}\E_{\v} \left[(\v \v^\top - I_d) S (\v \v^\top - I_d)\right] H^{-1}\succ 0,
\end{align}
which suggests the {\akw} estimator suffers an inevitable loss of efficiency compared to the $\widehat\tth^{\rm}$. 
{ In Section \ref{sec:multiple}, we analyze {\akw} with multiple function-value queries at each iteration. With the price of additional per-iteration computational complexity, one is able to improve the statistical efficiency of {\akw} and achieve the optimal asymptotic variance $H^{-1}SH^{-1}$.}

\begin{remark}
	\label{rem:kw}
	 To complete the distributional analysis on {\kw} iterates, we also provide the asymptotic distribution of the $n$-th iterate $\tth_n^{\kw}$ of  \eqref{eq:kw} without averaging. Assume the Hessian matrix has decomposition $	H = P \Lambda P^\top$, where $P$ is an orthogonal matrix and  $\Lambda$ is a diagonal matrix. Using the proof in \cite{fabian1968asymptotic}, we establish the following asymptotic distribution for $\tth_n^{\kw}$,
	\begin{align}\label{eq:dist_nonavg}
		n^{\alpha/2} (\tth_n^{\kw} - \tth^\star) \Longrightarrow \mathcal{N}(0, \Sigma),
	\end{align}
	where each $(k,\ell)$-th entry of the covariance matrix $\Sigma$ is,
	\begin{align*}
		\Sigma_{k \ell} = \eta_0 \big(P^\top Q P\big)_{kl} \big(\Lambda_{kk}+\Lambda_{\ell \ell}\big)^{-1},\quad 1\leq k,\ell\leq d.
	\end{align*}
	Here $\eta_0>0$ and $\alpha \in (\frac12, 1)$ are specified in the step size $\eta_n = \eta_0 n^{-\alpha}$. As $\alpha<1$, the $n$-th iterate  $\tth_n^{\kw}$ without averaging converges at a slower rate $n^{-\alpha/2}$ than that of {\akw} in Theorem \ref{thm:clt}. 
\end{remark}

\subsection{Examples: choices of direction distribution}
\label{sec:choice_dir}

By Theorem \ref{thm:clt}, the asymptotic covariance matrix of {\akw} estimator, $H^{-1}QH^{-1}$, depends on the distribution of search direction $\Pv$ via $Q=\E[\v\v^\top S\v\v^\top]$. In this section, we 
compare the asymptotic covariance matrices of the {\akw} estimator when the random directions $\{\v_i\}_{i=1}^n$ are sampled from different $\Pv$'s. Several popular choices of $\Pv$ are listed as follows,
\begin{enumerate}
\item[\tg\label{tg}] Gaussian: $\v \sim \mathcal{N}(0, I)$.
\item[\ts\label{ts}] Spherical: $\v$ is sampled from the uniform distribution on the sphere $\|\v\|^2 = d$.
\item[\ti\label{ti}] Uniform in the canonical basis: $\v$ is sampled from $\big\{\sqrt{d} \e_1, \sqrt{d}\e_2, \ldots, \sqrt{d}\e_d\big\}$ with equal probability, where $\left\{\e_1,\e_2,\dots, \e_d\right\}$ is the canonical basis of $\R^d$.
\end{enumerate}

It is easy to verify that the above three classical choices of $\Pv$ satisfy Assumption \ref{assumption4}, {among which {\tg} and {\ts} are  continuous distributions, while {\ti} is a discrete distribution. In particular, {\ti} is a discrete uniform distribution with equal probability among the $d$ vectors of the standard basis of Euclidean space $\R^n$, which can be generalized in the following two forms.}
\begin{enumerate}
\item[\tu\label{tu}] Uniform in an arbitrary orthonormal basis $U$: $\v_i$ is sampled uniformly from $\left\{\sqrt{d}\u_1, \sqrt{d}\u_2, \right.$ $\left. \ldots,\sqrt{d}\u_d\right\}$, where $\left\{\u_1,\u_2,\dots, \u_d\right\}$ is an arbitrary \emph{orthonormal basis} of $\R^d$, i.e., the matrix $U=(\u_1,\u_2,\dots, \u_d)$ is a $d \times d$ orthonormal matrix such that $UU^\top=U^\top U=I$.
\item[\tp\label{tp}] Non-uniform in the canonical basis with probability $(p_1,p_2,\dots,p_d)$: $\v = \sqrt{1/p_k} \, \e_k$ with probability $p_k > 0$, for $k \in [d]$ and $\sum_{k=1}^d p_k = 1$.
\end{enumerate}

The following proposition provides expressions of the matrix $Q$ for the above five choices of $\Pv$. 
\begin{proposition}
\label{cor:clt0}
Under the assumptions in Theorem \ref{thm:clt}, for above examples of $\Pv$,
\begin{enumerate}[(i)]
\item[{\hyperref[tg]{\tg}}] Gaussian: $\Qg = \left(2S+ \tr(S)I_d\right)$.
\item[{\hyperref[ts]{\ts}}] Spherical: $\Qs = \frac{d}{d+2}\left(2S+\tr(S)I_d\right)$.
\item[{\hyperref[ti]{\ti}}] Uniform in the canonical basis: $\Qi = d \, \diag(S)$.
\item[{\hyperref[tu]{\tu}}] Uniform in an arbitrary orthonormal basis $U$: $\Qu=d\,U\diag(U^\top S U) U^\top$.
\item[{\hyperref[tp]{\tp}}] Non-uniform in a natural coordinate basis: $\Qp = \diag(S_{1 1} / p_1, S_{2 2}/p_2, \ldots, S_{d d}/p_d)$.
\end{enumerate}
\end{proposition}

From Proposition \ref{cor:clt0}, one can see that any of the above choices of $\Pv$ leads to a $Q^{(\cdot)}$ that strictly dominates $S$. 
Take $S=I_d$ as an example, we have $\Qg=(d+2)I_d$ and $\Qs=\Qi=\Qu=d I_d$ and $\Qp=\diag(p_1^{-1},p_2^{-1},\dots, p_d^{-1})\succ I_d$ where $p_1+p_2+\dots+p_d=1$. {Several additional findings and implications of Proposition \ref{cor:clt0} are discussed in Section \ref{subsec:suppchoice} of the Appendix. To briefly mention a few, the Gaussian direction {\tg} is always inferior to the spherical direction {\ts}. Among the rest of the choices, there is \emph{no domination} relationship, and different optimality criterion in the experimental design leads to different optimal choices of $\Pv$.
}

\subsection{Multi-query extension and statistical efficiency}
\label{sec:multiple}
{
We now consider the {\akw} estimator using $(m+1)$ function queries $\overline\tth_n^{(m)}$ in \eqref{eq:Avg-RandGradEst},
\[
\overline\tth_n^{(m)}=\frac1n\sum_{i=1}^n\tth_i^{(m)},\qquad\text{where  }\tth_{i}^{(m)}=\tth_{i-1}^{(m)}-\eta_i \overline{g}_{n}^{(m)}(\tth_{i-1}; \zet_i)=\tth_{i-1}^{(m)}-\frac{\eta_i}{m}\sum_{j=1}^m \widehat{g}_{h_i,\v_i^{(j)}}(\tth_{i-1}; \zet_i).
\]
 Here we first consider using the same sampling distribution across $m$ queries and $n$ iterations. In other words, $\v_i^{(j)}$ is sampled \emph{i.i.d.} from $\Pv$ for $i=1,2,\dots, n$ and $j=1,2,\dots, m$.
}
Analogous to Theorem \ref{thm:clt}, we present the asymptotic distribution of  multi-query {\akw},
\begin{theorem}
\label{thm:multiple_1}
Under the assumptions in Theorem \ref{thm:clt}, the $(m+1)$-query {\akw} estimator has the following asymptotic distribution, as $n\rightarrow \infty$,
\begin{align*}
\sqrt{n}\left(\overline{\tth}_n^{(m)} - \tth^\star\right) \Longrightarrow \N\left(\0, H^{-1} Q_m H^{-1}\right),\; \;  \text{where } Q_m = \frac{1}{m}Q + \frac{m-1}{m}S.
\end{align*}
\end{theorem}

Theorem \ref{thm:multiple_1} illustrates a trade-off effect between the statistical efficiency and computational efficiency. When $m=1$ and only two queries of function evaluations are available, Theorem \ref{thm:multiple_1} reduces to Theorem \ref{thm:clt}, and $Q_m=Q$. Conversely, as $m\rightarrow \infty$, we have $Q_m\rightarrow S$. Therefore, the asymptotic covariance of $(m+1)$-query {\akw} estimator $\overline{\tth}_n^{(m)}$ approaches the optimal covariance $H^{-1}SH^{-1}$ 
as $m$ approaches infinite. Nevertheless, the algorithm requires $m$ function-value queries at each iteration, which significantly increases the computation complexity. 

For a finite $m$, a slight revision of the sampling scheme of the direction vectors $\{\v_i^{(j)}\}_{j=1,2,\dots, m}$ provides a remedy to achieve a smaller and indeed optimal asymptotic covariance matrix. {Particularly at the $i$-th iteration, one may sample $m$ direction vectors $\{\v_i^{(j)}\}_{j=1,2,\dots, m}$ from a discrete distribution (such as \hyperref[ti]{\ti} and \hyperref[tu]{\tu}) \emph{without replacement}. In such settings, the direction vectors $\big\{\v_i^{(1)}, \v_i^{(2)}, \dots, \v_i^{(m)}\big\}$ are no longer independent but they have the same marginal distribution.} 
The asymptotic distribution of the multi-query {\kw} algorithm sampling without replacement is provided in the following theorem. 

\begin{theorem}
\label{thm:multiple_2}
Under the assumptions in Theorem \ref{thm:clt}, and the direction vectors in all iterations $\big\{\widetilde V_{i}\big\}_{i=1}^n$ are \emph{i.i.d.} from $\Pv$ such that $\widetilde V_{i}=\big(\v_{i}^{(1)},\v_{i}^{(2)},\dots,\v_{i}^{(m)}\big)$ follows discrete sampling scheme in \hyperref[ti]{\ti} and \hyperref[tu]{\tu} WithOut Replacement {\tt(WOR)}, the $(m+1)$-query {\akw} estimator, referred to as $\overline{\tth}_n^{(m,\tnr)}$, has the following asymptotic distribution, as $n\rightarrow\infty$,
\begin{align*}
\sqrt{n}\left(\overline{\tth}_n^{(m,\tnr)}- \tth^\star\right) \Longrightarrow \N\left(\0, H^{-1} Q_m^{(\tnr)}H^{-1}\right),\; \; \text{ where } Q_m^{(\tnr)} = \frac{(d-m)}{m(d-1)}Q + \frac{d(m-1)}{m(d-1)}S.
\end{align*}
\end{theorem}

By comparing the asymptotic covariance matrices in Theorems~\ref{thm:multiple_1} and~\ref{thm:multiple_2},  $Q_m^{(\tnr)}$ for sampling without replacement  case is strictly smaller than $Q_m$  in Theorems~\ref{thm:multiple_1} when we consider multi-query evaluation ($m \geq 2$). Moreover, when $m=d$, it is easy to see that $Q_m^{(\tnr)}=S$. Therefore, the $(d+1)$-query {\akw} estimator $\overline{\tth}_n^{(m,\tnr)}$ achieves the same limiting covariance as that of the averaged {\rm} estimator. 
Under a well-specified parametric model, the limiting covariance matrix $H^{-1}SH^{-1}$ achieves the Cram\'{e}r-Rao lower bound.  This result indicates that the $(d+1)$-query $\overline{\tth}_n^{(d,\tnr)}$ is asymptotically efficient \citep{van2000asymptotic}.

{\red \subsection{Asymptotic behavior of {\akw} estimator for nonsmooth losses}\label{sec:nonsmooth}

The analysis of the asymptotic distribution of the {\akw} estimator remains valid naturally for some nonsmooth loss functions $F(\tth)$ including quantile regression, specifically,
\begin{example}[Quantile Regression]\label{ex3}
Consider a quantile regression model $y_i = \x_i^\top \tth^\star + \varepsilon_i$ where $\{\zet_i=(\x_i,y_i),~i=1,2,\dots, n\}$ is an \emph{i.i.d.} sample of $\zet=(\x,y)$ and the noise $\varepsilon_i$ is independent from $\x_i$ and $\Pr(\varepsilon_i\leq 0)=\tau$. The individual loss $f(\tth;\zeta)=\rho_{\tau}(y-\x^\top \tth)$, where $\rho_{\tau}(z) = z(\tau - 1_{\{z<0\}})$. Although $\rho_\tau$ is non-differentiable, the {\kw} gradient estimator $\widehat{g}_{h, \v}$ is well-defined and takes the following form,
\begin{align*}
\widehat{g}_{h, \v}(\tth; \{\x,y\})=&\frac{\v}{h}\left[\rho_{\tau}\big(y-\x^\top (\tth+h\v)\big)-\rho_{\tau}\big(y-\x^\top \tth\big)\right]\\
=&\v\v^\top \x\big(\tau-1_{\{y-\x^\top\tth<0\}}\big),\quad \text{for }0<h<\left|\frac{y-\x^\top \tth}{\x^\top \v}\right|.
\end{align*}
\end{example}
\noindent We next state modeling assumptions for some nonsmooth losses including Example \ref{ex3}. 
\begin{assumption}\label{assumption6}
Assume that $f(\tth;\zeta)=\rho(y-\x^\top\tth)$ where $\rho(u)$ is a convex function with a subgradient $\psi(u)$, and $|\psi(u)| \leq C(|u| + 1)$ for some constant $C>0$. The covariates $\x$ is independent of $\epsilon$ and $\x$ has finite eighth moments and nondegenerate covariance matrices. Assume the probability density function $p(x)$ of $\varepsilon$ is in $C^3$, its derivatives up to third order are all integrable, and $\varepsilon$ has finite fourth moment. Assume $\phi(u) = \E[\psi(u+\varepsilon)]$ is differentiable, $\phi(0) = 0,$ and $u\phi(u)>0$ for any $u \neq 0$. We further assume  $\phi'(0) >0$, and there exist constants $C > 0$ such that $|\phi'(u)|\leq C$ and $|\phi'(u) - \phi'(v)| \leq C|u-v|$.
\end{assumption}
Assumption \ref{assumption6} essentially guarantees that the population loss function $F(\tth)$ is smooth and locally strongly convex, and the distribution of the noise $\epsilon$ is smooth enough such that, the empirical loss (averaged individual loss) well approximates the population loss asymptotically. We now restate Theorem \ref{thm:clt} below for certain nonsmooth losses under Assumption \ref{assumption6}. The the proof Theorem \ref{thm:clt-qt} is relegated to Section \ref{subsec:quantile} of the Appendix. 

\begin{theorem}
\label{thm:clt-qt}
Let Assumptions  \ref{assumption4} and \ref{assumption6} hold. Under the step size and spacing parameter conditions specified in Theorem \ref{thm:clt}, the averaged estimator $\overline\tth_n$ satisfies,
\begin{align}
\label{eq:clt-qt}
&\sqrt{n} \; \left(\overline{\tth}_n - \tth^\star\right) \Longrightarrow \N\left(\0, H^{-1} Q H^{-1}\right),  \qquad \text{as} \quad  n\rightarrow \infty,
\end{align}
where $Q = \E[ \v\v^\top S \v\v^\top]$,  $S = \E[\psi^2(\varepsilon) \x\x^\top]$, and $H = \E[\phi'(0) \x\x^\top]$.
\end{theorem}

From Theorem~\ref{thm:clt-qt}, we know that the {\akw} estimator of the above quantile regression model is asymptotically normal with asymptotic covariance matrix $H^{-1} Q H^{-1}$ where $Q$ depends on the sampling directions (see Proposition~\ref{cor:clt0}). In an quantile regression example \ref{ex3} when the noise $\epsilon$ follows the normal distribution with standard deviation $1$ and $\mathrm{Pr}(\epsilon\leq 0)=\tau$, a direct computation shows $H=\varphi\big(\Phi^{-1}(\tau)\big)\mathrm{E}[\x\x^\top]$, where $\varphi$ and $\Phi$ are the probability and cumulative distribution functions of a standard normal distribution. Furthermore, if $\v$ is sampled uniformly from the canonical basis with two function queries, we can see from Proposition \ref{cor:clt0} that $Q=Q^{\ti}=d\diag(S)$ where  $S=\tau(1-\tau)\mathbb{E}[\x\x^\top]$.

}

\section{Online Statistical Inference}
\label{sec:infer}

In the previous section, we provide the asymptotic distribution for the  {\akw} estimator. For the purpose of conducting statistical inference of $\tth^\star$, we need a consistent estimator of the limiting covariance $H^{-1}QH^{-1}$ in \eqref{eq:clt}.  A direct way is to construct a pair of consistent estimators $\hat{H}$ and $\hat Q$ of $H$ and $Q$, respectively, and estimate the asymptotic covariance by the \emph{plug-in} estimator $\hat{H}^{-1} \hat Q \hat{H}^{-1}$. Offline construction of those estimators is generally straightforward. However, as the {\kw} scheme typically applies to sequential data, it is ideal to estimate the asymptotic covariance in an online fashion without storing the data. Therefore, one cannot simply replace the true parameter $\tth^\star$ by its estimate $\overline \tth_n$ in $Q$ and $H$ in an online setting, since we can no longer access  the data stream $\{\zet_i\}_{i=1}^n$ after the estimator $\overline\tth_n$ is obtained. To address this challenge, we first propose the following finite-difference Hessian estimator at each iteration $n$: 
\begin{align}\label{eq:emp_hess}
\widetilde{G}_n = \sum_{k=1}^d \sum_{\ell=1}^d \widetilde{G}_{n,kl}\e_k \e_\ell^\top= \frac{1}{h_n^2} \sum_{k=1}^d \sum_{\ell=1}^d \left[\Delta_{h_n,\e_k}f(\tth_{n-1} + h_n \e_\ell; \zet_n) - \Delta_{h_n, \e_k} f(\tth_{n-1}; \zet_n)\right]\e_k \e_\ell^\top,
\end{align}
This construction can be viewed  as a multi-query (with $d^2+1$ queries of function values at each iteration) {\kw} scheme with  the \hyperref[ti]{\ti} choice of the random directions. Other choices of the search directions can be used as well, and discussions are provided in Section \ref{sec:G_con} of the Appendix. Each additional function-value query beyond the first one provides an estimate $\widetilde{G}_{n,kl}$ for the $(k,l)$-th entry of the matrix $\widetilde{G}_n$. To reduce the computational cost in $\widetilde{G}_n$, at each iteration, the algorithm may compute a random subset of entries of $\widetilde{G}_n$ and partially inhere the remaining entries from the previous estimator $\widetilde{G}_{n-1}$. For example, each entry $\widetilde{G}_{n,k\ell}$ is updated with probability $p\in(0,1]$.
The procedure thus requires $\O(p d^2)$ function-value queries at each step. If we set $p = \O(1/d^2)$, then the query complexity is reduced to $\O(1)$ per step. Since the construction of \eqref{eq:emp_hess} does not guarantee symmetry, an additional symmetrization step needs to be conducted, as
\begin{align}
\label{eq:hessian-est2}
\widetilde{H}_n = \frac{1}{n} \sum_{i=1}^n \frac{\widetilde{G}_i + \widetilde{G}_i^\top}{2}.
\end{align}

The next lemma quantifies the estimation error of the Hessian estimator $\widetilde{H}_n$ in \eqref{eq:hessian-est2} and the proof is provided in Section~\ref{sec:app-c} of the Appendix.
\begin{lemma}
\label{lemma:hessian}
 Assume Assumptions \ref{assumption1p}, \ref{assumption2}, \ref{assumption3}, \ref{assumption4} hold, or Assumptions \ref{assumption4}, \ref{assumption6} hold, then $\widetilde{H}_n$ converges in probability to $H$. Additionally,  if Assumptions \ref{assumption1p}, \ref{assumption2}, \ref{assumption3}, \ref{assumption4} hold with $\delta_1=+\infty$ and $\delta=2$, we have
$\E \| \widetilde{H}_n - H \|^2 \leq C_1 n^{-\alpha} + C_2 p^{-1}n^{-1}.
$\end{lemma}

From Lemma \ref{lemma:hessian}, as $n \rightarrow \infty$, the error rate is dominated by the $C_1 n^{-\alpha}$ term, where $\alpha$ is the  parameter of the decaying step sizes. 
\begin{remark}
In construction of the estimator of the limiting covariance matrix $H^{-1}QH^{-1}$, it is necessary to avoid the possible singularity of $\widetilde H_n$. A common practice is to adopt a thresholding version of $\widetilde{H}_n$ in \eqref{eq:hessian-est2}. Let $U \widetilde\Lambda_n U^\top$ be the eigenvalue decomposition of $\widetilde{H}_n$, and define
\begin{align}
\label{eq:hessian-hat}
\widehat{H}_n = U \widehat{\Lambda}_n U^\top, \; \; \widehat{\Lambda}_{n,kk} = \max \left\{\kappa_1, \widetilde\Lambda_{n,kk}\right\}, \; k =1,2,\dots, d,
\end{align}
for any positive constant $\kappa_1 < \lambda$ where $\lambda$ is defined in Assumption~\ref{assumption1p}.  It is guaranteed by construction that $\widehat{H}_n$ is strictly positive definite and thus invertible.
\end{remark}

On the other hand, the estimator of Gram matrix $Q$ can be naturally constructed as 
\begin{align} \label{eq:Q-hat}
\widehat{Q}_n := \frac{1}{n} \sum_{i=1}^n \widehat{g}_{h_i, \v_i}(\tth_{i-1}; \zet_i) \, \widehat{g}_{h_i, \v_i}(\tth_{i-1}; \zet_i)^\top,
\end{align}
where $ \widehat{g}_{h_i, \v_i}(\tth_{i-1}; \zet_i)$ is the {\kw} update in the $i$-th iteration obtained by \eqref{eq:hatgv}. As both $\widehat{H}_n$ in \eqref{eq:hessian-hat} and $\widehat{Q}_n$ in \eqref{eq:Q-hat} can be constructed sequentially without storing historical data\footnote{The sequence $\widehat{Q}_n := \frac{1}{n} \sum_{i=1}^n Q_i$ with $Q_i = \widehat{g}_{h_i, \v_i}(\tth_{i-1}; \zet_i) \, \widehat{g}_{h_i, \v_i}(\tth_{i-1}; \zet_i)^\top$ can be constructed in one pass over the sequential data. In particular, we could compute $\widehat{Q}_n = \frac{1}{n} ((n-1) \widehat{Q}_{n-1}+Q_i)$ sequentially.}, the final plug-in estimator  $\widehat{H}_n^{-1} \widehat{Q}_n \widehat{H}_n^{-1}$ can also be constructed in an online fashion. Based on Lemma~\ref{lemma:hessian}, we obtain the following consistency result of the covariance matrix estimator $\widehat{H}_n^{-1} \widehat{Q}_n \widehat{H}_n^{-1}$.

\begin{theorem}
\label{thm:plugin}
Assume Assumptions \ref{assumption1p}, \ref{assumption2}, \ref{assumption3}, \ref{assumption4} hold, or Assumptions \ref{assumption4}, \ref{assumption6} hold. Under the step size and spacing parameter conditions specified in Theorem \ref{thm:clt} or Theorem \ref{thm:clt-qt}, we have $\widehat{H}_n^{-1} \widehat{Q}_n \widehat{H}_n^{-1}$ converges in probability to $H^{-1} Q H^{-1}$.

Furthermore, if Assumptions \ref{assumption1p}, \ref{assumption2}, \ref{assumption3}, \ref{assumption4} hold with $\delta_1=+\infty$ and $\delta=2$, we have $\E \left\|  \widehat{H}_n^{-1} \widehat{Q}_n \widehat{H}_n^{-1} - H^{-1} Q H^{-1} \right \| \leq C n^{-\alpha/2}$.
\end{theorem}

We defer the technical proof to Section~\ref{sec:app-c} of the Appendix. Theorem \ref{thm:plugin} establishes the consistency and the rate of the convergence of our proposed covariance matrix estimator $\widehat{H}_n^{-1} \widehat{Q}_n \widehat{H}_n^{-1}$. Given Theorems \ref{thm:clt} and \ref{thm:plugin}, a confidence interval of the projected true parameter $\w^\top \tth^\star$ for any $\w\in\R^d$ can be constructed via a projection of $\overline{\tth}_n$ and $\widehat{H}_n^{-1} \widehat{Q}_n \widehat{H}_n^{-1}$ onto $\w$. Specifically, for a pre-specified confidence level $q$ and the corresponding $z$-score $z_{q/2}$, we obtain an asymptotic exact confidence interval as $n \rightarrow \infty$,
\begin{align*}
\P \left\{\w^\top \tth^\star \in \left[ \w^\top  \overline{\tth}_n -  \frac{z_{q/2}}{\sqrt{n}} \sqrt{\w^\top \widehat{H}_n^{-1} \widehat{Q}_n \widehat{H}_n^{-1} \w},  \;\;\; \w^\top  \overline{\tth}_n +\frac{z_{q/2}}{\sqrt{n}} \sqrt{\w^\top \widehat{H}_n^{-1} \widehat{Q}_n \widehat{H}_n^{-1} \w}\right] \right\} \rightarrow 1 - q.
\end{align*}

\subsection{Online inference without additional function-value queries}

Despite the simplicity of the plug-in approach, the proposed estimator $\widehat{H}_n^{-1} \widehat{Q}_n \widehat{H}_n^{-1}$ incurs additional computational and storage cost as it requires additional function-value queries for constructing $\widehat{H}_n$. It raises a natural question: \emph{is it possible to conduct inference only based on {\kw} iterates $\{\tth_i\}_{i=1,2,\dots}$ without additional function-value queries}?

In this section, we provide an affirmative answer to this question, and propose an alternative online statistical inference procedure using the intermediate {\kw} iterates only, without requiring any additional function-value query. Intuitively, the {\akw} estimator in \eqref{eq:akw} is constructed as the average of all intermediate  {\kw} iterates $\{\tth_i\}_{i=1}^n$. If all iterates were independent and identically distributed, the asymptotic covariance
could have been directly estimated by the sample covariance matrix of the iterates $\frac1n\sum_{i=1}^n(\tth_i-\overline\tth)(\tth_i-\overline\tth)^\top$. Unfortunately, the {\kw} iterates are far from independent and indeed highly correlated. Nevertheless, the autocorrelation structure of the iterates can be carefully analyzed and utilized to construct the estimator of $H^{-1}QH^{-1}$.

In this paper, we adopt an alternative approach to take more advantage of the autocorrelation structure  by leveraging the techniques from robust testing literature \citep{abadir1997two,kiefer2000simple,lee2021fast}. Such an estimator is often referred to as the Fixed Bandwidth Heteroskedasticity and Autocorrelation Robust estimator (\emph{fixed-$b$} HAR) in the econometrics literature to overcome the series correlation and heteroskedasticity in the error terms for the OLS estimates of the linear regression. For stochastic approximation, \cite{lee2021fast} first utilized and generalized this technique to construct an online statistical inference procedure, and refer to this method as \emph{random scaling}, followed by \cite{lee2022fast,chen2023sgmm} for extension to quantile regression and generalized method of moments.\label{page:lee3}

In particular, we present the following theorem based on a functional extension of the distributional analysis of the intermediate {\kw} iterates $\{ \tth_t\}$ as a stochastic process.

\begin{theorem} \label{thm:fixed-b}
For any $\w\in\mathbb{R}^d$, under the assumptions in Theorem \ref{thm:clt} or Theorem \ref{thm:clt-qt}, we have
\begin{align}\label{eq:fclt}
\sqrt{n}\frac{\w^\top (\overline{\tth}_n- \tth^\star)}{\sqrt{\w^\top V_n \w}}\Longrightarrow \frac{W_1}{\sqrt{\int_0^1 \left(W_r-r W_1\right)^2 \intd r}},
\end{align}
where $V_n = \frac{1}{n^2}\sum_{i= 1}^n i^2(\overline{\tth}_i - \overline{\tth}_n)(\overline{\tth}_i - \overline{\tth}_n)^{\top}$,  and $\overline{\tth}_i=\frac{1}{i}\sum_{\ell=1}^i\tth_\ell$ is the average of  iterates up to the $i$-th iteration, and $\{W_t\}_{t\geq 0}$ is the standard one-dimensional Brownian motion.
\end{theorem}

As an important special case, when $\w = \e_k$ for $k = 1,2, \ldots, d$, we have the convergence in each coordinate to the following pivotal limiting distribution,
\begin{align}\label{eq:fclt2}
\frac{\sqrt{n}(\overline{\tth}_{n, k} - \tth^\star_{k})}{\sqrt{V_{n, kk}}} \Longrightarrow \frac{W_1}{\sqrt{\int_0^1 \left(W_r-r W_1\right)^2 \intd r}},
\end{align}
For the asymptotic distribution defined on the right hand side in \eqref{eq:fclt2}, we repeat the quantiles of the distribution published by \cite{abadir1997two} in Table~\ref{table:fclt}\footnote{Since the distribution on the right hand side of \eqref{eq:fclt} is symmetric, we provide one-side quantiles only. 
}. Combining the asymptotic results in \eqref{eq:fclt2} and Table~\ref{table:fclt}, we can construct coordinate-wise confidence intervals for the true parameter $\tth^\star$. In addition, as 
\begin{align}\label{eq:updatevn}
V_n&= \frac{1}{n^2}\sum_{i= 1}^n i^2(\overline{\tth}_i - \overline{\tth}_n)(\overline{\tth}_i - \overline{\tth}_n)^{\top}\notag\\
&=\frac{1}{n^2}\sum_{i= 1}^n i^2\overline{\tth}_i \overline{\tth}_i^\top- \frac{\overline{\tth}_n}{n^2}\sum_{i= 1}^n i^2\overline{\tth}_i^\top-\frac{1}{n^2}\Big(\sum_{i= 1}^n i^2\overline{\tth}_i\Big)\overline{\tth}_n+\frac{1}{n^2}\sum_{i= 1}^n i^2\overline{\tth}_n\overline{\tth}_n^{\top}
\end{align}
can be constructed in an online fashion via the iterative updates of the matrix $\sum_{i= 1}^n i^2\overline{\tth}_i \overline{\tth}_i^\top$ and the vector $\sum_{i= 1}^n i^2\overline{\tth}_i$, the proposed online inference procedure only requires one pass over the data.

\begin{table}
  \centering
  \small
  \begin{tabular}{c|cccc}
  \toprule
   Quantile & 90\% & 95\% & 97.5\% & 99\% \\
  \hline
  \cite{abadir1997two} Table 1 & 3.875 & 5.323 & 6.747 & 8.613 \\
  \bottomrule
  \end{tabular}
  \caption{Cumulative probability table of the limiting distribution.}
  \label{table:fclt}
\end{table}

\section{Numerical Experiments}
\label{sec:exp}

In this numerical section, we first investigate the empirical performance of the proposed inference procedures and their corresponding coverage rates. We consider linear regression and logistic regression in this section (Examples \ref{ex1}--\ref{ex2}) and conduct simulations for quantile regression (Example \ref{ex3}) in the next subsection. Here $\{\x_i, y_i\}_{i=1}^n$ is an \emph{i.i.d.} sample with the covariate $ \x \sim \N(\0, \Sigma)$ and the response $y \in \R$. We set the sample size $n = 10^5$ and the parameter dimension $d=5,20,50$. 
The true model parameter  $\tth^\star \in \R^d$ is selected uniformly from the unit sphere. For both models, we consider two different structures of the covariance matrices $\Sigma$: identity matrix $I_d$ and equicorrelation covariance matrix (Equicorr in the tables), i.e., $\Sigma_{k\ell} = 0.2$ for all $k \neq \ell$ and $\Sigma_{kk}=1$ for $k=1,2,\dots,d$. The stepsize and spacing parameters for the first $50\times d$ iterations are set flat to avoid a sharp change in the learning rate. Particularly, the stepsize $\eta_n=\eta_0(\max\{n,50d\})^{-\alpha}$ and spacing parameters $h_n=h_0(\max\{n,50d\})^{-\gamma}$ where the exponents are set to $\alpha=\gamma=0.501$ to satisfy the assumptions in Theorem \ref{thm:clt}. The constant $h_0$ is set to $0.01$ for both examples, and $\eta_0$ is a tunable hyperparameter. The  variance of noise $\varepsilon$ in the linear regression model (Example \ref{ex1}) is set to $\sigma^2 = 1$. For both examples, the algorithm initialized from $\tth_0$ randomly sampled spherically with radius $0.01$.  

\begin{table}[!t]
  \centering{}
  \small
  \renewcommand{\baselinestretch}{1.6}
  \begin{tabular}{ccc|c|ccc|ccc}
    \toprule
    &$d$ & $\Sigma$ &{Estimation error} & \multicolumn{3}{c|}{Average coverage rate} & \multicolumn{3}{c}{Average length}\\
    &&& (standard error) & Plug-in &  Oracle &Fixed-$b$ & Plug-in & Oracle & Fixed-$b$  \\
    \hline
    \multicolumn{3}{l|}{Linear} &&&&&&&\\
    &                      &  Identity    &  0.015	(	0.005	)&	0.944	&	0.938	&	0.940	&	0.028	&	0.028	&	0.036	\\
        &    \multirow{-2}{*}{5}                  & Equicorr  & 0.017	(	0.006	)&	0.958	&	0.954	&	0.946	&	0.032	&	0.032	&	0.041	\\
          &                      &   Identity          & 0.066	(	0.010	)&	0.943	&	0.938	&	0.928	&	0.058	&	0.056	&	0.074	\\
    &\multirow{-2}{*}{20}  & Equicorr & 0.082	(	0.014	)&	0.938	&	0.931	&	0.923	&	0.071	&	0.068	&	0.087	\\
            &                      &  Identity    &   0.180	(	0.018	)&	0.947	&	0.917	&	0.881	&	0.097	&	0.089	&	0.108	\\
        &    \multirow{-2}{*}{50}                  & Equicorr  &0.227	(	0.026	)&	0.937	&	0.912	&	0.860	&	0.121	&	0.110	&	0.126	\\
   \hline
    \multicolumn{3}{l|}{Logistic} &&&&&&&\\
    &                      &  Identity      & 0.037	(	0.011	)&	0.946	&	0.938	&	0.916	&	0.065	&	0.065	&	0.075	\\
    &    \multirow{-2}{*}{5}                  & Equicorr  & 0.042	(	0.015	)&	0.934	&	0.932	&	0.908	&	0.073	&	0.073	&	0.085	\\
  &                      &   Identity          & 0.152	(	0.025	)&	0.943	&	0.937	&	0.862	&	0.128	&	0.125	&	0.136	\\
    &\multirow{-2}{*}{20}  & Equicorr & 0.177	(	0.030	)&	0.939	&	0.935	&	0.848	&	0.154	&	0.150	&	0.158	\\
    &                      &  Identity      & 0.404	(	0.040	)&	0.914	&	0.912	&	0.688	&	0.199	&	0.197	&	0.140	\\
    &    \multirow{-2}{*}{50}                  & Equicorr  & 0.495	(	0.051	)&	0.920	&	0.917	&	0.620	&	0.245	&	0.241	&	0.142	\\
    \bottomrule
  \end{tabular}
  \caption{Estimation errors, averaged coverage rates, and average lengths of the proposed algorithm with search direction {\ti} and two function queries ($m=1$). Sample size $n=10^5$. Corresponding standard errors are reported in the brackets. We compare the plug-in covariance estimator (plug-in) based inference \eqref{eq:hessian-est2}  and fixed-$b$ HAR (fixed-$b$)  based inference \eqref{eq:fclt2}.}
  \label{table:inference}
\end{table}

We first report the performance of {\akw} with the search direction uniformly sampled from the natural basis, referred to as {\ti} in Section \ref{sec:choice_dir}. 
In Table~\ref{table:inference}, we present the mean and standard error of the estimation errors in the Euclidean norm (i.e., $\| \overline{\tth}_n - \tth^\star \|$, see the first column), with $100$ Monte-Carlo simulations. 
Next, we set the nominal coverage probability as $95\%$ and we project $\tth\in\mathbb{R}^d$ onto $\e_k$ to construct confidence intervals, where $\e_k$ is the standard basis in $\R^d$ with the $k$-th coordinate as $1$ and the other coordinates as $0$. We record the average coverage rate and the average length of the intervals among the $d$ coordinates for (1) the plug-in covariance matrix estimator\footnote{{Here we use updating probability $p=1$ for the plug-in estimation. In other words, $d^2+1$ queries of function values are obtained at each iteration.}} \eqref{eq:emp_hess} and (2) the fixed-$b$ HAR procedure in  \eqref{eq:updatevn}, for each simulation and report the mean coverage and median interval lengths. As an oracle benchmark, we also report the length of the confidence interval with respect to the true covariance matrix $H^{-1}QH^{-1}$ of the plug-in approach and the corresponding mean coverage rate. As shown from Table \ref{table:inference}, the coverage rate of the plug-in covariance estimator and the oracle coverage rates are very close to the desired $95\%$ coverage, while the fixed-$b$ HAR approach is comparable in small dimension $d=5, 20$ but has lower coverage rates for the large dimension $d=50$.   The average lengths of the plug-in method are comparable to the lengths derived from the true limiting covariance. Due to the space constraints, we relegate the additional simulation results for other choices of direction distributions and multi-query methods to Section \ref{sec:simu_supp} of the Appendix.

\begin{table}[t]
  \centering{}
  \small
  \renewcommand{\baselinestretch}{1.6}
  \begin{tabular}{ccc|c|ccc|ccc}
    \toprule
    &\multirow{2}{*}{$\tau$} & Search& Estimation error & \multicolumn{3}{c|}{Plug-in} & \multicolumn{3}{c}{Fixed-$b$}\\
        &&direction &  (standard error) & Coverage & Length & Time & Coverage & Length & Time  \\
    \hline
    &                      & {\ti} & 0.041	(	0.007	)&0.923	&0.033	&	318.9	&0.893	&0.041	&	99.7	\\
    &                      & {\ts} & 0.042	(	0.007	)&0.950	&0.040	&	306.4	&0.885	&0.041	&	89.0	\\
    &\multirow{-3}{*}{0.1}   & {\tg} & 0.043	(	0.007	)&0.896	&0.032	&	344.4	&0.873	&0.043	&	83.8	\\ \hline
    &                      & {\ti} &0.027	(	0.004	)&0.903	&0.020	&	303.7	&0.915	&0.028	&	94.8	\\
        &                      & {\ts} & 0.026	(	0.004	)&0.904	&0.020	&	295.0	&0.919	&0.027	&	85.4	\\
    &\multirow{-3}{*}{0.5}   & {\tg} & 0.027	(	0.004	)&0.928	&0.022	&	268.5	&0.911	&0.028	&	64.2	\\ \hline
    &                      & {\ti} & 0.041	(	0.006	)&0.934	&0.034	&	299.8	&0.900	&0.041	&	93.7	\\
    &                      & {\ts} & 0.041	(	0.007	)&0.949	&0.040	&	293.1	&0.904	&0.042	&	84.9	\\
    &\multirow{-3}{*}{0.9}   & {\tg} & 0.043	(	0.007	)&0.892	&0.032	&	267.3	&0.897	&0.043	&	64.0	\\
        \bottomrule
  \end{tabular}
\caption{Estimation errors, averaged coverage rates (Coverage), median interval lengths (Length), and computation time (Time) of the proposed algorithm with search direction {\ti}, {\ts}, {\tg} defined in Section \ref{sec:choice_dir} and two function queries ($m=1$), under quantile regression model. Sample size $n=10^6$, dimension $d=20$. Corresponding standard errors are reported in the brackets. We compare the plug-in covariance estimator (plug-in) based inference \eqref{eq:hessian-est2}  and fixed-$b$ HAR (fixed-$b$) based inference \eqref{eq:fclt2}.  }  \label{table:quantile}
\end{table}

\subsection{Numerical Experiments on Non-smooth Loss Function}\label{eqref:simuqr}

In this section, we provide simulation studies to illustrate the performance of the {\akw} estimator and inference procedures on quantile regression. Our data is generated from a linear model, $y_i = \x_i^\top \tth^\star + \epsilon_i$, where $\{\zet_i=(\x_i,y_i)\}_{i=1}^n$ is an \emph{i.i.d.} sample with the covariate $ \x \sim \N(\0, \Sigma)$ and the noise $\{\epsilon_i\}$ follows an \emph{i.i.d.} normal distribution such that $\epsilon_i \sim \mathcal{N}(-\sigma \Phi^{-1}(\tau), \sigma^2)$, $\operatorname{Pr}\left(\epsilon_{i} \leq 0 \mid \boldsymbol{x}_{i}\right)=\tau$.
Here $\Phi(\cdot)$ is the cumulative density function of standard normal distribution and $\Phi^{-1}(\cdot)$ is its inverse function. For each quantile level $\tau \in (0,1)$, the individual loss is $f(\boldsymbol{\theta} ; \zeta)=\rho_{\tau}\left(y-\boldsymbol{x}^{\top} \boldsymbol{\theta}\right)$, where $\rho_{\tau}(z)=z\left(\tau-1{\{z<0\}}\right)$. In this example, we have $H= \frac{1}{\sigma}\phi\big(\Phi^{-1}(\tau)\big) \E [\x \x^\top] $ and $Q=\E[ \v\v^\top S \v\v^\top]$ where $S= \tau (1-\tau) \E [\x \x^\top] $, according to Theorem \ref{thm:clt-qt}. The explicit form of $Q$ is provided in Proposition \ref{cor:clt0}, for example, if we sample uniformly from the canonical basis with two function queries ($m=1$), then $\Qi = d \, \diag(S)$. 

In the numerical experiments below, we fix sample size $n=10^6$, dimension $d=20$, and the noise variance $\sigma^2=1$. The stepsizes and spacing parameters are set with same specifications except for $h_0=1$ and $\eta_0=0.1$. We present our results below in Table~\ref{table:quantile} with three quantile levels $\tau = 0.1, 0.5, 0.9$ and three searching direction schemes {\ti}, {\ts}, {\tg}, detailed descriptions of which are presented in Proposition \ref{cor:clt0} of Section \ref{sec:choice_dir}. As can be inferred from the table, plug-in estimators have a good coverage rate close to the oracle ones. The fixed-$b$ HAR inference structure provides coverage around $90\%$ without additional function queries.

\section*{Acknowledgements} The authors are very grateful to Professor Yuan Liao from Rutgers University for the constructive comments that improved the quality of this paper.

\setcounter{section}{0}
\renewcommand\thesection{\Alph{section}}
\appendix

\numberwithin{figure}{section}
\numberwithin{table}{section}

\clearpage
\section*{Appendix}

Throughout the appendix, we will assume, without loss of generality, $F(\cdot)$ achieves its minimum at $\tth^\star = \0$ and $F(\0) = 0$. We now introduce some notations as follows,
\begin{align*}
\bxi_n &= \nabla F (\tth_{n-1}) - \E_{n-1} (\frac{1}{h_n}[F(\tth_{n-1}+h_n \v_n)-F(\tth_{n-1})] \v_n), \\
\bgamma_n &= \E_{n-1} \frac{1}{h_n}[F(\tth_{n-1}+h_n \v_n)-F(\tth_{n-1})] \v_n - \frac{1}{h_n}[F(\tth_{n-1}+h_n \v_n)-F(\tth_{n-1})] \v_n, \\
\bvarepsilon_n &= \frac{1}{h_n}[F(\tth_{n-1}+h_n \v_n)-F(\tth_{n-1})] \v_n - \frac{1}{h_n}[f(\tth_{n-1}+h_n \v_n ;\zet_n)-f(\tth_{n-1}; \zet_n)]\v_n.
\end{align*}
We also note that $C$ refers to constants that may change from equation to equation. 

\section{Appendix for Section~\ref{sec:clt}}
\label{sec:app-a}

\subsection{Two-query approximation}
\label{subsec:two-query}

\subsubsection*{Proof of Lemma~\ref{lem:bv}}

\begin{proof}
By definition, $\E_{\zet} \; \widehat{g}_{h, \v}(\tth; \zet) = \frac{1}{h}\Delta_{h,\v}F(\tth) \v = \frac{1}{h} \left[F(\tth+ h \v)-F(\tth)\right] \v$. For the first inequality, we have
\begin{align}
\label{eq:lem1:est1}
\left\| \E \; \widehat{g}_{h, \v}(\tth; \zet) - \nabla F(\tth) \right\| &=\left\| \E \; \frac{1}{h} \left[F(\tth+ h \v)-F(\tth)\right] \v - \nabla F(\tth) \right\| \notag\\
&= \left\| \E \; \v \v^\top \nabla F(\tth) + \frac{1}{2}h\E \; \v \v^\top \nabla^2 F(\tth_{h, \v}) \v  - \nabla F(\tth) \right\| \notag\\
&=  \frac{1}{2}h\left\| \E \; \v \v^\top \nabla^2 F(\tth_{h, \v}) \v \right\| \notag\\
&\leq \frac{1}{2}h L_f \E \|\v\|^3,
\end{align}
where in the third equality we use the Taylor expansion of $F(\tth)$, and $\tth_{h, \v}$ comes from the remainder term of the Taylor expansion.
\end{proof}

\subsubsection*{Proof of Proposition~\ref{lemma:x}}

\begin{lemmax}
{ Assume Assumptions \ref{assumption1p}, \ref{assumption2}, and \ref{assumption4} hold.} Set the step size as $\eta_n = \eta_0 n^{-\alpha}$ for some constant $\eta_0>0$ and $\alpha\in \left(\frac{1}{2}, 1\right)$ and the spacing parameter as $h_n = h_0 n^{-\gamma}$ for some constant $h_0 > 0$, and $\gamma \in \left(\frac{1}{2}, 1\right)$. The {\kw} iterate $\tth_n$ converges to $\tth^\star$ almost surely. 
{ Furthermore, assume Assumptions~\ref{assumption1p} holds with $\delta_1=+\infty$.} For sufficiently large $n$, we have for $0 < \delta \leq 2$,
\begin{align*}
\E\| \tth_n - \tth^\star\|^{2+\delta} &\leq C n^{-\alpha(2+\delta)/2}.
\end{align*}
where the constant $C$ depends on $d, \lambda,  L_f,\alpha, \gamma,$ $\eta_0, h_0$.
\end{lemmax}
\begin{remark} \label{rmk:bach} The parameter dependency in Proposition~\ref{lemma:x} could be given explicitly as follows,
\begin{align*}
\E \| \tth_n - \tth^\star\|^2 & \leq \exp\left(CM_1\eta_0/(2\alpha -1)+C M_2/(2\beta -1)-C\lambda \eta_0 n^{1-\alpha}/(1-\alpha)\right) \|\tth_0\|^2\\
& \quad + M_3\left( \exp\left(-C \lambda \eta_0 n^{1-\alpha}/(1-\alpha)\right) + \frac{\eta_0 n^{-\alpha}}{\lambda} \right)\\
& \quad + \frac{M_3}{M_1}\exp\left(CM_1\eta_0/(2\alpha -1)+CM_2/(2\beta -1)-C\lambda \eta_0 n^{1-\alpha}/(1-\alpha)\right),
\end{align*}
where the constant $C$ above is a universal constant that does not depend on any constant/parameter in the assumptions. The other terms $M_1, M_2, M_3$ above are given below,
\begin{align*}
M_1 &= C \left(L_f^2 \E \|\v\|^4 + M^{\frac{2}{2+\delta}} \E \|\v\|^{4} +  L_f^2\right), \\
M_2 &= C L_f^2 \E\|\v\|^3, \\
M_3 &= C\left(\E \|\v\|^3 + \left(h_n^2 L_f^2 \E \|\v\|^6 + M^{\frac{2}{2+\delta}} \E \|\v\|^{4} (h_n^{2}\|\v\|^{2} +1)\right)\right).
\end{align*}
\end{remark}

We will prove both Proposition~\ref{lemma:x} and Remark~\ref{rmk:bach} below.
\begin{proof}
{ We first assume Assumptions \ref{assumption1p}, \ref{assumption2}, and \ref{assumption4} and give some bounds on $\bxi_n, \bgamma_n, \bvarepsilon_n$.} By definition, $\E_{n-1}\bgamma_n = \E_{n-1}\bvarepsilon_n = 0$. From \eqref{eq:lem1:est1},
\begin{align}
\label{eq:fact2}
\|\bxi_n\| \leq \frac{1}{2}h_n L_f \E \|\v\|^3.
\end{align}
We can bound $\bgamma_n$ by the following
\begin{align}
\label{eq:fact3}
\E \|\bgamma_n \|^2 &\leq \E \left\|\frac{1}{h_n}[F(\tth_{n-1}+h_n \v_n)-F(\tth_{n-1})] \v_n \right\|^2 \notag\\
&\leq \E \|\langle \nabla F(\tth_{n-1}), \v_n \rangle \v_n \|^2 + \frac{1}{4} h_n^2 L_f^2 \E \|\v\|^6 \notag\\
&\leq L_f^2 \E \|\v\|^4 \E\|\tth_{n-1}\|^2 + \frac{1}{4} h_n^2 L_f^2 \E \|\v\|^6.
\end{align}
We also have the following fact for $\bvarepsilon$.
\begin{align}
\label{eq:fact4}
&\E_{n-1} \left[ \|\bvarepsilon_n\|^2 | \v_n \right] \notag\\
= \; \; &\E_{n-1} \left[ \left\|\frac{1}{h_n}\int_0^{h_n}\langle \nabla F(\tth_{n-1}+s \v_n) - \nabla f(\tth_{n-1}+s \v_n;\zet_n), \v_n\rangle \v_n \intd s\right\|^2 \bigg| \v_n \right] \notag \\
\leq \; \; &\|\v_n\|^{4} \E_{n-1} \left[ \frac{1}{h_n}\int_0^{h_n} \left\| \nabla F(\tth_{n-1}+s \v_n) - \nabla f(\tth_{n-1}+s \v_n;\zet_n)\right\|^{2} \intd s \big| \v_n \right] \notag \\
\leq \; \; &M^{\frac{2}{2+\delta}} \|\v_n\|^{4} \frac{1}{h_n}\int_0^{h_n}( \|\tth_{n-1}+s \v_n\|^{2}+1) \intd s \notag \\
\leq \; \; &M^{\frac{2}{2+\delta}} \|\v_n\|^{4} (\|\tth_{n-1}\|^{2}+h_n^{2}\|\v_n\|^{2} +1),
\end{align}
where in the second inequality, we use Assumption~\ref{assumption2}.

Now decompose the update step as follows,
\begin{align*}
\tth_{n}&=\tth_{n-1}-\eta_n\frac{1}{h_n}[f(\tth_{n-1}+h_n \u_n ;\zet_n)-f(\tth_{n-1} ;\zet_n)]\\
&= \tth_{n-1}-\eta_n\nabla F(\tth_{n-1}) + \eta_n(\bxi_n + \bgamma_n + \bvarepsilon_n).
\end{align*}
{ We can derive that,
\begin{align} \label{eq:decomp-theta}
\| \tth_n \|^2 \leq &\| \tth_{n-1} \|^2 - 2 \eta_n \langle \nabla F(\tth_{n-1}), \tth_{n-1}\rangle + 2 \eta_n \langle \bxi_n + \bgamma_n + \bvarepsilon_n, \tth_{n-1} \rangle \notag\\
&+\eta_n^2 \| \bxi_n + \bgamma_n + \bvarepsilon_n - \nabla F(\tth_{n-1}) \|^2.
\end{align}
For the first part in the RHS of the above inequality, using Lemma B.1 in \cite{su2018uncertainty}, we have
\begin{align}
\langle \tth, \nabla F(\tth) \rangle \geq \rho \| \tth \| \min \left\{\| \tth \|, \delta_1\right\}.
\end{align}
for some $\rho > 0$. Moreover,
\begin{align}
\label{eq:RHSp2}
\left| \eta_n \E_{n-1} \langle \bxi_n + \bgamma_n + \bvarepsilon_n, \tth_{n-1} \rangle \right|
&= \eta_n \left|\E_{n-1} \langle \bxi_n, \tth_{n-1}\rangle \right| \notag \\
&\leq \frac{1}{2} \eta_n h_n L_f\|\tth_{n-1}\| \E \|\v\|^3 \notag \\
&\leq C L_f^2 \E \|\v\|^3 h_n^2 \| \tth_{n-1} \|^2 + C \E \|\v\|^3 \eta_n^2, \\
\label{eq:RHSp3}
\E_{n-1}\| \bxi_n + \bgamma_n + \bvarepsilon_n - \nabla F(\tth_{n-1}) \|^2
&\leq 4\| \bxi_n \|^2 + 4\|\bgamma_n \|^2+ 4\| \bvarepsilon_n \|^2 +4\| \nabla F(\tth_{n-1}) \|^2\notag \\
&\leq h_n^2 L_f^2\E (\|\v\|^3)^2 + 4 L_f^2 \E \|\v\|^4 \|\tth_{n-1}\|^2 + h_n^2 L_f^2 \E \|\v\|^6 \notag \\
& \quad + 4 M^{\frac{2}{2+\delta}} \E \|\v_n\|^{4} (\|\tth_{n-1}\|^{2}+h_n^{2}\|\v_n\|^{2} +1) + 4 L_f^2 \|\tth_{n-1}\|^2 \notag\\
&: = M_1 \|\tth_{n-1}\|^2 + M_2
\end{align}
where we use Cauchy-Schwarz inequality in \eqref{eq:RHSp2}, \eqref{eq:RHSp3} and $M_1 = C \bigl(L_f^2 \E \|\v\|^4 + M^{\frac{2}{2+\delta}} \E \|\v\|^{4} +  L_f^2\bigr), M_2 = C\bigl(h_n^2 L_f^2 \E \|\v\|^6 + M^{\frac{2}{2+\delta}} \E \|\v\|^{4} (h_n^{2}\|\v\|^{2} +1)\bigr)$.

Combining all estimates, we have
\begin{align*}
\E_{n-1} \| \tth_n \|^2 \leq \left(1 + C \eta_n^2 + C h_n^2\right) \| \tth_{n-1} \|^2 - 2 \eta_n \rho \| \tth_{n-1} \| \min \left\{\| \tth_{n-1} \|, \delta_1\right\} +  C \eta_n^2.
\end{align*}
This provides an analogous recursion form as in \cite[proof of Theorem 2, part 1]{polyak1992acceleration}. Given the assumptions on $\eta_n$ and $h_n$, and we have $\tth_n$ converges almost surely to $\0$, and
\begin{align}\label{eq:asconv}
\sum_{i=1}^\infty \frac{\|\tth_i\|^2}{i^{\frac{1}{2}}} < \infty
\end{align}
almost surely, where \eqref{eq:asconv} is analogous to the second last equation of \cite[proof of Theorem 2, part 4]{polyak1992acceleration}.

Now assume Assumption~\ref{assumption1p} holds with $\delta_1=+\infty$, we have a stronger estimate,
\begin{align*}
\langle \nabla F(\tth_{n-1}), \tth_{n-1}\rangle \geq F(\tth_{n-1}) + \frac{\lambda}{2} \| \tth_{n-1} \|^2 \geq \lambda \| \tth_{n-1} \|^2.
\end{align*}
} So combining all inequalities, we have
\begin{align}
\label{eq:recur1}
\E_{n-1} \| \tth_n \|^2 \leq \left[1 - 2\lambda \eta_n + M_1 \eta_n^2 + M_3 h_n^2 \right] \| \tth_{n-1} \|^2 +  M_4 \eta_n^2,
\end{align}
where $M_3, M_4$ is defined by $M_3 = C L_f^2 \E\|\v\|^3$, $M_4 = C(\E \|\v\|^3 + M_2)$. Following the proof of Theorem 1 of \cite{bach2011non}, we can apply the recursion and get
\begin{align*}
\E \| \tth_n \|^2 \leq \prod_{k=1}^n \left[1 - 2\lambda \eta_k + M_1 \eta_k^2 + CM_3 h_k^2\right] \| \tth_{0} \|^2+  M_4 \sum_{k=1}^n \prod_{i=k+1}^n \left[1 - 2\lambda \eta_i + M_1 \eta_k^2 + M_3 h_k^2 \right]\eta_k^2.
\end{align*}
We can then bound the first term on the RHS,
\begin{align*}
\prod_{k=1}^n \left[1 - 2\lambda \eta_k + M_1 \eta_k^2 +M_3 h_k^2\right]
\leq \exp\left(- 2\lambda \sum_{k=1}^n\eta_k\right)\exp\left(M_1\sum_{k=1}^n \eta_k^2\right)\exp\left(M_3\sum_{k=1}^n h_k^2\right),
\end{align*}
as well as the second term on the RHS
\begin{align*}
&\sum_{k=1}^n \prod_{i=k+1}^n \left[1 - 2\lambda \eta_i + M_1 \eta_k^2 + M_3 h_k^2\right]\eta_k^2 \\
\leq \; \; &\exp\left(-\lambda \sum_{k=m+1}^n \eta_k\right) \sum_{k=1}^n \eta_k^2 +\frac{\eta_m}{\lambda} +\frac{1}{M_1}\exp\left(M_1\sum_{k=1}^{n_0} \eta_k^2\right)\exp\left(M_3\sum_{k=1}^{n_0} h_k^2\right)\exp\left(-\lambda \sum_{k=1}^n \eta_k \right),
\end{align*}
where we denote by $n_0 = \inf \{k\in \mathbb{N}, 1 - 2\lambda \eta_k + M_1 \eta_k^2 + M_3 h_k^2  \leq 1 - \lambda \eta_k\}$ and $m$ is any integer in $\{1, \dots, n\}$. Choose $m = n/2$ and bound $n_0$ by $n$. Notice that $\sum_{k=1}^n \eta_k^2$ converge. So we can get
\begin{align*}
\E \| \tth_n \|^2 & \leq \exp\left(CM_1\eta_0/(2\alpha -1)+CM_3/(2\beta -1)-C\lambda \eta_0 n^{1-\alpha}/(1-\alpha)\right) \|\tth_0\|^2\\
& \quad + M_4\left( \exp\left(-C \lambda \eta_0 n^{1-\alpha}/(1-\alpha)\right) + \frac{\eta_0 n^{-\alpha}}{\lambda} \right)\\
& \quad + \frac{M_4}{M_1}\exp\left(CM_1\eta_0/(2\alpha -1)+CM_3/(2\beta -1)-C\lambda \eta_0 n^{1-\alpha}/(1-\alpha)\right).
\end{align*}
Only the term $M_4\eta_0 n^{-\alpha}/\lambda$ decreases at the order of $O(n^{-\alpha})$ while all the other terms decrease much faster. 

Notice that all $C$'s in the above inequality are universal constants which do not depend on any parameters in the assumptions. This proves Remark~\ref{rmk:bach}.

From now on, we will absorb all parameters (other than $n$) into $C$ to make the asymptotic analysis more clear. By martingale convergence theorem, $\|\tth_n\|$ converges almost surely. Because its second moment converges to $\0$, it must converge to $\0$ almost surely.

We now show that,
\begin{align*}
\E \|\tth_n - \tth^\star\|^{2+\delta} \leq Cn^{-\alpha(2+\delta)/2}.
\end{align*}
By same arguments as in \eqref{eq:fact2}, \eqref{eq:fact3}, \eqref{eq:fact4}, we can get $\|\bxi_n\|^{2+\delta} \leq C h_n^{2+\delta}$, $\E_{n-1} \|\bgamma_n\|^{2+\delta} \leq \|\tth_{n-1}\|^{2+\delta} + Ch_n^{2+\delta}$, $\E_{n-1} \left[ \|\bvarepsilon_n\|^{2+\delta} \right] \leq C(\|\tth_{n-1}\|^{2+\delta} + 1)$.

By similar arguments as in Lemma~\ref{lemmamoments}, there exists constants $C$ such that for any $\a, \b$,
\begin{align*}
\|\a + \b\|^{2+\delta} \leq \|\a\|^{2+\delta} + (2+\delta)\langle \a, \b\rangle \|\a\|^{\delta} + C\|\a\|^{\delta} \|\b\|^2  + C\|\b\|^{2+\delta}.
\end{align*}
So we have the bound
\begin{align*}
\E_{n-1}\|\tth_n\|^{2+\delta} &\leq \|\tth_{n-1}\|^{2+\delta} + \eta_n(2+\delta)\E_{n-1} \langle \tth_{n-1}, -\nabla F(\tth_{n-1}) + \bxi_n+ \bgamma_n + \bvarepsilon_n \rangle \|\tth_{n-1}\|^{\delta}\\
& \quad  + C\eta_n^2 \|\tth_{n-1}\|^{\delta} \E_{n-1}\|-\nabla F(\tth_{n-1}) +\bxi_n+ \bgamma_n + \bvarepsilon_n\|^2 \\
& \quad + C\eta_n^{2+\delta}\E_{n-1}\|-\nabla F(\tth_{n-1}) +\bxi_n+ \bgamma_n + \bvarepsilon_n\|^{2+\delta}\\
& \leq (1-(2+\delta)\lambda\eta_n) \|\tth_{n-1}\|^{2+\delta} + C\eta_n h_n \|\tth_{n-1}\|^{1+\delta} \\
& \quad + C\eta_n^2 (\|\tth_{n-1}\|^2+1) \|\tth_{n-1}\|^{\delta} + C\eta_n^{2+\delta}(\|\tth_{n-1}\|^{2+\delta}+1).
\end{align*}
If $0<\delta\leq 1$, by previous bound $\E\|\tth_n\|^2 \leq Cn^{-\alpha}$, we can get $\E\|\tth_n\|^{1+\delta} \leq Cn^{-\alpha(1+\delta)/2}$ and $\E\|\tth_n\|^\delta \leq Cn^{-\alpha\delta/2} $ by H\"older's inequality. So we can further get
\begin{align*}
\E\|\tth_n\|^{2+\delta} \leq (1-Cn^{-\alpha}+Cn^{-2\alpha}) \E \|\tth_{n-1}\|^{2+\delta} +C n^{-(2+\delta)\alpha/2},
\end{align*}
which implies $\E\|\tth_n\|^{2+\delta} \leq C n^{-(2+\delta)\alpha/2}$ as in the above proof after \eqref{eq:recur1}.

Now the case for $0 < \delta \leq 1$ is proved. We can then use induction. If $\E\|\tth_n\|^{2+\delta} \leq C n^{-(2+\delta)\alpha/2}$ for all $\delta \leq n$, then we can use the same method to prove the same inequality holds for $\delta \in (n, n+1]$. Thus the inequality holds for all $\delta$.
\end{proof}

\subsubsection*{Proof of Lemma~\ref{lem:variance}}

\begin{proof}
By Assumption~\ref{assumption2}, we know that
\begin{align*}
\E \|\nabla f(\tth;\zet) - \nabla F(\tth)\|^{2+\delta} \leq M(\|\tth\|^{2+\delta}+d^{2+\delta}).
\end{align*}
Therefore, the following holds for some constant $C > 0$,
\begin{align}
\label{eq:lem1:est2}
\E \|\nabla f(\tth;\zet) - \nabla F(\tth)\|^2 \leq C(\|\tth\|^2+d^2).
\end{align}
In particular,
\begin{align}
\label{eq:lem1:est3}
\E \|\nabla f(\0; \zet) - \nabla F(\0)\|^2 \leq C.
\end{align}
From Assumption~\ref{assumption3}, we can get the following estimate for the Hessian matrix $\nabla^2 f(\tth;\zet)$,
\begin{align*}
\E \| \nabla^2 f(\tth;\zet) \|^2 &\leq 2\E \| \nabla^2 f(\0;\zet) \|^2 + 2\E \left\| \nabla^2 f(\tth;\zet) - \nabla^2 f(\0;\zet) \right\|^2 \\
&\leq C(1+\|\tth\|^2).
\end{align*}
Using the above observation, we find that
\begin{align}
\label{eq:lem1:est4}
& \quad \E \|\nabla f(\tth; \zet) - \nabla F(\tth) - \nabla f(\0; \zet) + \nabla F(\0)\|^2 \notag\\
&\leq C\|\tth\|^2 + 2  \E \|\nabla f(\tth;\zet)-\nabla f(\0;\zet))\|^2 \notag\\
&=C\|\tth\|^2 + 2  \E \left\|\int_0^1 \nabla^2 f(s\tth;\zet) \tth \intd s \right\|^2 \notag\\
&\leq C\|\tth\|^2 + 2  \E \int_0^1 \| \nabla^2 f(s\tth;\zet) \tth\|^2 \intd s \notag\\
&\leq C\|\tth\|^2 (1+\int_0^1 \E  \| \nabla^2 f(s\tth;\zet)\|^2 \intd s) \notag\\
&\leq C\|\tth\|^2(1+\|\tth\|^2).
\end{align}
Define the function $\Sigma(\tth_1, \tth_2)$ by
\begin{align*}
\Sigma(\tth_1, \tth_2):= \E (\nabla f(\tth_1; \zet) - \nabla F(\tth_1))(\nabla f(\tth_2; \zet) - \nabla F(\tth_2))^\top.
\end{align*}
Then combining inequalities \eqref{eq:lem1:est2}, \eqref{eq:lem1:est3}, \eqref{eq:lem1:est4}, we have
\begin{align}
\label{eq:lem1:est5}
\|\Sigma(\tth_1, \tth_2) -S\|&\leq \E \|(\nabla f(\tth_1; \zet) - \nabla F(\tth_1))(\nabla f(\tth_2; \zet) - \nabla F(\tth_2))^\top \notag\\
& \quad - (\nabla f(\0; \zet) - \nabla F(\0))(\nabla f(\0; \zet) - \nabla F(\0))^\top\| \notag\\
&\leq \E \|\nabla f(\tth_1; \zet) - \nabla F(\tth_1)\|\|\nabla f(\tth_2; \zet) - \nabla F(\tth_2)-\nabla f(\0; \zet) + \nabla F(\0)\| \notag\\
&\quad+\E\|\nabla f(\tth_1; \zet) - \nabla F(\tth_2)-\nabla f(\0; \zet) +\nabla F(\0)\|\|\nabla f(\0; \zet) - \nabla F(\0)\| \notag\\
&\leq C(d+\|\tth_1\|)\|\tth_2\|(1+\|\tth_2\|) +C\|\tth_1\|(1+\|\tth_1\|).
\end{align}
Notice that
\begin{align*}
&\quad \E_{\zet} \widehat{g}_{h, \v}(\tth; \zet) \widehat{g}_{h, \v}(\tth; \zet)^\top - (\frac{1}{h}\Delta_{h, \v}F(\tth) \v)(\frac{1}{h}\Delta_{h, \v}F(\tth) \v)^\top \\
&= \E_{\zet} (\widehat{g}_{h, \v}(\tth; \zet) - \frac{1}{h}\Delta_{h, \v}F(\tth) \v)(\widehat{g}_{h, \v}(\tth; \zet) - \frac{1}{h}\Delta_{h, \v}F(\tth) \v)^\top \\
&= \frac{1}{h^2} \E_{\zet} \v  (f(\tth+ h \v; \zet)-f(\tth; \zet) - F(\tth+ h \v) + F(\tth))^2  \v^\top  \\
&= \frac{1}{h^2}  \E_{\zet}\v \v^\top \bigg[ \int_0^{h} \int_0^{h} \left(\nabla F(\tth+s_1 \v) - \nabla f(\tth+s_1 \v;\zet)\right) \\
& \quad \left(\nabla F(\tth+s_2 \v) - \nabla f(\tth+s_2 \v;\zet)\right)^\top \intd s_1 \intd s_2  \bigg| \bigg] \v \v^\top \\
&= \frac{1}{h^2} \E_{\zet} \v \v^\top  \int_0^{h} \int_0^{h} \Sigma(\tth+s_1 \v, \tth+s_2 \v) \intd s_1 \intd s_2 \v \v^\top.
\end{align*}
We can use \eqref{eq:lem1:est5} and derive that
\begin{align*}
&\quad \| \E \widehat{g}_{h, \v}(\tth; \zet) \widehat{g}_{h, \v}(\tth; \zet)^\top - (\frac{1}{h}\Delta_{h, \v}F(\tth) \v) (\frac{1}{h}\Delta_{h, \v}F(\tth) \v)^\top -  \v \v^\top S \v \v^\top \| \\
&\leq C\|\v\|^4 (\|\tth\|+h\|\v\|)(1 + \|\tth\|+h\|\v\|)(d + \|\tth\|+h\|\v\|).
\end{align*}
Now we have
\begin{align}
\label{eq:estcov}
&\quad \| \E \widehat{g}_{h, \v}(\tth; \zet) \widehat{g}_{h, \v}(\tth; \zet)^\top - \E (\frac{1}{h}\Delta_{h, \v}F(\tth) \v) (\frac{1}{h}\Delta_{h, \v}F(\tth) \v)^\top -  \E\v \v^\top S \v \v^\top \| \notag\\
&\leq C\E\|\v\|^4 (\|\tth\|+h\|\v\|)(1 + \|\tth\|+h\|\v\|)(d + \|\tth\|+h\|\v\|).
\end{align}
By the same argument,
\begin{align*}
& \quad \| \E(\frac{1}{h}\Delta_{h, \v}F(\tth) \v)(\frac{1}{h}\Delta_{h, \v}F(\tth) \v)^\top \|\\
& \leq \frac{1}{h^2}  \E \bigg\|\v \v^\top \bigg[ \int_0^{h} \int_0^{h} \left(\nabla F(\tth+s_1 \v)\right) \left(\nabla F(\tth+s_2 \v) \right)^\top \intd s_1 \intd s_2  \bigg| \bigg] \v \v^\top \bigg\| \\
& \leq C \E \|\v\|^4(\|\tth\|^2+h^2\|\v\|^2).
\end{align*}
So we finally get
\begin{align*}
\| \E \widehat{g}_{h, \v}(\tth; \zet) \widehat{g}_{h, \v}(\tth; \zet)^\top -  \E\v \v^\top S \v \v^\top \| \leq C\E\|\v\|^4(\|\tth\|+h\|\v\|)(1 + \|\tth\|+h\|\v\|)(d + \|\tth\|+h\|\v\|).
\end{align*}
for some constant $C > 0$.
\end{proof}

\subsubsection*{Proof of Theorem~\ref{thm:clt}}

\begin{proof}
We follow the proof in \cite{polyak1992acceleration}. The update step is
\begin{align*}
\tth_{n} &= \tth_{n-1}-\eta_n\nabla F(\tth_{n-1}) + \eta_n(\bxi_n + \bgamma_n + \bvarepsilon_n)\\
&= (I_d-\eta_n H)\tth_{n-1}+ \eta_n(H\tth_{n-1}-\nabla F(\tth_{n-1})+\bxi_n + \bgamma_n + \bvarepsilon_n).
\end{align*}
We only need to prove the following three claims. First, the following term converges almost surely
\begin{align}
\label{polyakcondition1}
\sum_{i=1}^\infty \frac{1}{\sqrt{i}}\|H\tth_{i-1}-\nabla F(\tth_{i-1})+\bxi_i\| < \infty.
\end{align}
Furthermore, if the bounded sequence $\{w_i^n\}_{1 \leq i \leq n}$ satisfies $\frac{1}{n}\sum_{i=1}^n \|w_i^n\| \to 0$, then
\begin{align}
\label{polyakcondition2}
\frac{1}{n}\sum_{i=1}^n \|w_i^n (\bgamma_i + \bvarepsilon_i)\|^2 \to 0
\end{align}
in probability. When $t \rightarrow \infty$, the following convergence in probability,
\begin{align}
\label{polyakcondition3}
\frac{1}{\sqrt{n}}\sum_{i=1}^n (\bgamma_i + \bvarepsilon_i) \Longrightarrow \N(\0, Q).
\end{align}
Condition~\eqref{polyakcondition1} is used in Part 4 of the proof of Theorem 2 in \cite{polyak1992acceleration}. It implies the error term introduced by the KW estimator is negligible. To be more precise, Define
\[
\Delta^\prime_i = \Delta^\prime_{i-1} - \eta_i H \Delta^\prime_{i-1} + \eta_n(\bgamma_n + \bvarepsilon_n), \; \; \Delta_0^\prime = \tth_0.
\]
Then
\[
\sqrt{n}\|\bar\Delta^\prime_n - \bar\tth_n\| = \|\frac{1}{\sqrt{n}}\sum_{i=0}^{n-1}(H^{-1} + w^{n}_i)(H\tth_{i-1}-\nabla F(\tth_{i-1})+\bxi_i)\|.
\]
Condition~\eqref{polyakcondition1} means that $\sqrt{n}\|\bar\Delta^\prime_n - \bar\tth_n\| \to 0$ almost surely.

Conditions \eqref{polyakcondition2} and \eqref{polyakcondition3} appears in Part 1 of the proof of Theorem 1 in \cite{polyak1992acceleration} to establish the central limit theorem for $\bar\Delta^\prime_n$. Conditions \eqref{polyakcondition2} is implicitly assumed in the proof, but it is not trivial. It is easy to check that as long as those three conditions are satisfied, the rest of the proof of Theorem 2 in \cite{polyak1992acceleration} works in our problem without change. 

In the previous proof of Proposition~\ref{lemma:x}, we have derived the following estimates that hold when $\E\|\tth_n\|^2$ does not necessarily converges to 0.
\begin{align*}
\E_{n-1} \|\bgamma_n\|^2 &\leq C\|\tth_{n-1}\|^2 + Ch_n^2,\\
\E_{n-1} \|\bvarepsilon_n\|^2 &\leq C\|\tth_{n-1}\|^2 + C,\\
\|\bxi_n\| &\leq Ch_n,
\end{align*}
Assumption~\ref{assumption3} implies that
\begin{align*}
\|\nabla^2 F(\tth) - \nabla^2F(\y) \|^2 \leq L_h \|\tth-\y\|^2.
\end{align*}
By Taylor expansion, we can further derive the bound
\begin{align}\label{eq:use3}
\|H\tth_{i-1}-\nabla F(\tth_{i-1})\| \leq C\|\tth_{i-1}\|^2.
\end{align}
Combining the previous inequality with inequality~\eqref{eq:fact2}, we know that
\begin{align*}
\|H\tth_{i-1}-\nabla F(\tth_{i-1})+\bxi_i\| \leq C(\|\tth_{i-1}\|^2 + h_i^2),
\end{align*}
which indicates that
\begin{align*}
\sum_{i=1}^\infty \frac{1}{\sqrt{i}}\|H\tth_{i-1}-\nabla F(\tth_{i-1})+\bxi_i\| &\leq C \sum_{i=1}^\infty \frac{1}{\sqrt{i}}(\|\tth_{i-1}\|^2 + h_i^2) \\
&\leq C + C\sum_{i=1}^\infty \frac{1}{\sqrt{i}}\|\tth_{i-1}\|^2 < \infty,
\end{align*}
where the last inequality follows from inequality~\ref{eq:asconv}. So condition~\eqref{polyakcondition1} holds. 

Because $\E_{i-1}\|\bgamma_i + \bvarepsilon_i\|^2 \leq C\|\tth_{n-1}\|^2 + C$ is almost surely bounded, for any $\varepsilon$, there exists $N$, such that
\[
\Pr\{C\|\tth_{n-1}\|^2 + C < N, \forall n > 0 \} > 1 - \varepsilon.
\]
So we have the following estimate
\begin{align*}
&\quad \E[\|\bgamma_i + \bvarepsilon_i\|^2 | C\|\tth_{n-1}\|^2 + C < N, \forall n >0] \\
& \leq \E[\|\bgamma_i + \bvarepsilon_i\|^2 I\{C\|\tth_{n-1}\|^2 + C < N, \forall n >0\}]/(1 - \varepsilon)\\
& \leq \E[\|\bgamma_i + \bvarepsilon_i\|^2 I\{C\|\tth_{i-1}\|^2 + C < N\}]/(1 - \varepsilon)\\
& \leq \E[\|\bgamma_i + \bvarepsilon_i\|^2 | C\|\tth_{i-1}\|^2 + C < N]/(1 - \varepsilon) \leq \frac{N}{1 - \varepsilon}.
\end{align*}
So in the event $\{C\|\tth_{n-1}\|^2 + C < N, \forall n > 0 \}$, $\E \|\bgamma_i + \bvarepsilon_i\|^2$ is bounded, so $\frac{1}{n}\sum_{i=1}^n \|w_i^n (\bgamma_i + \bvarepsilon_i)\|^2 \to 0$ in probability conditioned on this event. Because $\varepsilon$ can be arbitrarily small, $\frac{1}{n}\sum_{i=1}^n \|w_i^n (\bgamma_i + \bvarepsilon_i)\|^2 \to 0$ in probability. So condition~\eqref{polyakcondition2} holds. 

To prove condition~\eqref{polyakcondition3}, it suffices to verify that,
\begin{align*}
\frac{1}{\sqrt{n}}\sum_{i=1}^n \bvarepsilon_i \Longrightarrow \N(\0, Q).
\end{align*}
By martingale central limit theorem \cite[Theorem 8.2.4]{durrett2019probability}, we only need to verify two conditions,
\begin{align}
\label{cltcondition1}
\frac{1}{n}\sum_{i=1}^n\E_{i-1}[\bvarepsilon_i \bvarepsilon_i^\top] \rightarrow Q, \\
\label{cltcondition2}
\frac{1}{n}\sum_{i=1}^n\E_{i-1}\left[\|\bvarepsilon_i\|^2 \1_{\|\bvarepsilon_i\|>a\sqrt{n}}\right] \rightarrow 0,
\end{align}
in probability for all $a>0$.

Notice that \eqref{eq:estcov} is equivalent to the following inequality,
\begin{align} \label{expectationoveru}
\| \E_{n-1} \bvarepsilon_{n} \bvarepsilon_n^\top -  \E\v \v^\top S \v \v^\top \| \leq C (\|\tth_{n-1}\|+h_n)(1 + \|\tth_{n-1}\|^2+h_n^2).
\end{align}
Thus $\E_{n-1}[\bvarepsilon_n \bvarepsilon_n^\top]$ converges almost surely to $Q$ and condition~\eqref{cltcondition1} holds.

Now consider the quantity in \eqref{cltcondition2}, by Proposition~\ref{lemma:x},
\begin{align*}
\E_{i-1}\left[\|\bvarepsilon_i\|^2 \1_{\|\bvarepsilon_i\|>a\sqrt{n}}\right] \leq \left[\E_{i-1}\left[\|\bvarepsilon_i\|^{2+\delta}\right] \right]^{\frac{2}{2+\delta}} \left[\E_{i-1}\left[ \1_{\|\bvarepsilon_i\|>a\sqrt{n}}\right]\right]^\frac{\delta}{2+\delta}.
\end{align*}
Note that
\begin{align*}
\E_{i-1}\left[\1_{\|\bvarepsilon_i\|>a\sqrt{n}}\right] = \P_{i-1}\left(\|\bvarepsilon_i\|>a\sqrt{n}|\tth_{i-1} \right) \leq \frac{1}{a\sqrt{n}}\E_{i-1} \|\bvarepsilon_i\|.
\end{align*}
Therefore, it can be bounded by
\begin{align*}
\E_{i-1}\left[\|\bvarepsilon_i\|^2 \1_{\|\bvarepsilon_i\|>a\sqrt{n}}\right] \leq C \left(\frac{1}{a\sqrt{n}}\right)^\frac{\delta}{2+\delta} \left(1+\|\tth_{i-1}\|^{2+\delta}\right)^\frac{2}{2+\delta} \left(1+\|\tth_{i-1}\|\right)^\frac{\delta}{2+\delta},
\end{align*}
The sum can be bounded by
\begin{align*}
\frac{1}{n}\sum_{i=1}^n\E_{i-1}\left[\|\bvarepsilon_i\|^2 \1_{\|\bvarepsilon_i\|>a\sqrt{n}}\right] &\leq C \frac{1}{n}\sum_{i=1}^n \left(\frac{1}{a\sqrt{n}}\right)^\frac{\delta}{2+\delta} \left(1+\|\tth_{i-1}\|^{2+\delta}\right)^\frac{2}{2+\delta} \left(1+\|\tth_{i-1}\|\right)^\frac{\delta}{2+\delta}\\
&\leq Cn^{-1-\frac{\delta}{2(2+\delta)}} \sum_{i=1}^n \left(1+\|\tth_{i-1}\|^{2+\delta}\right)^\frac{2}{2+\delta} \left(1+\|\tth_{i-1}\|\right)^\frac{\delta}{2+\delta}\\
&\leq Cn^{-1-\frac{\delta}{2(2+\delta)}} \sum_{i=1}^n (1+\|\tth_{i-1}\|^{2+\frac{\delta}{2+\delta}})\\
&\leq Cn^{-\frac{\delta}{2(2+\delta)}} + Cn^{-1-\frac{\delta}{2(2+\delta)}} \sum_{i=1}^n \|\tth_{i-1}\|^{2+\frac{\delta}{2+\delta}}.
\end{align*}
The first term converges to 0. The inequality~\eqref{eq:asconv} implies
\[
\sum_{i=1}^\infty \frac{\|\tth_{i-1}\|^{2+\frac{\delta}{2+\delta}}}{i^{\frac{1}{2}+\frac{\delta}{4(2+\delta)}}}< \infty.
\]
The Kronecker's lemma implies that for any $0<b_1\leq b_2 \leq \dots$ diverges to infinity,
\[
\frac{1}{b_n}\sum_{i=1}^n \frac{b_i \|\tth_{i-1}\|^{2+\frac{\delta}{2+\delta}}}{i^{\frac{1}{2}+\frac{\delta}{4(2+\delta)}}} \to 0.
\]
Take $b_i = i^{\frac{1}{2}+\frac{\delta}{4(2+\delta)}}$ and we get
\[
n^{-\frac{1}{2}-\frac{\delta}{4(2+\delta)}}\sum_{i=1}^n \|\tth_{i-1}\|^{2+\frac{\delta}{2+\delta}} \to 0.
\]
and $\frac{1}{n}\sum_{i=1}^n\E_{i-1}\left[\|\bvarepsilon_i\|^2 \1_{\|\bvarepsilon_i\|>a\sqrt{n}}\right] \to 0$ almost surely.
\end{proof}

\subsubsection*{Proof of Proposition~\ref{cor:clt0}}

\begin{proof}
For $\Qg$, let $\z \sim \mathcal{N}(\0, I_d)$, and we now calculate $\E \z \z^\top S \z \z^\top$. The $(i, i)$-th entry is
\begin{align*}
\E \sum_{j,k} z_i z_j S_{jk} z_k z_i = \sum_{j\neq i}S_{jj} + 3S_{ii} = 2S_{ii} + \tr(S).
\end{align*}
For $i\neq j$, the $(i, j)$th entry is
\begin{align*}
\E \sum_{k,l} z_i z_k S_{kl} z_l z_j = 2S_{ij}.
\end{align*}
So $\E \z \z^\top S \z \z^\top = 2S + \tr(S) I_d$.

For $\Qs$, let $\v$ be sampled from the uniform distribution on the sphere $\|\v\| = d$. The Gaussian vector $\z$ can be decomposed into independent radius part and spherical part,
\begin{align*}
\E[\z\z^\top] = \E \left[\|\z\|^2 \frac{\z}{\|\z\|} \frac{\z^\top}{\|\z\|} \right] &= \ \E \v\v^\top, \\
\E[\z\z^\top S \z\z^\top] = \E \left[\|\z\|^4 \frac{\z}{\|\z\|} \frac{\z^\top}{\|\z\|} S \frac{\z}{\|\z\|} \frac{\z^\top}{\|\z\|} \right] &= \frac{d+2}{d}\E \v\v^\top S \v \v^\top.
\end{align*}
Now we have
\begin{align*}
\E \v \v^\top  = I_d, \; \; \E \v \v^\top S \v \v^\top = \frac{d}{d+2}(2S+\tr(S)I_d).
\end{align*}
For $\Qu$, let $\u$ obey the uniform distribution on $\{\sqrt{d}e_1, \dots, \sqrt{d}e_d\}$. By direct calculation, we have
\begin{align*}
\E \u \u^\top S \u \u^\top = \sum_{j=1}^d \frac1d \cdot d ^2 S_{jj}=d \ \diag(S).
\end{align*}

The final two cases for $\Qu, \Qp$ can also be verified by direct calculation.
\end{proof}

\subsection{Illustration of choices of directions $\Pv$}
\label{subsec:suppchoice}

We first note that $\Qg\succ\Qs$ regardless of the dimension $d$ and Gram matrix $S$. Intuitively, when the direction $\v$ is generated by Gaussian \hyperref[tg]{\tg}, it can be decomposed into two independent random variables: the radical part $\| \v\|$ and the spherical part $\v / \| \v\|$. The spherical part $\v / \| \v\|$ follows the same distribution as the uniform distribution on the sphere with radius $d$ (which is identical to \hyperref[ts]{\ts}). The extra randomness in the radical part $\| \v\|^2 \sim \chi^2(d)$ leads to a larger magnitude of $Q$ compared to that of \hyperref[ts]{\ts}. Therefore the {\akw} estimator with Gaussian directions \hyperref[tg]{\tg} is always inferior to that with spherical directions \hyperref[ts]{\ts}, asymptotically. However, for the other candidates, they are not directly comparable, and the optimal choice of $\Pv$ depends on the optimality criterion, and Gram matrix $S$.

As a simple illustration, we consider $S = \diag(1, r_0)$ for some $r_0>0$. We have
\begin{enumerate}[(i)]
\item[{\hyperref[ts]{\ts}}] Spherical: $\Qs=\diag\left(\frac{r_0+3}{2},\frac{3r_0+1}{2}\right)$.
\item[{\hyperref[ti]{\ti}}] Uniform in a natural coordinate basis: $\Qi = \diag(2,2r_0)$.
\item[{\hyperref[tu]{\tu}}] Uniform in an arbitrary orthonormal basis $U$: when $U = (\cos\omega,\sin\omega; -\sin\omega, \cos\omega)$ and $\omega=0$, we have $\Qu=\Qi=\diag(2,2r_0)$; when $\omega=\pi/4$, we have $\Qu=\diag(1+r_0, 1+r_0)$.
\item[{\hyperref[tp]{\tp}}] Non-uniform in a natural coordinate basis: $\diag\left(\frac{1}{p_1},\frac{r_0}{1-p_1}\right), p_1 \in (0,1)$.
\end{enumerate}

From the above we can see that, the choices of the distribution of direction vectors $\Pv$ depends on the optimality-criteria on comparing the covariance matrices. Specifically in the above example, if one seeks to minimize
\begin{itemize}
\item the trace of covariance matrix, we have
\begin{align*}
\tr(\Qs)=\tr(\Qi)=\tr(\Qu)=2+2r_0, \; \; \tr(\Qp) = \frac{1}{p_1} + \frac{r_0}{1-p_1},
\end{align*}
and the optimal distribution that minimizes the trace depends on the value of $p_1$.
\item the determinant of covariance matrix, we have
\begin{alignat*}{2}
&\det(\Qs)=\frac{3r_0^2+10r_0+3}{4},  &&\det(\Qi)=4r_0, \\
&\det(\Qu)=\frac{-\cos(4 \omega) (r_0 - 1)^2+ r_0^2+6r_0+ 1}{2}, \;\;&&\det(\Qp) = \frac{r_0}{p_1(1-p_1)}.
\end{alignat*}
By a simple derivation, we have
$\det(\Qs)\geq \det(\Qu)\geq \det(\Qi)$ and $\det(\Qp) \geq \det(\Qi)$.

\item the operator norm of covariance matrix, i.e., the largest eigenvalue, we have
\begin{alignat*}{2}
&\lambda_{\max}(\Qs)=\frac{r_0+3}{2},&&\lambda_{\max}(\Qi)=2,\\
&\lambda_{\max}(\Qp) = \max \left\{\frac{1}{p_1}, \frac{r_0}{1-p_1}\right\}, \;\;&&\lambda_{\max}(\Qu)=r_0+1+(1-r_0)\left|\cos(2\omega)\right|.
\end{alignat*}
The smallest operator norm for $\Qp$ is given by $p_1 = \frac{1}{1+r_0}$. When $r_0\leq 1$, and $0\leq \omega\leq \pi/6$, we have $ \lambda_{\max}(\Qi)\geq\lambda_{\max}(\Qu)\geq \lambda_{\max}(\Qs) \geq \lambda_{\max}(\Qp)$. When $r_0\geq 1$, and $0\leq \omega\leq \pi/6$, we have $ \lambda_{\max}(\Qp)\geq  \lambda_{\max}(\Qs)\geq\lambda_{\max}(\Qu)\geq \lambda_{\max}(\Qi)$. For other choices of $\omega$, we can obtain a comparison analogously.
%
\end{itemize}

\begin{figure}[!t]
  \centering
  \includegraphics[width=0.6\textwidth]{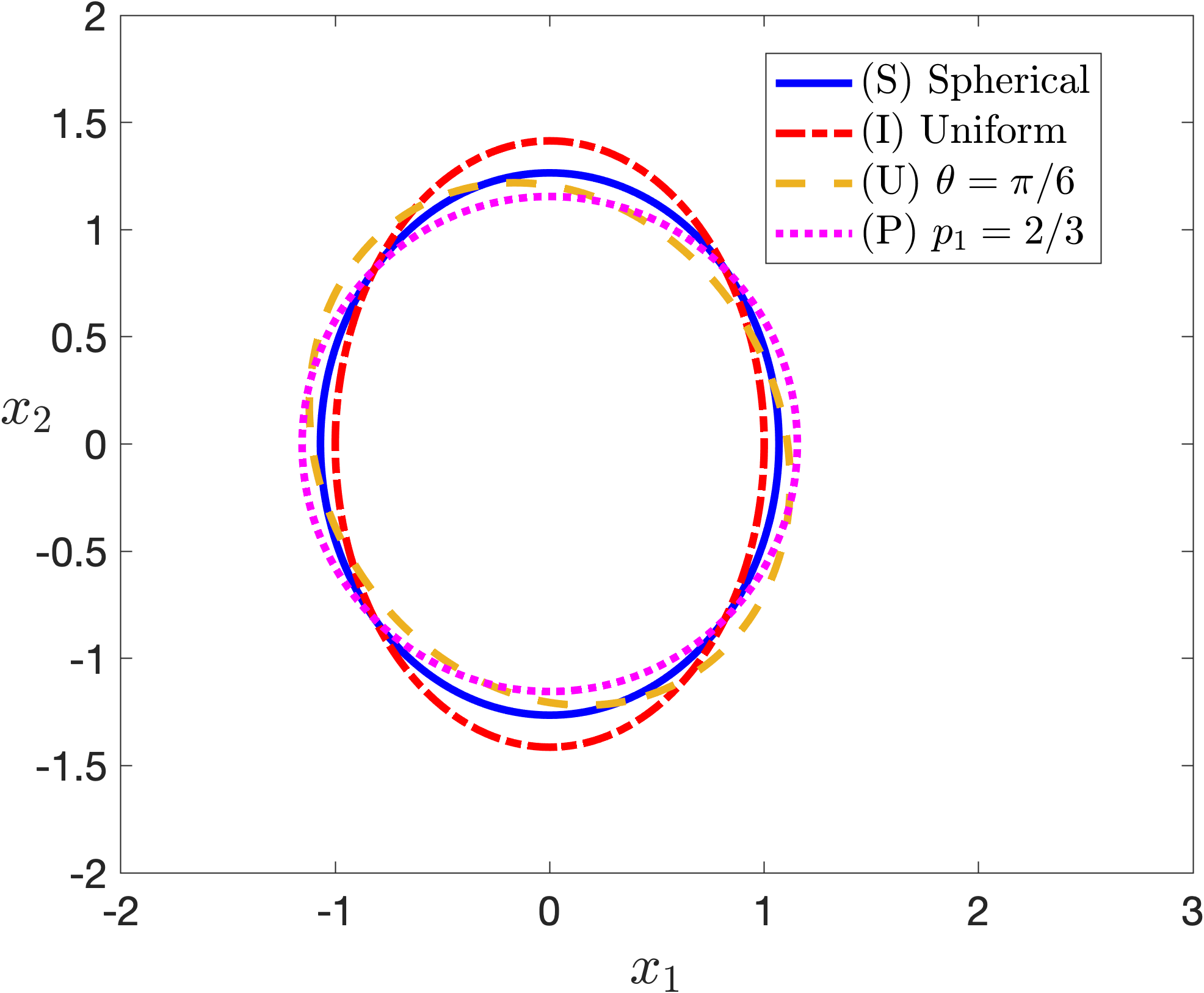}
  \caption{Comparison of $Q$ matrices under different direction distributions $\Pv$ when $S = \diag(1, 1/2)$.}
  \label{fig-1}
\end{figure}
In general, it is natural to  use Loewner order to compare two positive semi-definite matrix $A, B \in \R^{d \times d}$, i.e., $A \succeq B$ if $\x^\top A \x \geq \x^\top B \x$ for any $\x \in \R^d$. It is equivalent to say, for any positive constant $c > 0$, the ellipsoid $\{\x \in \R^d: \x^\top A \x \leq c\}$ contains the ellipsoid $\{\x \in \R^d: \x^\top B \x \leq c\}$. To better illustrate the result, we consider the 2-dimensional case where $S = \diag(1, 1/2)$ and plot the ellipse $\{\x \in \R^2: \x^\top Q^{(\cdot)} \x = 2\}$. In Figure~\ref{fig-1}, we compare $\Qs$, $\Qi$ (as a special case of $\Qu$ with $\theta=0$), $\Qu$ with $\theta= \frac{\pi}{6}$, and $\Qp$ with $p_1=\frac{1}{1+r_0} = \frac{2}{3}$. As can be inferred from the plot, none of the ellipsoids contain any other ellipsoids.

As shown in this illustrative example, there is no unique optimal direction distribution, and a practitioner might choose a search direction based on her favorable optimality criterion. 

Lastly, in the following Remark \ref{rmk:sample-with-p}, we show that, if the optimality criterion degenerates to one dimension, one may utilize the non-uniform distribution ${\hyperref[tp]{\tp}}$ to obtain a smaller limiting variance.
In particular, consider the application where we are only interested in the first coordinate of $\tth^\star$, in which cases the optimality criterion of the limiting variance is on $\theta^\star_1$. We will show that the {\akw} estimator with the non-uniform distribution $\tp$ achieves the Cram\'{e}r-Rao lower bound.

\begin{remark} \label{rmk:sample-with-p}
Assume the population loss function $F(\cdot)$ has Hessian $H = I_d$. Considering a non-uniform sampling $\tp$ from $\{\e_k\}_{k=1}^d$ for the direction distribution $\Pv$. We choose $\v=\e_k$ with probability $p_k$ for $k=1,2,\dots, d$, where $p_1=1-p$ for some constant $p\in(0,1]$ and $p_k=p/(d-1)$ for $k\neq 1$. Define i.i.d. random variables $k_n$ where $k_n = 1$ with probability $1-p$ and $k_n = 2,\ldots,d$ uniformly with probability $p/(d-1)$. The gradient estimator is defined by,
\begin{align*}
\widehat{g}(\tth_{n-1}; \zet_n) = \frac{f(\tth_{n-1}+h_n\e_{k_n}; \zet_n)-f(\tth_{n-1};\zet_n)}{h_n p_n} \e_{k_n},
\end{align*}
where $p_n = 1-p$ if $k_n = 1$, $p_n = p/(d-1)$ for $k_n > 1$.  By the same argument as Proposition~\ref{cor:clt0}, the variance for $\overline{\tth}_n$ in the direction $\e_1$ is,
\begin{align*}
n \mathrm{Var}\left(\e_1^\top (\overline{\tth}_n - \tth^\star) \right) = \frac{S_{11}}{1-p}.
\end{align*}
As $p\rightarrow 0$, we approximately obtain the optimal variance given by Cram\'{e}r-Rao lower bound in the direction $\e_1$. However, in order to approach the optimal variance in the direction $\e_1$, we increase the magnitude of variance in all other directions, where the variance in other directions is given by $n \mathrm{Var}\left(\e_k^\top (\overline{\tth}_n - \tth^\star) \right) = (d-1)S_{kk}/p$ for $k=2,\ldots, d$.
\end{remark}

\subsection{Nonsmooth loss function}
\label{subsec:quantile}

\subsubsection*{Proof of Theorem \ref{thm:clt-qt}}
\begin{proof} 

We first prove that Theorem \ref{thm:clt} is valid under Assumptions \ref{assumption1p}, \ref{assumption2}, \ref{assumption4}, and \ref{assumption5} below.

\begin{assumption}\label{assumption5}
There exist $C>0, L_h>0, \delta_2 > 0, Q$ such that for any $\tth$ in the $\delta_2$-ball centered at $\tth^\star$, and for $h$ sufficiently small,
\[
\left\|\E\left[\widehat g_{h,\v}(\tth;\zet) \widehat g_{h,\v}(\tth;\zet) ^\top \right] - Q\right\|\leq C(h+\|\tth - \tth^\star\|),\qquad
\|\nabla^2 F(\tth) - H \| \leq L_h \|\tth-\tth^\star\|,
\]
where $H$ is the Hessian matrix of the population loss function $F(\cdot)$, i.e.,
$H=\nabla^2 F(\tth^\star)$.
\end{assumption}

Recall that, Proposition \ref{lemma:x} holds under Assumptions \ref{assumption1p}, \ref{assumption2}, \ref{assumption4}. The only two places that Assumption \ref{assumption3} is used in the proof of Theorem \ref{thm:clt} are \eqref{eq:use3} and \eqref{expectationoveru}.

The first part of assumption~\ref{assumption5} is a slightly different version of Lemma \ref{lem:variance} with a matching leading term. It implies that
\begin{align*}
&\quad\|\E_{n-1} \bvarepsilon_{n} \bvarepsilon_n^\top -  Q \|\\
 &= \|\E_{n-1}\left[\widehat g_{h_n,\v}(\tth_{n-1};\zet_n) \widehat g_{h_n,\v}(\tth_{n-1};\zet_n) ^\top \right]  -(\frac{1}{h_n}\Delta_{h_n, \v}F(\tth_{n-1}) \v) (\frac{1}{h_n}\Delta_{h_n, \v}F(\tth_{n-1}) \v)^\top -  Q \|\\
&\leq \left\|\E\left[\widehat g_{h_n,\v}(\tth_{n-1};\zet_n) \widehat g_{h_n,\v}(\tth_{n-1};\zet_n) ^\top \right] - Q\right\| + \|\E_{n-1}(\frac{1}{h_n}\Delta_{h_n, \v}F(\tth_{n-1}) \v) (\frac{1}{h_n}\Delta_{h_n, \v}F(\tth_{n-1}) \v)^\top\|\\
&\leq C(\|\tth_{n-1}\| + h_n) + C(\|\tth_{n-1}\|^2 + h_n^2).
\end{align*}
So \eqref{expectationoveru} is still true with the same leading term. The second part of assumption~\ref{assumption5} implies \eqref{eq:use3} by Taylor expansion. So Assumption~\ref{assumption5} can replace Assumption~\ref{assumption3} in the previous proof.

It finally suffices to show that Assumption \ref{assumption6} implies Assumptions \ref{assumption1p}, \ref{assumption2} and \ref{assumption5}, where $\nabla f(\tth;(\x,y))$ is defined as a subgradient of the loss function $f(\tth;\zeta)=\rho(y-\x^\top\tth)$, particularly, $\x \psi(y-\x^\top \tth)$. Notice that
\begin{align*}
\E\left[ \widehat g_{h,\v}(\tth;\zet) \widehat g_{h,\v}(\tth;\zet) ^\top \right]  &= \E \frac{\v\v^\top}{h^2}\left[\rho\big(y-\x^\top (\tth+h\v) \big)-\rho\big( y-\x^\top \tth \big)\right]^2\\
&= \E \frac{\v\v^\top}{h^2}\left[\rho\big(\varepsilon-\x^\top (\tth+h\v) \big)-\rho\big( \varepsilon-\x^\top \tth \big)\right]^2.
\end{align*}
Define the function $D(h, \tth):= \E \v\v^\top \left[\rho\big(\varepsilon-\x^\top (\tth+h\v) \big)-\rho\big( \varepsilon-\x^\top \tth \big)\right]^2$. Direct computation shows that
\[
D(0, 0) = 0, \frac{\partial}{\partial h}D(0, 0) = 0, \frac{\partial^2}{\partial h^2}D(0, 0) = 2\E[\psi^2(\varepsilon) \v\v^\top\x\x^\top\v\v^\top] = 2Q,
\]
Because $\varepsilon$ has $C^3$ density,  $\frac{\partial^2}{\partial h^2}D(h, \tth)$ is actually differentiable (jointly with respect to both variables). By Taylor expansion, for sufficiently small $h$ and $\tth$, $\|\frac{\partial^2}{\partial h^2}D(h, \tth)-2Q\| \leq C(\|\tth\| + h)$, which implies
\[
\| \E\left[ \widehat g_{h,\v}(\tth;\zet) \widehat g_{h,\v}(\tth;\zet) ^\top \right] - Q\| = \|D(h, \tth)/h^2 - Q\|\leq C(\|\tth\| + h),
\]
for sufficiently small $h$ and $\tth$.

Furthermore, 
\begin{align*}
\nabla^2 F(\tth) &= \nabla^2 \E [\rho(y-\x^\top \tth)]\\
&= \nabla \E [\x \psi(\x^\top (\tth^\star - \tth)+\varepsilon)]\\
&= \nabla \E [\x \phi(\x^\top (\tth^\star - \tth))]\\
&= \E [\x \x^\top \phi'(\x^\top (\tth^\star - \tth))].
\end{align*}
Since $\rho$ is convex, $\psi, \phi$ are monotone. Therefore $\phi' \geq0$ and $\nabla^2 F(\tth) \geq 0$. Since $\nabla^2 F(\tth^\star) = H \succ 0$, and $\phi'$ is upper bounded, Assumption~\ref{assumption1p} holds.  Next,
\begin{align*}
\|\nabla^2 F(\tth) - H\| &= \|\E [\x \x^\top (\phi'(\x^\top (\tth^\star - \tth)) - \phi'(0))]\|\\
&\leq C \|\E [\x \x^\top |\x^\top (\tth^\star - \tth)|]\|\\
&\leq C \|(\tth^\star - \tth)\|.
\end{align*}
Thus Assumption~\ref{assumption5} holds. Lastly, \begin{align*}
&\quad \E \|\x \psi(y-\x^\top \tth) - \nabla F(\tth)\|^4\\
&\leq C\E \|\x \psi(y-\x^\top \tth)\|^4\\
&\leq C \E\|\x \psi(\x^\top (\tth^\star - \tth)+\varepsilon)\|^4\\
&\leq C \E\|\x (|\x^\top (\tth^\star - \tth)+\varepsilon| +1)\|^4\\
&\leq C (\|\theta^\star - \theta\|^4 +1).
\end{align*}
So Assumption~\ref{assumption2} holds with $\delta = 2$.
\end{proof}

\subsection{Multi-query approximation}
\label{subsec:multi-query}

\subsubsection*{Proof of Theorem~\ref{thm:multiple_1}}
\begin{proof}
The convergence result can be obtained as in the two function evaluation case. The only difference is the following calculation:
\begin{align*}
\E \left(\frac{1}{m}\sum_{i=1}^m \v_i \v_i^\top\right)S\left(\frac{1}{m}\sum_{i=1}^m \v_i \v_i^\top\right) = \frac{1}{m}\E \v \v^\top S \v \v^\top +\frac{m-1}{m}S,
\end{align*}
which implies the desired result.
\end{proof}

\subsubsection*{Proof of Theorem~\ref{thm:multiple_2}}
\begin{proof}
It is clear that $Q_m = S$ for $m =d$. We need to compute the quantity
\begin{align*}
Q_m = \frac{d^2}{m^2}\E \bigg(\sum_{i=1}^m \v_i \v_i^\top\bigg) S \bigg(\sum_{i=1}^m \v_i \v_i^\top\bigg),
\end{align*}
which can be simplifies to
\[
Q_m = \frac{d^2}{m^2}\E \bigg(\sum_{i=1}^m \v_i \v_i^\top S  \v_i \v_i^\top\bigg) + \frac{d^2}{m^2}\E \bigg(\sum_{i\neq j} \v_i \v_i^\top S  \v_j \v_j^\top\bigg).
\]
By symmetry, it equals to
\[
Q_m = \frac{d^2}{m}\E \v_1 \v_1^\top S  \v_1 \v_1^\top + \frac{d^2(m-1)}{m}\E \v_1 \v_1^\top S  \v_2 \v_2^\top.
\]

We know $\E \v_1 \v_1^\top S  \v_1 \v_1^\top = \frac{1}{d^2} Q$ and $Q_d = S$. So we can solve for $\E \v_1 \v_1^\top S  \v_2 \v_2^\top$ and get
\[
\E \v_1 \v_1^\top S  \v_2 \v_2^\top = \frac{1}{d(d-1)}(\frac{1}{d}Q - \diag(S)).
\]

Therefore,
\begin{align*}
Q_m &= \frac{1}{m}Q + \frac{d(m-1)}{m(d-1)}(\frac{1}{d}Q-\diag S)\\
&= \frac{d-m}{m(d-1)}Q+ \frac{d(m-1)}{m(d-1)}S. \qedhere
\end{align*}
\end{proof}


\section{Proofs of Results in Section~\ref{sec:infer}}
\label{sec:app-c}

\subsection{Proof of Lemma~\ref{lemma:hessian}}
\label{sec:G_con}
Before we come to the proof of the Hessian estimator~\eqref{eq:hessian-est2} in Lemma~\ref{lemma:hessian}, we first introduce a naive method to estimate Hessian matrix $H$ which we omit in the main text.

Inspired by the previous gradient estimator, we can estimate the Hessian matrix $H$ by the following
\begin{align*}
\widehat{G}_n
= \frac{1}{m h_n^2} \sum_{j=1}^{m} \left[\Delta_{h_n \v_n^{(j)}} f(\tth_{n-1} + h_n \u_n^{(j)}; \zet_n) - \Delta_{h_n \v_n^{(j)}} f(\tth_{n-1}; \zet_n)\right]\u_n^{(j)} \v_n^{(j)\top},
\end{align*}
where $\{\u_n^{(j)}\}_{j=1}^m$ and $\{\v_n^{(j)}\}_{j=1}^m$ are \emph{i.i.d.} random vectors and $m > 0$ is a parameter (which might be different from $m$ in the previous section). Therefore, our naive Hessian estimator is,
\begin{align}
\label{eq:hessian-est1}
\widetilde{H}_n = \frac{1}{n} \sum_{i=1}^n \frac{\widehat{G}_i+\widehat{G}_i^\top}{2}.
\end{align}
where the $(\widehat{G}_i + \widehat{G}_i^\top)/2 \,$ term ensures the symmetry of $\widetilde{H}_n$. The function query complexity is $\O(m)$ per step for this Hessian estimation.

Now we restate our Lemma~\ref{lemma:hessian} for the both estimators \eqref{eq:hessian-est1} and \eqref{eq:hessian-est2}.
\begin{lemma}
\label{lemma:hessian2}
If Assumption \ref{assumption1p}, \ref{assumption2}, \ref{assumption3}, \ref{assumption4} hold, or Assumption \ref{assumption4}, \ref{assumption6} holds, then $\widetilde{H}_n$ converges in probability to $H$. 

If Assumption \ref{assumption1p}, \ref{assumption2} hold with $\delta_1=+\infty$ and $\delta_2$, then we have the following result for the Hessian estimator~\eqref{eq:hessian-est1},
\begin{align}
\E \| \widetilde{H}_n - H \|^2 \leq C_1 n^{-\alpha} + C_2\left(1+\frac{1}{m}\right)n^{-1}.
\end{align}
The Hessian estimator~\eqref{eq:hessian-est2} satisfies,
\begin{align}
\E \| \widetilde{H}_n - H \|^2 \leq C_1 n^{-\alpha} + C_2 p^{-1}n^{-1}.
\end{align}
\end{lemma}
\begin{proof}
We first assume Assumption \ref{assumption1p}, \ref{assumption2}, \ref{assumption3}, and \ref{assumption4} hold. In the case of naive Hessian estimator \eqref{eq:hessian-est1}, we decompose $\widetilde{H}_n - H$ for estimator~\eqref{eq:hessian-est1} as follows,
\begin{align}
\label{eq:hessian-decomp}
\widetilde{H}_n - H &= \frac{1}{n} \sum_{i=1}^n\frac{\widehat{G}_i+\widehat{G}_i^\top}{2} - H \notag \\
&= \frac{1}{n} \sum_{i=1}^n \left(\frac{\widehat{G}_i+\widehat{G}_i^\top}{2} - \left(\frac{1}{m}\sum_{j=1}^m\u_i^{(j)}\u_i^{(j)\top}\right) \nabla^2 f(\tth_{n-1}; \zet_n) \left(\frac{1}{m}\sum_{j=1}^m\v_i^{(j)}\v_i^{(j)\top}\right) \right) \notag \\
&\quad + \frac{1}{n} \sum_{i=1}^n \left(\left(\frac{1}{m}\sum_{j=1}^m\u_i^{(j)}\u_i^{(j)\top}\right) \nabla^2 f(\tth_{n-1}; \zet_n) \left(\frac{1}{m}\sum_{j=1}^m\v_i^{(j)}\v_i^{(j)\top}\right)- \nabla^2 f(\tth_{i-1}; \zet_i)\right) \notag\\
&\quad + \frac{1}{n} \sum_{i=1}^n \left[\nabla^2 f(\tth_{i-1}; \zet_i) - \nabla^2 f(\0; \zet_i)\right] + \frac{1}{n} \sum_{i=1}^n \left(\nabla^2 f(\0; \zet_i) - H\right).
\end{align}
For the first term in the decomposition~\eqref{eq:hessian-decomp},
\begin{align*}
&\quad \E_{n-1} \left[\| \frac{1}{h_n^2}\big[ f(\tth_{n-1} + h_n \u+ h_n \v; \zet_n) -  f(\tth_{n-1} + h_n \u; \zet_n)- f(\tth_{n-1} + h_n \v; \zet_n) \right.\\
&\quad + \left. \left. f(\tth_{n-1}; \zet_n) \big] \u\v^\top- \u\u^\top \nabla^2 f(\tth_{n-1}; \zet_n) \v \v^\top \right\|^2 \bigg| \u, \v\right]\\
&\leq \E_{n-1} \left[ \left\| \frac{1}{h_n^2} \u\u^\top \int_0^{h_n} \int_0^{h_n} \nabla^2f(\tth_{n-1}+s_1 \u +s_2 \v ;\zet_n) - \nabla^2 f(\tth_{n-1}; \zet_n) \intd s_1 \intd s_2 \v\v^\top \right\|^2 \bigg| \u, \v\right]\\
&\leq \frac{1}{h_n^2} \|\u\|^2\|\v\|^2\int_0^{h_n} \int_0^{h_n}\E_{n-1}\left[ \left\|\nabla^2f(\tth_{n-1}+s_1 \u +s_2 \v ;\zet_n) - \nabla^2 f(\tth_{n-1}; \zet_n) \right\|^2 \big| \u, \v\right] \intd s_1 \intd s_2 \\
&\leq \frac{C}{h_n^2} \|\u\|^2\|\v\|^2 \int_0^{h_n} \int_0^{h_n} \left\|s_1 \u +s_2 \v\right\|^2  \intd s_1 \intd s_2 \leq Ch_n^2 \|\u\|^2\|\v\|^2 (\|\u\|^2+\|\v\|^2).
\end{align*}
The above derivation implies that
\begin{align*}
\E \|\widehat{G}_n - \left(\frac{1}{m}\sum_{j=1}^m\u_i^{(j)}\u_i^{(j)\top}\right) \nabla^2 f(\tth_{n-1}; \zet_n) \left(\frac{1}{m}\sum_{j=1}^m\v_i^{(j)}\v_i^{(j)\top}\right)\| \leq Ch_n^2.
\end{align*}
Therefore, we can show that
\begin{align}
\label{hessian-est1}
& \quad \E \left\| \frac{1}{n} \sum_{i=1}^n \left(\frac{\widehat{G}_i+\widehat{G}_i^\top}{2} - \left(\frac{1}{m}\sum_{j=1}^m\u_i^{(j)}\u_i^{(j)\top}\right) \nabla^2 f(\tth_{i-1}; \zet_i) \left(\frac{1}{m}\sum_{j=1}^m\v_i^{(j)}\v_i^{(j)\top}\right) \right) \right\|^2 \notag\\
&\leq \E \left\| \frac{1}{n} \sum_{i=1}^n \left(\widehat{G}_i - \left(\frac{1}{m}\sum_{j=1}^m\u_i^{(j)}\u_i^{(j)\top}\right)  \nabla^2 f(\tth_{i-1}; \zet_i) \left(\frac{1}{m}\sum_{j=1}^m\v_i^{(j)}\v_i^{(j)\top}\right) \right) \right\|^2 \notag \\
&\leq C\frac{1}{n}\sum_{i=1}^n h_i^2 \leq Cn^{-2\gamma},
\end{align}
where in the first inequality, we use the fact that, $\widehat{G}_i$ and $\widehat{G}_i^\top$ has the same distribution.

For the second term, notice that
\begin{align*}
&\quad \E_{n-1} \left\| \left(\frac{1}{m}\sum_{j=1}^m\u_j\u_j^\top\right) \nabla^2 f(\tth_{n-1}; \zet_n) \left(\frac{1}{m}\sum_{j=1}^m\v_j\v_j^\top\right) - \nabla^2 f(\tth_{n-1}; \zet_n) \right\|^2\\
&\leq \E_{n-1} \left\|\frac{1}{m}\u_i\u_i^\top -I_d\right\|^2 \left\|\nabla^2 f(\tth_{n-1}; \zet_n)\right\|^2 \left\|\frac{1}{m}\v \v^\top - I_d\right\|^2 \\
& \quad + \E_{n-1} \left\|\frac{1}{m}\u_i\u_i^\top -I_d\right\|^2 \left\|\nabla^2 f(\tth_{n-1}; \zet_n)\right\|^2 + \E_{n-1} \left\|\nabla^2 f(\tth_{n-1}; \zet_n)\right\|^2 \left\|\frac{1}{m}\v \v^\top - I_d\right\|^2 \\
&\leq \frac{C}{m}\left(1+\|\tth_{n-1}\|^2\right).
\end{align*}
Furthermore, the second term is a sum of martingale difference sequence and we have
\begin{align}
\label{hessian-est2}
& \quad \E \| \frac{1}{n} \sum_{i=1}^n \left(\left(\frac{1}{m}\sum_{j=1}^m\u_i^{(j)}\u_i^{(j)\top}\right) \nabla^2 f(\tth_{n-1}; \zet_n) \left(\frac{1}{m}\sum_{j=1}^m\v_i^{(j)}\v_i^{(j)\top}\right)- \nabla^2 f(\tth_{i-1}; \zet_i)\right) \|^2 \notag\\
&= \frac{1}{n} \sum_{i=1}^n \E \|\left(\left(\frac{1}{m}\sum_{j=1}^m\u_i^{(j)}\u_i^{(j)\top}\right) \nabla^2 f(\tth_{n-1}; \zet_n) \left(\frac{1}{m}\sum_{j=1}^m\v_i^{(j)}\v_i^{(j)\top}\right)- \nabla^2 f(\tth_{i-1}; \zet_i)\right)\|^2 \notag \\
&\leq C\frac{1}{n^2}\sum_{i=1}^n \frac{1}{m}\left(1+\E\|\tth_{n-1}\|^2\right) \leq C\frac{1}{mn}.
\end{align}
For the third term in \eqref{eq:hessian-decomp}, we have
\begin{align}
\label{hessian-est3}
\E  \left\| \frac{1}{n} \sum_{i=1}^n \nabla^2 f(\tth_{i-1}; \zet_i) - \nabla^2 f(\0; \zet_i)\right\|^2 &\leq \frac{1}{n} \sum_{i=1}^n \E \left\| \nabla^2 f(\tth_{i-1}; \zet_i) - \nabla^2 f(\0; \zet_i) \right\|^2 \notag \\
&\leq \frac{C}{n} \sum_{i=1}^n \E \| \tth_i \|^2 \leq C n^{-\alpha}.
\end{align}
For the final term, we have
\begin{align}
\label{hessian-est4}
\E \left\|\frac{1}{n} \sum_{i=1}^n \nabla^2 f(\0; \zet_i) - H \right\|^2 &\leq  \frac{1}{n^2} \sum_{i=1}^n\E \left\| \nabla^2 f(\0; \zet_i) - H \right\|^2 \notag\\
& \leq  \frac{C}{n^2} \sum_{i=1}^n\E \left\| \nabla^2 f(\0; \zet_i)^2 - H^2 \right\| \leq C n^{-1},
\end{align}
where the second inequality is due to the fact that it is an equality in Frobenius norm.

Combine the previous estimates~\eqref{hessian-est1}, \eqref{hessian-est2}, \eqref{hessian-est3} and \eqref{hessian-est4}, our naive Hessian estimator satisfies,
\begin{align*}
\E \left\| \widetilde{H}_n- H \right\|^2 \leq Cn^{-\alpha}+ C(1+\frac{1}{m})n^{-1}.
\end{align*}

For the Hessian estimator~\eqref{eq:hessian-est2}, we need to distinguish the original version, denoted by $\hat{G}$, and the bernoulli sampling version, denoted by $\tilde{G}$. We then have the following decomposition,
\begin{align}
\label{eq:hessian-decomp2}
\widetilde{H}_n - H = &\frac{1}{n} \sum_{i=1}^n \frac{\widetilde{G}_i+\widetilde{G}_i^\top}{2} - H \notag \\
= &\frac{1}{n} \sum_{i=1}^n \frac{\widetilde{G}_i+\widetilde{G}_i^\top}{2} - \frac{\widehat{G}_i+\widehat{G}_i^\top}{2} + \frac{1}{n} \sum_{i=1}^n \left(\frac{\widehat{G}_i+\widehat{G}_i^\top}{2} - \nabla^2 f(\tth_{i-1}; \zet_i)\right) \notag \\
&+ \frac{1}{n} \sum_{i=1}^n \left[\nabla^2 f(\tth_{i-1}; \zet_i) - \nabla^2 f(\0; \zet_i)\right] + \frac{1}{n} \sum_{i=1}^n \nabla^2 f(\0; \zet_i) - H.
\end{align}
Given $\widehat{G}_n$, our Bernoulli sampling Hessian estimator $\widetilde{G}_n$ satisfies,
\begin{align*}
\E \left\|\widetilde{G}_n - \widehat{G}_n \right\|^2_{\Fro}
&= \E \left[\sum_{j=1}^d \sum_{k=1}^d \frac{1}{p} \left( \widehat{G}_n^{(jk)} B_n^{(jk)} - \widehat{G}_n^{(jk)}\right)^2\right] \\
&= \sum_{j=1}^d \sum_{k=1}^d   \E \left(\frac{1}{p} B_n^{(jk)} - 1 \right)^2 \left(\widehat{G}_n^{(jk)}\right)^2\\
&= \frac{1-p}{p} \sum_{j=1}^d \sum_{k=1}^d \E\left(\widehat{G}_n^{(jk)}\right)^2 =  \frac{1-p}{p} \| \widehat{G}_n \|_{\Fro}^2,
\end{align*}
where the entries of $B_n$ are \emph{i.i.d.} and follow a Bernoulli distribution, i.e., $B_n^{(k \ell)} \sim \mathrm{Bernoulli}(p)$, for some fixed $p \in (0,1)$. Here the second equality uses the fact that $B_i^{(jk)}$ are independent from each other. Therefore,
\begin{align*}
\E \left\| \frac{1}{n} \sum_{i=1}^n \widetilde{G}_i - \widehat{G}_i \right\|^2 \leq \E \left\| \frac{1}{n} \sum_{i=1}^n \widetilde{G}_i - \widehat{G}_i \right\|^2_{\Fro} \leq C \frac{1-p}{p} \, n^{-2} \sum_{i=1}^n \E \left\|\widehat{G}_i \right\|^2.
\end{align*}
With $1/t \sum_{i=1}^n\E \|\widehat{G}_i \|^2 \leq C + Cn^{-\alpha}$, the first term in decomposition~\eqref{eq:hessian-decomp2} satisfies,
\begin{align}
\label{hessian-est5}
\E \left\| \frac{1}{n} \sum_{i=1}^n \widetilde{G}_i - \widehat{G}_i \right\|^2 \leq C\frac{1-p}{p} n^{-1}.
\end{align}
Other terms can be bounded similarly as in the first case:
\begin{align}
\label{hessian-est6}
\E \left\| \frac{1}{n} \sum_{i=1}^n \frac{\widehat{G}_i+\widehat{G}_i^\top}{2} - \nabla^2 f(\tth_{i-1}; \zet_i) \right\|^2 &\leq C n^{-2\gamma}, \\
\label{hessian-est7}
\E  \left\| \frac{1}{n} \sum_{i=1}^n \nabla^2 f(\tth_{i-1}; \zet_i) - \nabla^2 f(\0; \zet_i)\right\|^2 &\leq C n^{-\alpha}, \\
\label{hessian-est8}
\E \left\|\frac{1}{n} \sum_{i=1}^n \nabla^2 f(\0; \zet_i) - H \right\|^2 &\leq C n^{-1}.
\end{align}
Combine inequality~\eqref{hessian-est5}, \eqref{hessian-est6}, \eqref{hessian-est7} and \eqref{hessian-est8}, we obtain the desired result for Hessian estimator~\eqref{eq:hessian-est2}.

Now we assume Assumption \ref{assumption4}, \ref{assumption6} hold. For estimator~\eqref{eq:hessian-est1}, we have
\begin{align*}
\widetilde{H}_n - H &= \frac{1}{n} \sum_{i=1}^n\frac{\widehat{G}_i+\widehat{G}_i^\top}{2} - H  \\
&= \frac{1}{n} \sum_{i=1}^n(\frac{\widehat{G}_i+\widehat{G}_i^\top}{2} - \E\frac{\widehat{G}_i+\widehat{G}_i^\top}{2})\\
&\quad + \frac{1}{n} \sum_{i=1}^n (\E\frac{\widehat{G}_i+\widehat{G}_i^\top}{2} - \nabla^2F(\tth_{i-1}))\\
&\quad + \frac{1}{n} \sum_{i=1}^n (\nabla^2F(\tth_{i-1}) - \nabla^2F(\tth))).
\end{align*}
In the case of quantile regression, $\widehat{G}$ has the form
\begin{align*}
&\quad \frac{1}{h^2}\big[ f(\tth + h \u+ h \v; \zet) -  f(\tth + h \u; \zet)- f(\tth + h \v; \zet) + f(\tth; \zet) \big]\u \v^\top \\
&= \frac{1}{h^2}\big[ \rho(y-\x^\top (\tth + h \u+ h \v)) - \rho(y-\x^\top (\tth + h \u)) - \rho(y-\x^\top (\tth + h \v)) + \rho(y-\x^\top \tth)\big]\u \v^\top.
\end{align*}
Let $D(h) = \E [\rho(y-\x^\top (\tth + h \u+ h \v)) - \rho(y-\x^\top (\tth + h \u)) - \rho(y-\x^\top (\tth + h \v)) + \rho(y-\x^\top \tth)]\u \v^\top$. Direct computation shows that $D(0) = 0, D'(0) = 0, D''(0) = 2 \E \x \x^\top \phi'(\x^\top(\tth^\star-\tth)) = 2\nabla^2 F(\tth), \|D''(h)-D''(0)\| \leq Ch$. So we have $\|D(h) - \nabla^2 F(\tth) h^2 \| \leq Ch^3$. This imples
\[
\|\E \hat{G}_n - \nabla^2 F(\tth_{n-1})\| \leq Ch_n,
\]
which converges to 0. So the second term
\[
\|\frac{1}{n} \sum_{i=1}^n (\E\frac{\widehat{G}_i+\widehat{G}_i^\top}{2} - \nabla^2(\tth_{i-1}))\| \leq \frac{C}{n}\sum_{i=1}^n h_i \to 0.
\]
The third term clearly converges to 0, so the only thing left is the first term.
\begin{align*}
&\quad \E\|\frac{1}{n} \sum_{i=1}^n(\frac{\widehat{G}_i+\widehat{G}_i^\top}{2} - \E\frac{\widehat{G}_i+\widehat{G}_i^\top}{2})\|^2\\
&\leq \frac{1}{n^2} \sum_{i=1}^n \E\|\frac{\widehat{G}_i+\widehat{G}_i^\top}{2} - \E\frac{\widehat{G}_i+\widehat{G}_i^\top}{2})\|^2\\
&\leq \frac{1}{n^2} \sum_{i=1}^n \E\|\widehat{G}_i\|^2.
\end{align*}
Because $\varepsilon$ has $C^3$ density, the function $\phi(u) = \E[\psi(u+\varepsilon)]$ is actually also $C^3$. This means that it is possible to take forth derivative of the function
\[
\E\|[\rho(y-\x^\top (\tth + h \u+ h \v)) - \rho(y-\x^\top (\tth + h \u)) - \rho(y-\x^\top (\tth + h \v)) + \rho(y-\x^\top \tth)]\u \v^\top\|^2,
\]
with respect to $h$. After a routine computation, we can find that it is the form $Ch^4 + O(h^5)$ for $h$ sufficiently small, which means that
\[
\E\|\widehat{G}_i\|^2 \leq C,
\]
so the first term converges to 0. The computation for estimator~\eqref{eq:hessian-est2} is similar.
\end{proof}

\subsubsection*{Proof of Theorem~\ref{thm:plugin}}

To prove Theorem~\ref{thm:plugin}, we first present the following lemma on the error rate of $\widehat{Q}_n$.
\begin{lemma}
\label{lemma:Gram}
If Assumption \ref{assumption1p}, \ref{assumption2}, \ref{assumption3}, \ref{assumption4} hold, or Assumption \ref{assumption4}, \ref{assumption6} holds, then $\widehat{Q}_n$ converges in probability to $Q$. If Assumption \ref{assumption1p}, \ref{assumption2} hold with $\delta_1=+\infty$ and $\delta=2$, we have $\widehat{Q}_n$ has the following convergence rate,
$\E\|\widehat{Q}_n-Q\| \leq Cn^{-\alpha/2}.
$\end{lemma}

\begin{proof}
First Assume Assumption \ref{assumption1p}, \ref{assumption2}, \ref{assumption3}, and \ref{assumption4} hold with $\delta_1=\infty$. Recall the update rule,
\begin{align*}
\tth_{n}= \tth_{n-1}-\eta_n\nabla F(\tth_{n-1}) + \eta_n(\bxi_n + \bgamma_n + \bvarepsilon_n),
\end{align*}
and our Gram matrix estimate $\widehat{Q}_n$ is,
\begin{align*}
\widehat{Q}_n = \frac{1}{n}\sum_{i=1}^n (\nabla F(\tth_{i-1}) -\bxi_i - \bgamma_i - \bvarepsilon_i) (\nabla F(\tth_{i-1}) -\bxi_i - \bgamma_i - \bvarepsilon_i)^\top.
\end{align*}
It can be seen that we have the following estimates,
\begin{align*}
\E \left\|\frac{1}{n}\sum_{i=1}^n \nabla F(\tth_{i-1})\nabla F(\tth_{i-1})^\top \right\| &\leq C\frac{1}{n}\sum_{i=1}^n \E\|\tth_{i-1}\|^2 \leq Cn^{-\alpha}, \\
\E \left\|\frac{1}{n}\sum_{i=1}^n \bxi_i \bxi_i^\top \right\| &\leq C \frac{1}{n}\sum_{i=1}^n h_n^2 \leq Cn^{-2\gamma}, \\
\E \left\|\frac{1}{n}\sum_{i=1}^n \bgamma_i \bgamma_i^\top \right\| &\leq C \frac{1}{n}\sum_{i=1}^n (\E\|\tth_{i-1}\|^2+h_n^2) \leq Cn^{-\alpha}, \\
\E \left\|\frac{1}{n}\sum_{i=1}^n \bvarepsilon_i \bvarepsilon_i^\top \right\| &\leq C \frac{1}{n}\sum_{i=1}^n (\E\|\tth_{i-1}\|^2+h_n^2 +1) \leq C.
\end{align*}
The crossing terms between them can be bounded by Cauchy-Schwarz inequality. Therefore, we can find that all terms in $\widehat{Q}_n$ except $\sum_{i=1}^n \bvarepsilon_i \bvarepsilon_i^\top / t$ can be bounded by $C n^{-\alpha/2}$. So it suffices to prove,
\begin{align}
\label{eq:Gram-est1}
\E \left\|\frac{1}{n}\sum_{i=1}^n \bvarepsilon_i \bvarepsilon_i^\top-Q \right\| \leq Cn^{-\alpha/2}.
\end{align}
Define a new sequence $z_n :=\bvarepsilon_n \bvarepsilon_n^\top -\E_{n-1}\bvarepsilon_n \bvarepsilon_n^\top$. Then $z_n$ is a martingale difference sequence and we have
\begin{align*}
\left\|\bvarepsilon_n \bvarepsilon_n^\top -Q \right\| &\leq \|z_n\| + \left\|\E_{n-1}\bvarepsilon_n \bvarepsilon_n^\top -Q \right\| \\
&\leq   \|z_n\| + C \left(\|\tth_{n-1}\|+\|\tth_{n-1}\|^3+h_n +h_n^3 \right),
\end{align*}
where the last inequality leverages inequality~\eqref{expectationoveru}. Now we have,
\begin{align*}
\E \left\|\frac{1}{n}\sum_{i=1}^n \bvarepsilon_i \bvarepsilon_i^\top-Q \right\| &\leq \E \left\|\frac{1}{n}\sum_{i=1}^n z_i \right\| + C \E \left(\|\tth_{n-1}\|+\|\tth_{n-1}\|^3+h_n +h_n^3 \right)\\
& \leq \E \left\|\frac{1}{n}\sum_{i=1}^n z_i\right\|  + C n^{-\alpha/2}.
\end{align*}
Thus we turn the proof of \eqref{eq:Gram-est1} into,
\begin{align}
\label{eq:Gram-est2}
\E \left\|\frac{1}{n}\sum_{i=1}^n z_i \right\| \leq Cn^{-1/2}.
\end{align}
By H\"older's inequality, it can be derived that,
\begin{align*}
\E_{n-1} \| z_n \|^{2}  \leq \E_{n-1} \| \bvarepsilon_n \|^{4} \leq C(\|\tth_{n-1}\|^{4}+h_n^{4} +1).
\end{align*}
Combine Lemma~\ref{lemmamoments} with Lemma~\ref{lemma:x}, we have
\begin{align*}
\E \left\| \frac{1}{n} \sum_{i=1}^n z_i \right\|^{2}  \leq \frac{1}{n^2} \sum_{i=1}^n C\E \left(\|\tth_{i-1}\|^{4}+h_i^{4} +1 \right)\leq Cn^{-1}.
\end{align*}
Therefore, condition~\eqref{eq:Gram-est2} is satisfied through Jensen's inequality.

If Assumption \ref{assumption1p}, \ref{assumption2}, \ref{assumption3}, \ref{assumption4} hold, then we have the estimate
\begin{align*}
\frac{1}{n}\sum_{i=1}^n \left\|\nabla F(\tth_{n})\nabla F(\tth_{n})^\top \right\| &\leq C \frac{1}{n}\sum_{i=1}^n\|\tth_{n}\|^2 \to 0, \\
\frac{1}{n}\sum_{i=1}^n \E \left\|\bxi_i \bxi_i^\top \right\| &\leq C \frac{1}{n}\sum_{i=1}^n h_n^2 \to 0, \\
\E_{n-1} \|\bgamma_n \bgamma_n^\top \| &\leq C\|\tth_{n-1}\|^2 + Ch_n^2, \\
\E_{n-1} \left\| \bvarepsilon_i \bvarepsilon_i^\top \right\| &\leq C(\|\tth_{n-1}\|^2+h_n^2 +1).
\end{align*}

We further claim that $\frac{1}{n}\sum_{i=1}^n \|\bgamma_i \bgamma_i^\top\|$ converges to 0 in probability. Because $ C\|\tth_{n-1}\|^2 + Ch_n^2$ converges almost surely to 0, for any $\varepsilon, \delta$, there exists $N$, such that
\[
\Pr\{C\|\tth_{n-1}\|^2 + Ch_n^2 < \delta, \forall n \geq N\} > 1 - \varepsilon.
\]
There exists $N'>0$, such that
\[
\Pr\{\frac{1}{N}\sum_{i=1}^N \|\bgamma_i \bgamma_i^\top\| > N' \} < \varepsilon.
\]
For any $m > N$,
\begin{align*}
&\quad \E[\|\bgamma_m \bgamma_m^\top\| | C\|\tth_{n-1}\|^2 + Ch_n^2 < \delta, \forall n \geq N] \\
& \leq \E[\|\bgamma_m \bgamma_m^\top\| I\{C\|\tth_{n-1}\|^2 + Ch_n^2 < \delta, \forall n \geq N\}]/(1 - \varepsilon)\\
& \leq \E[\|\bgamma_m \bgamma_m^\top\| I\{C\|\tth_{m-1}\|^2 + Ch_m^2 < \delta\}]/(1 - \varepsilon)\\
& \leq \E[\|\bgamma_m \bgamma_m^\top\| | C\|\tth_{m-1}\|^2 + Ch_m^2 < \delta]/(1 - \varepsilon) \leq \frac{\delta}{1 - \varepsilon}.
\end{align*}
So for any integer $K_1$, 
\[
\frac{1}{K_1N}\E[\sum_{m = N+1}^{K_1N}\|\bgamma_m \bgamma_m^\top\| | C\|\tth_{n-1}\|^2 + Ch_n^2 < \delta, \forall n \geq N] \leq \frac{(K_1-1)\delta}{K_1(1 - \varepsilon)}.
\]
For any $K_2 > 1$,
\[
\Pr\{\frac{1}{K_1N}\sum_{m = N+1}^{K_1N}\|\bgamma_m \bgamma_m^\top\|> \frac{K_2(K_1-1)\delta}{K_1(1 - \varepsilon)} |C\|\tth_{n-1}\|^2 + Ch_n^2 < \delta\} \leq \frac{1}{K_2},
\]
which implies 
\[
\Pr\{\frac{1}{K_1N}\sum_{m = N+1}^{K_1N}\|\bgamma_m \bgamma_m^\top\|> \frac{K_2(K_1-1)\delta}{K_1(1 - \varepsilon)}\} \leq \frac{1}{K_2}(1 - \varepsilon) + \varepsilon.
\]

So
\[
\Pr\{\frac{1}{K_1N}\sum_{m = 1}^{K_1N}\|\bgamma_m \bgamma_m^\top\|> \frac{K_2(K_1-1)\delta}{K_1(1 - \varepsilon)} + \frac{N'}{K_1}\} \leq \frac{1}{K_2}(1 - \varepsilon) + 2\varepsilon.
\]

We can choose approriate $K_2, \varepsilon, \delta$ to conclude that $\frac{1}{n}\sum_{i = 1}^{n}\|\bgamma_i \bgamma_i^\top\|$ converges to 0 in probability.

By similar arguments, the cross terms between $\nabla F(\tth_{i-1}), \bxi, \bgamma_n$ and $\bvarepsilon_i$ converge to 0 in probability. Finally,
\begin{align*}
\left\|\frac{1}{n}\sum_{i = 1}^{n} \bvarepsilon_i \bvarepsilon_i^\top -Q \right\| &\leq \|\frac{1}{n}\sum_{i = 1}^{n} z_i\| + \frac{1}{n}\sum_{i = 1}^{n} \left\|\E_{i-1}\bvarepsilon_i \bvarepsilon_i^\top -Q \right\| \\
&\leq \|\frac{1}{n}\sum_{i = 1}^{n}z_i\| + C \frac{1}{n}\sum_{i = 1}^{n} \left(\|\tth_{i-1}\|+\|\tth_{i-1}\|^3+h_i +h_i^3 \right),
\end{align*}
The term $\frac{1}{n}\sum_{i = 1}^{n} \left(\|\tth_{i-1}\|+\|\tth_{i-1}\|^3+h_i +h_i^3 \right)$ converge almost surely to 0. 

From the estimate
\[
\E_{n-1} \| z_n \|^{2}  \leq \E_{n-1} \| \bvarepsilon_n \|^{4} \leq C(\|\tth_{n-1}\|^{4}+h_n^{4} +1),
\]
we claim that $\|\frac{1}{n}\sum_{i = 1}^{n} z_i\|$ converges to 0. Recall the inequality~\ref{eq:asconv}
\[
\sum_{i=1}^\infty \frac{\|\tth_{i-1}\|^2}{i^{1/2}}< \infty.
\]
It imples that $\sum_{i=1}^\infty \frac{\|\tth_{i-1}\|^4}{i}< \infty$. So
\[
\sum_{i=1}^\infty \frac{\E_{i-1} \| z_i \|^{2}}{i^2} < \infty.
\]
By Theorem 2.18 in \cite{hall2014martingale}, we have $\frac{1}{n}\sum_{i = 1}^{n} z_i \to 0$ almost surely.

If Assumptions \ref{assumption4}, \ref{assumption6} hold, then the above proof holds without change.
\end{proof}

We now come back to the main proof of Theorem~\ref{thm:plugin}.
\begin{customthm}{\ref{thm:plugin}}
If Assumption \ref{assumption1p}, \ref{assumption2}, \ref{assumption3}, \ref{assumption4} hold, or Assumption \ref{assumption4}, \ref{assumption6} holds, then $\widehat{H}_n^{-1} \widehat{Q}_n \widehat{H}_n^{-1}$ converges in probability to $H^{-1} Q H^{-1}$.

Assume Assumptions \ref{assumption1p} to \ref{assumption4} hold for $\delta_1=+\infty$ and $\delta=2$. Set the step size as $\eta_n = \eta_0 n^{-\alpha}$ for some constant $\eta_0>0$ and $\alpha\in \left(\frac{1}{2}, 1\right)$, and the spacing parameter as $h_n = h_0 n^{-\gamma}$ for some constant $h_0 > 0$, and $\gamma \in \left(\frac{1}{2}, 1\right)$. We have
$\E \left\|  \widehat{H}_n^{-1} \widehat{Q}_n \widehat{H}_n^{-1} - H^{-1} Q H^{-1} \right \| \leq C n^{-\alpha/2}.
$\end{customthm}

\begin{proof}
If Assumption \ref{assumption1p}, \ref{assumption2}, \ref{assumption3}, \ref{assumption4} hold, or Assumption  \ref{assumption4}, \ref{assumption6} holds, then the only thing we need to check is that $\widehat{H}_n^{-1}$ converges to $H^{-1}$. For the thresholding estimator $\widehat{H}_n$, since $\| \widehat{H}_n - \widetilde{H}_n \| \leq \| \widetilde{H}_n - H \|$ by construction, we have the bound
\[
\| \widehat{H}_n - H \| \leq  \| \widetilde{H}_n - H \| +  \| \widehat{H}_n - \widetilde{H}_n \| \leq 2 \| \widetilde{H}_n - H \|.
\]
So $\widehat{H}_n^{-1}$ converges to $H^{-1}$.

Now assume Assumptions \ref{assumption1p} to \ref{assumption4} hold for $\delta_1=+\infty$ and $\delta=2$. The thresholding estimator $\widehat{H}_n$ has the rate below,
\begin{align}
\label{eq:res1}
\E \| \widehat{H}_n - H \|^2
&\leq  4 \E \| \widetilde{H}_n - H \|^2 \leq C n^{-\alpha},
\end{align}
where the last inequality from Lemma~\ref{lemma:hessian}.

By Lemma~\ref{lemma:inverse-mat}, the inverse matrix error satisfies,
\begin{align}
\label{eq:res2}
&\E \| \widehat{H}^{-1}_n - H^{-1}\|^2 \notag \\
\leq \;\; &\E \left[\1_{\|H^{-1}(\widehat{H}_n - H)\| \leq 1/2} 2 \| \widehat{H}_n - H\| \|H^{-1}\|^2 + \1_{\|H^{-1}(\widehat{H}_n - H)\| \geq 1/2} \| \widehat{H}_n^{-1} - H^{-1}\|\right]^2 \notag \\
\leq \;\; &8 \|H^{-1}\|^4 \E \| \widehat{H}_n - H\|^2 + 2 (\kappa_1^{-1} + \lambda_{\min}^{-1}(H))^2 \P \left(\|H^{-1}(\widehat{H}_n - H)\| \geq \frac{1}{2}\right) \notag \\
\leq \;\; &8 \|H^{-1}\|^4 \E \| \widehat{H}_n - H\|^2 + \frac{1}{2 \lambda^2} (\kappa_1^{-1} + \lambda_{\min}^{-1}(H))^2 \E \| \widehat{H}_n - H\|^2 \notag \\
\leq \;\; &C\, n^{-\alpha},
\end{align}
where the third inequality follows from Markov's inequality and the last one from \eqref{eq:res1}.

We now consider our target term, with our previous results \eqref{eq:res1}, \eqref{eq:res2}, and Lemma~\ref{lemma:Gram}, we can obtain that,
\begin{align*}
&\E \left \|\widehat{H}_n^{-1} \widehat{Q}_n \widehat{H}_n^{-1} - H^{-1} Q H^{-1} \right\| \\
=\;\; &\E \left \|\widehat{H}_n^{-1} (\widehat{Q}_n - Q) \widehat{H}_n^{-1} + (H^{-1} +\widehat{H}_n^{-1} - H^{-1}) Q (H^{-1} +\widehat{H}_n^{-1} - H^{-1}) - H^{-1} Q H^{-1}\right\| \\
\leq \; \; &\E \left\|\widehat{H}_n^{-1} (\widehat{Q}_n - Q) \widehat{H}_n^{-1}\right\| + \E \left\|H^{-1} Q (\widehat{H}_n^{-1} - H^{-1})\right\| +\E \left\|(\widehat{H}_n^{-1} - H^{-1}) Q H^{-1}\right\| \\
+ &\E \left\|(\widehat{H}_n^{-1} - H^{-1}) Q (\widehat{H}_n^{-1} - H^{-1})\right\| \\
\leq \;\; &\kappa_1^{-2} \E \left\| \widehat{Q}_n - Q \right\| + 2 \lambda^{-1} \| Q \| \E \left \| \widehat{H}_n^{-1} - H^{-1} \right \| + \| Q \| \E \left \| \widehat{H}_n^{-1} - H^{-1} \right \|^2 \\
\leq \;\;  &Cn^{-\alpha/2},
\end{align*}
which completes the proof.
\end{proof}

\subsubsection*{Proof of Theorem~\ref{thm:fixed-b}}
\begin{proof}
We first show that we can extend Theorem~\ref{thm:clt} and Theorem~\ref{thm:clt-qt} to the following form,
\begin{align*}
\frac{1}{\sqrt{n}} \sum_{i=1}^{\lfloor nr \rfloor} \tth_i \Longrightarrow \Sigma^{1/2} \boldsymbol{W}_r, \; \; r \in [0,1].
\end{align*}
where $\Sigma = H^{-1} Q H^{-1}$ and $\boldsymbol{W}_r$ is a $d$-dimensional vector of independent standard Brownian motions on $[0,1]$. For any $r \in [0,1]$, we consider the following partial summation process,
\begin{align*}
\overline{B}_n(r) = \frac{1}{n} \sum_{i=1}^{\lfloor nr \rfloor} \Delta_i,
\end{align*}
where $\Delta_i = \tth_i - \tth^\star = \tth_i$. Now consider the following alternative partial summation process,
\begin{align*}
\overline{B}^\prime_n(r) = \frac{1}{n} \sum_{i=1}^{\lfloor nr \rfloor} \Delta^\prime_i,
\end{align*}
where
\begin{align*}
\Delta^\prime_i &= \Delta^\prime_{i-1} - \eta_i H \Delta^\prime_{i-1} + \eta_n(\bgamma_n + \bvarepsilon_n), \; \; \Delta_0^\prime = \Delta_0 = \tth_0.
\end{align*}
In the proof of Theorem~\ref{thm:clt}, we follow the proof of Theorem 2 in \cite{polyak1992acceleration} and establish that $\sqrt{n}\|\overline{B}^\prime_n(1) - \overline{B}_n(1)\| = \sqrt{n} \|\bar\Delta^\prime_n - \bar\tth_n\| \to 0$ almost surely. In other words, with probability $1$, for any $\varepsilon$, there exists $N > 0$, such that for any $n > N$, $\sqrt{n}\|\overline{B}^\prime_n(1) - \overline{B}_n(1)\| < \varepsilon$. Let $M = \max_{1\leq i \leq N} \sqrt{n}\|\overline{B}^\prime_n(1) - \overline{B}_n(1)\|$. Then for any $n > \max\{N, \frac{M^2}{\varepsilon^2}N\}$,
\begin{align*}
\sqrt{n}\|\overline{B}^\prime_n(r) - \overline{B}_n(r)\| &= \frac{\sqrt{\lfloor nr \rfloor}}{\sqrt{n}} \sqrt{\lfloor nr \rfloor}\|\overline{B}^\prime_{\lfloor nr \rfloor}(1) - \overline{B}_{\lfloor nr \rfloor}(1)\|.
\end{align*}
Now when $\lfloor nr \rfloor > N$, we can establish the following,
\begin{align*}
\sqrt{n}\|\overline{B}^\prime_n(r) - \overline{B}_n(r)\| \leq \varepsilon \frac{\sqrt{\lfloor nr \rfloor}}{\sqrt{n}} \leq \varepsilon.
\end{align*}
Moreover, when $\lfloor nr \rfloor \leq N$, it can be derived that,
\begin{align*}
\sqrt{n}\|\overline{B}^\prime_n(r) - \overline{B}_n(r)\| &= \frac{\sqrt{\lfloor nr \rfloor}}{\sqrt{n}}\sqrt{\lfloor nr \rfloor} \|\overline{B}^\prime_{\lfloor nr \rfloor}(1) - \overline{B}_{\lfloor nr \rfloor}(1)\| \leq \sqrt{r}M \leq \varepsilon.
\end{align*}
Therefore, we have $\sqrt{n}\|\overline{B}^\prime_n(r) - \overline{B}_n(r)\| \leq \varepsilon$ for sufficiently large $n$. We can then further derive that $\sqrt{n} \sup_r \|\overline{B}^\prime_n(r) - \overline{B}_n(r)\| \to 0$ almost surely. Now we may consider the convergence of $\overline{B}^\prime_n(r)$ instead. The $\overline{B}^\prime_n(r)$ has the following decomposition,
\begin{align}
\label{eq:fixed-b-proof-decomp}
\sqrt{n} \overline{B}^\prime_n(r) = \frac{1}{\sqrt{n} \lfloor nr \rfloor \eta_{\lfloor nr \rfloor}} \tth_0 + \frac{1}{\sqrt{n}} \sum_{i=1}^{\lfloor nr \rfloor} H^{-1} (\bgamma_i + \bvarepsilon_i) + \frac{1}{\sqrt{n}} \sum_{i=1}^{\lfloor nr \rfloor} w_i^{\lfloor nr \rfloor} (\bgamma_i + \bvarepsilon_i),
\end{align}
where $\frac{1}{\sqrt{n}} \sum_{i=1}^n \| w_i^n\| \rightarrow 0$. The first term on the RHS converges to $0$ from above. For the third term on the RHS, because $\{\bgamma_i + \bvarepsilon_i\}$ is a martingale difference sequence, by Doob's martingale inequality,
\begin{align*}
\Pr\{\sup_{0\leq r \leq 1} \|\frac{1}{\sqrt{n}} \sum_{i=1}^{\lfloor nr \rfloor} w_i^{\lfloor nr \rfloor} (\bgamma_i + \bvarepsilon_i)\| \geq K \} &\leq \frac{1}{K}\E\{\|\frac{1}{\sqrt{n}} \sum_{i=1}^{n} w_i^{n} (\bgamma_i + \bvarepsilon_i)\|\}.
\end{align*}
In the event $\{C\|\tth_{n-1}\|^2 + C < N, \forall n > 0 \}$, $\E \|\bgamma_i + \bvarepsilon_i\|^2$ is bounded, so $\E \frac{1}{n}\sum_{i=1}^n \|w_i^{n} (\bgamma_i + \bvarepsilon_i)\|^2 \to 0$ conditioned on this event. Thus, in this event, we have
\begin{align*}
\Pr\left\{\sup_{0\leq r \leq 1} \|\frac{1}{\sqrt{n}} \sum_{i=1}^{\lfloor nr \rfloor} w_i^{\lfloor nr \rfloor} (\bgamma_i + \bvarepsilon_i)\| \geq K \right\} \to 0.
\end{align*}
Since the probability of this event can be arbitrarily close to $1$, and $K$ can be arbitrarily small, we can further have the convergence in probability below,
\begin{align*}
\sup_{0\leq r \leq 1} \|\frac{1}{\sqrt{n}} \sum_{i=1}^{\lfloor nr \rfloor} w_i^{\lfloor nr \rfloor} (\bgamma_i + \bvarepsilon_i)\| \to 0.
\end{align*}
Finally, for the second term on the RHS of Equation~\eqref{eq:fixed-b-proof-decomp}, we could use a direct application of functional martingale central limit theorem. Combining Theorem 4.1 from \cite{hall2014martingale} and Equation~\eqref{polyakcondition3}, we can establish that
\begin{align*}
\frac{1}{\sqrt{n}} \sum_{i=1}^{\lfloor nr \rfloor} H^{-1} (\bgamma_n + \bvarepsilon_n) \Longrightarrow \Sigma^{1/2} \boldsymbol{W}_r.
\end{align*}
In conclusion, we have proved that
\begin{align*}
\frac{1}{\sqrt{n}} \sum_{i=1}^{\lfloor nr \rfloor} \tth_i \Longrightarrow \Sigma^{1/2} \boldsymbol{W}_r, \; \; r \in [0,1].
\end{align*}
Therefore, for any $\w \in \R^d$,  we have
\begin{align*}
C_n(r) = \frac{1}{\sqrt{n}} \sum_{i=1}^{\lfloor nr \rfloor} \w^\top \tth_i \Rightarrow \w^\top (\w^\top \Sigma \w)^{1/2} W_r, \; \; r \in [0,1].
\end{align*}
Here $W_r$ is the standard one dimensional Brownian motion. In addition,
\begin{align*}
\w^\top V_n \w = \frac{1}{n} \sum_{i=1}^n \left[C_n \left(\frac{i}{n}\right) - \frac{i}{n}C_n(1)\right] \left[C_n \left(\frac{i}{n}\right) - \frac{i}{n}C_n(1)\right]^\top.
\end{align*}
Notice that $\w^\top (\overline{\tth}_n) = \frac{1}{\sqrt{n}} C_n(1)$, and
\begin{align*}
n \frac{(\w^\top \overline{\tth}_n)^2}{\w^\top V_n \w} \Rightarrow \frac{W_1^2}{\int_0^1 (W_r - r W_1)^2 \intd r},
\end{align*}
using the continuous mapping theorem.

\end{proof}

\subsubsection{Technical Lemmas}
\label{sec:app-d}

The following lemma is from \cite{assouad1975espaces}. We include the proof here.
\begin{lemma}[\citet{assouad1975espaces}]
\label{lemmamoments}
Let $\{\X_n\}$ be a martingale difference sequence, i.e. $\E[\X_n | \X_{n-1}] = 0$. For any $1 \leq p \leq 2$ and any norm $\|\cdot \|$ on $\mathbb{R}^d$, there exists a constant $C$ such that
\begin{align*}
\E \left\|\sum_{i=1}^n \X_i \right\|^p \leq C \sum_{i=1}^n\E\left[\| \X_i\|^p\right].
\end{align*}
\end{lemma}
\begin{proof}
We would like to show that there exists a constant $C$ (which depends on $d$ and $p$) such that for any $\a, \b \in \R^d$,
\begin{align*}
\frac{1}{2} \left(\|\a+\b\|_2^p + \|\a-\b\|_2^p\right) \leq \|\a\|_2^p + C\|\b\|_2^p,
\end{align*}
where $\|\cdot\|_2$ is the 2-norm. To see this, in the one dimensional case, this is equivalent to
\begin{align*}
\frac{1}{2} \left(|1+x|^p + |1-x|^p\right) \leq 1 + C|x|^p.
\end{align*}
At $x = 1$, the left hand side is differentiable and its first derivative is 0, so there exists a constant $C$ such that the inequality holds in a neighborhood of $x = 1$. At $x \to \pm \infty$, the inequality also holds with some constant $C$. So it is easy to find a constant $C$ such that the inequality holds for all $x$. The proof for the $d$-dimensional case is the same.

Using the above inequality, we have
\begin{align*}
\E_{n-1} \left\|\sum_{i=1}^n \X_i \right\|_2^p &= \E_{n-1} \left\|\sum_{i=1}^{n-1} \X_i + \X_n \right\|_2^p\\
&\leq 2 \left\|\sum_{i=1}^{n-1} \X_i\right\|_2^p + 2C \E_{n-1} \|\X_n\|_2^p- \E_{n-1} \left\|\sum_{i=1}^{n-1} \X_i - \X_n \right\|_2^p.
\end{align*}
On the other hand,
\begin{align*}
\E_{n-1} \left\|\sum_{i=1}^{n-1} \X_i - \X_n \right\|_2^p \geq \left\|\sum_{i=1}^{n-1} \X_i - \E_{n-1}\X_n \right\|_2^p =  \left\|\sum_{i=1}^{n-1} \X_i\right\|_2^p.
\end{align*}
So
\[
\E_{n-1} \left\|\sum_{i=1}^n \X_i \right\|_2^p \leq \left\|\sum_{i=1}^{n-1} \X_i\right\|_2^p + 2C \E_{n-1} \|\X_n\|_2^p.
\]
By induction, we then have
\begin{align*}
\E\left\|\sum_{i=1}^n \X_i\right\|_2^p \leq 2C \sum_{i=1}^n\E \left[\| \X_i\|_2^p\right].
\end{align*}
For any general norm, there exists a constant $C$ such that
\begin{align*}
\frac{1}{C}\|X\| \leq \|X\|_2 \leq C\|X\|.
\end{align*}
So the same result holds for any norm.
\end{proof}

We now provide a matrix perturbation inequality from \cite{chen2020statistical}.
\begin{lemma}
\label{lemma:inverse-mat}
If a matrix $B=A+E$ where $A$ and $B$ are invertible, we have,
\begin{align*}
\left\|B^{-1}-A^{-1}\right\| \leq \| A^{-1} \|^2 \|E \| \frac{1}{1 - \| A^{-1} E \|}.
\end{align*}
\end{lemma}

\begin{proof}
Notice that
\begin{align*}
B^{-1} = (A + E)^{-1} &= A^{-1} - A^{-1} \left(A^{-1} + E^{-1}\right)^{-1} A^{-1} \\
&= A^{-1} - A^{-1} E \left(A^{-1}E + I\right)^{-1} A^{-1}.
\end{align*}
Therefore, the inversion error is,
\begin{align*}
\| B^{-1} - A^{-1} \|
&= \left \| A^{-1} E \left(A^{-1}E + I\right)^{-1} A^{-1}\right \| \\
&\leq \| A^{-1} \|^2 \|E \| \| (A^{-1}E + I)^{-1} \| \\
&\leq \| A^{-1} \|^2 \|E \| \frac{1}{\lambda_{\min}(A^{-1} E + I)} \\
&\leq \| A^{-1} \|^2 \|E \| \frac{1}{1 - \| A^{-1} E \|},
\end{align*}
where we use Weyl's inequality in the last inequality.
\end{proof}

\subsection{Finite-difference stochastic Newton method}
\label{sec:newton}

As a by-product and an application, the online finite-difference estimator of Hessian in \eqref{eq:hessian-hat} enables us to develop the {\kw} version of the stochastic Newton's method. Existing literature that handles the {\rm} version of the stochastic Newton's method traces back to \cite{ruppert1985newton}.
Given an initial point $\tth_0$, the {\kw} stochastic Newton's method has the following updating rule, 
\begin{align}
\label{eq:newton-update}
\tth_n = \tth_{n-1} - \frac{1}{n} \widehat H_{n-1}^{-1} \widehat{g}_{h_n,\v_n}(\tth_{n-1}; \zet_{n}),
\end{align}
Here $\widehat H_n^{-1}$ a recursive estimator of $H^{-1}$. We modify the thresholding Hessian estimator $\widehat{H}_n $ in \eqref{eq:hessian-hat} as follows. Let $U \widetilde\Lambda_n U^\top$ be the eigenvalue decomposition of $\widetilde{H}_n$ in \eqref{eq:hessian-est2}, and define
\begin{align}\label{eq:c.19}
\widehat{H}_n = U \widehat{\Lambda}_n U^\top, \; \; \widehat{\Lambda}_{n,kk} =  \max \left\{\kappa_1, \min \big\{\kappa_2, \widetilde\Lambda_{n,kk}\big\}\right\}, \; k=1,2,\dots,d,
\end{align}
for some constants $0 < \kappa_1 < \lambda < L_f < \kappa_2$, where $\lambda, L_f$ are defined in Assumption~\ref{assumption1p}. 

\begin{theorem} \label{thm:newton} 
Assume Assumptions \ref{assumption1p} to \ref{assumption4} hold for $\delta_1=+\infty$ and $\delta=2$. The Hessian estimator $\widehat{H}_n$ in \eqref{eq:c.19} converges in probability to the empirical Hessian matrix $H$. The stochastic Newton estimator $\tth_n$ in \eqref{eq:newton-update} converges to $\tth^\star$ almost surely. The stochastic Newton estimator $\tth_n$ has the following limiting distribution,
\begin{align}
\label{eq:newton-clt}
\sqrt{n} \left(\tth_n - \tth^\star\right) \Longrightarrow \N\left(\0, H^{-1} Q H^{-1}\right),
\end{align}
for the same $Q$ as in Theorem~\ref{thm:clt}.
\end{theorem}

Theorem \ref{thm:newton} states that the final iterate of the {\kw} stochastic Newton method \eqref{eq:newton-update} entails the same asymptotic distribution as the averaged {\akw} estimator \eqref{eq:akw}. In contrast to {\akw}, \eqref{eq:newton-update} leverages additional Hessian information to achieve the asymptotic normality and efficiency. Nevertheless, the numerical implementation of the {\kw} stochastic Newton's method requires to update a Hessian estimator $\widehat H_{n}$ in all iterations, which demands significant additional computation unless such an estimator is yet computed and maintained along the procedure for other purposes.

\subsubsection*{Proof of Theorem~\ref{thm:newton}}
\begin{proof}
Notice that
\begin{align*}
\tth_n = \tth_{n-1} - \frac{1}{n} \widehat H_{n-1}^{-1} \nabla F(\tth_{n-1}) + \frac{1}{n} \widehat H_{n-1}^{-1} \left(\bxi_{n} + \bgamma_n + \bvarepsilon_n\right),
\end{align*}
where $\bxi_n, \bgamma_n, \bvarepsilon_n$ are defined at the beginning of the supplement. We can deduce
\begin{align*}
\E_{n-1} F(\tth_n) &\leq  F(\tth_{n-1}) + \E_{n-1}\langle \nabla F(\tth_{n-1}), - \frac{1}{n} \widehat H_{n-1}^{-1} \nabla F(\tth_{n-1}) + \frac{1}{n}\widehat H_{n-1}^{-1} \left(\bxi_{n} + \bgamma_n + \bvarepsilon_n\right) \rangle,\\
&\quad + C\E_{n-1}\|- \frac{1}{n} \widehat H_{n-1}^{-1} \nabla F(\tth_{n-1}) + \frac{1}{n} \widehat H_{n-1}^{-1} \left(\bxi_{n} + \bgamma_n + \bvarepsilon_n\right)\|^2\\
&\leq F(\tth_{n-1})-\frac{C}{n}\|\nabla F(\tth_{n-1})\|^2 + \frac{Ch_n}{n}\|\nabla F(\tth_{n-1})\|\\
&\quad + \frac{C}{n^2}(\|\nabla F(\tth_{n-1})\|^2 + \|\tth_{n-1}\|^2 + 1)\\
&\leq (1-\frac{C}{n}+\frac{Ch_n}{n})F(\tth_n) + \frac{C}{n^2}(F(\tth_n) + 1).
\end{align*}
which implies that $F(\tth_n)$ converges almost surely to 0, so $\tth_n$ converges almost surely to 0. Therefore, $\tth_n \rightarrow 0$ almost surely by martingale convergence theorem \citep{robbins1951stochastic}. It also implies the convergence rate
\begin{align} \label{eq:rate-hessian-1}
\E_{n-1} \| \tth_n \|^2 \leq C \left(n^{-C_1/2} + n^{-1}\right).
\end{align}
Similarly,
\begin{align} \label{eq:rate-hessian-2}
\E_{n-1} \| \tth_n \|^{2+\delta} \leq C \left(n^{-C_1/2} + n^{-(1+\delta)}\right).
\end{align}
Those results show that $\widehat H_n^{-1}$ converges to $H^{-1}$ by the arguments in the proof of Theorem~\ref{thm:plugin}. Now we consider the limiting distribution. 
\begin{align*}
\tth_n
&= \tth_{n-1} - \frac{1}{n} H^{-1} \nabla F(\tth_{n-1}) - \frac{1}{n} \left(\widehat H_{n-1}^{-1} - H^{-1}\right) \nabla F(\tth_{n-1}) + \frac{1}{n} \widehat H_{n-1}^{-1} \left(\bxi_{n} + \bgamma_n + \bvarepsilon_n\right) \\
&= \left(1 - \frac{1}{n}\right)\tth_{n-1} - \frac{1}{n} H^{-1} \bdelta_n - \frac{1}{n} \left(\widehat H_{n-1}^{-1} - H^{-1}\right) \nabla F(\tth_{n-1}) + \frac{1}{n} \widehat H_{n-1}^{-1} \left(\bxi_{n} + \bgamma_n + \bvarepsilon_n\right),
\end{align*}
where $\bdelta_n = \nabla F(\tth_{n-1}) - H \tth_{n-1}$. By induction, we can find that
\begin{align*}
\tth_n =
\frac{1}{n} \sum_{k=0}^{n-1} H_{k}^{-1} \bvarepsilon_{k+1} + \frac{1}{n} \sum_{k=0}^{n-1} \widehat H_{k}^{-1} \left(\bxi_{k+1} + \bgamma_{k+1}\right)
- \frac{1}{n} H^{-1} \sum_{k=0}^{n-1} \bdelta_{k+1} - \frac{1}{n} \sum_{k=0}^{n-1} \left(\widehat H_{k}^{-1} - H^{-1}\right) \nabla F(\tth_{k}).
\end{align*}
The last three terms in the RHS above all converge to zero. We only need to show that\\
$\frac{1}{\sqrt{n}} \sum_{k=0}^{n-1} H_{k}^{-1} \bvarepsilon_{k+1}$ converges to a normal distribution. Consider
\begin{align*}
\E_k \left[\widehat H_{k}^{-1} \bvarepsilon_{k+1} \bvarepsilon_{k+1}^\top \widehat H_{k}^{-1}\right] = \widehat H_{k}^{-1} \E_k \left[\bvarepsilon_{k+1} \bvarepsilon_{k+1}^\top\right] \widehat H_{k}^{-1},
\end{align*}
recall that in \eqref{expectationoveru} we have shown that $\E_k \left[\bvarepsilon_{k+1} \bvarepsilon_{k+1}^\top\right]$ converges almost surely to $Q$. Therefore, $\E_k \left[\widehat H_{k}^{-1} \bvarepsilon_{k+1} \bvarepsilon_{k+1}^\top \widehat H_{k}^{-1}\right]$ converges in probability to $H^{-1} Q H^{-1}$.

Finally, by martingale central limit theorem \citep[Theorem 2.1.9]{duflo1997random},
\begin{align*}
\frac{1}{\sqrt{n}} \sum_{k=0}^{n-1} \widehat H_{k-1}^{-1} \bvarepsilon_{k+1} \Longrightarrow \mathcal{N}\left(0, H^{-1} Q H^{-1}\right).
\end{align*}
\end{proof}

\section{Additional Results of Numerical Experiments}\label{sec:simu_supp}
In this section, we present additional simulation results. 

\subsection{Choices of the search direction distribution}

In this subsection, we compare the results for different directions $\mathcal{P}_{\v}$ for linear and logistic regression. We fix $n=10^5$, $d=20$ and report the results for linear regression in Table \ref{table:linearchoices} and those for logistic regression in Table \ref{table:logisticchoices}. The specification of the stepsizes and spacing parameters in this subsection are the same as those in Table \ref{table:inference} of Section \ref{sec:exp}.  Tables \ref{table:linearchoices}--\ref{table:logisticchoices} suggest the {\akw} algorithms with search directions {\ti}, {\ts}, {\tg} achieve similar performance for parameter estimation error and average coverage rates, while the average confidence intervals of {\tg} are generally larger. The observations in the numerical experiments match our Proposition \ref{cor:clt0} in Section \ref{sec:choice_dir}.

\begin{table}[t]
  \centering
  \begin{tabular}{lcc|c|ccc|ccc}
    \toprule
  $d$ &  $\Sigma$ & $\mathcal{P}_{\v}$ & Estimation error & \multicolumn{3}{c|}{Average coverage rate} & \multicolumn{3}{c}{Average length}\\
    &&& (standard error) & Plug-in  & Oracle & Fixed-$b$ & Plug-in  & Oracle & Fixed-$b$  \\
    \hline    &                       &{\ti}  & 0.015	(	0.005	)&	0.944	&	0.938	&	0.940	&	0.028	&	0.028	&	0.036	\\
    &                       &{\ts}& 0.015	(	0.005	)&	0.952	&	0.950	&	0.942	&	0.028	&	0.028	&	0.036	\\
    \multirow{-3}{*}{5}&   \multirow{-3}{*}{Identity} &{\tg}& 0.018	(	0.006	)&	0.964	&	0.962	&	0.942	&	0.033	&	0.033	&	0.042	\\
    \hline
    &                       &{\ti}  & 0.017	(	0.006	)&	0.958	&	0.954	&	0.946	&	0.032	&	0.032	&	0.041	\\
    &                       &{\ts}& 0.017	(	0.005	)&	0.960	&	0.958	&	0.934	&	0.031	&	0.031	&	0.039	\\
    \multirow{-3}{*}{5}& \multirow{-3}{*}{Equicorr}   &{\tg}& 0.020	(	0.006	)&	0.960	&	0.952	&	0.918	&	0.037	&	0.037	&	0.046	\\
   \hline
    &                      &{\ti}  & 0.066	(	0.010	)&	0.943	&	0.938	&	0.928	&	0.058	&	0.056	&	0.074	\\
    &                       &{\ts}& 0.066	(	0.010	)&	0.953	&	0.942	&	0.930	&	0.058	&	0.056	&	0.073	\\
    \multirow{-3}{*}{20}& \multirow{-3}{*}{Identity}   &{\tg} & 0.070	(	0.010	)&	0.948	&	0.936	&	0.929	&	0.061	&	0.058	&	0.076	\\ \hline
    &                      &{\ti}  & 0.082	(	0.014	)&	0.938	&	0.931	&	0.923	&	0.071	&	0.068	&	0.087	\\
    &                       &{\ts}& 0.081	(	0.011	)&	0.947	&	0.937	&	0.919	&	0.070	&	0.067	&	0.085	\\
    \multirow{-3}{*}{20}& \multirow{-3}{*}{Equicorr}   &{\tg} &  0.085	(	0.012	)&	0.947	&	0.936	&	0.923	&	0.074	&	0.071	&	0.089	\\ \hline
    &                      &{\ti}  &0.180	(	0.018	)&	0.947	&	0.917	&	0.881	&	0.097	&	0.089	&	0.108	\\
    &                       &{\ts}& 0.179	(	0.019	)&	0.945	&	0.920	&	0.876	&	0.097	&	0.089	&	0.109	\\
    \multirow{-3}{*}{50}& \multirow{-3}{*}{Identity}   &{\tg} &0.184	(	0.019	)&	0.943	&	0.916	&	0.875	&	0.100	&	0.091	&	0.110	\\ \hline
    &                      &{\ti}  & 0.227	(	0.026	)&	0.937	&	0.912	&	0.860	&	0.121	&	0.110	&	0.126	\\
    &                       &{\ts}& 0.224	(	0.024	)&	0.940	&	0.909	&	0.859	&	0.120	&	0.109	&	0.126	\\
    \multirow{-3}{*}{50}& \multirow{-3}{*}{Equicorr}   &{\tg} & 0.229	(	0.025	)&	0.942	&	0.913	&	0.859	&	0.123	&	0.112	&	0.129	\\   \bottomrule
  \end{tabular}
  \caption{Comparison among different direction distributions $\mathcal{P}_{\v}$ (Detailed specification of {\ti,\ts,\tg} can be referred to Section \ref{sec:choice_dir}). We consider the linear regression model, and the {\akw} estimators are computed based on the case of two function queries ($m=1$). Corresponding standard errors are reported in the brackets.}
  \label{table:linearchoices}
\end{table}

\begin{table}[t]
  \centering
  \begin{tabular}{lcc|c|ccc|ccc}
    \toprule
  $d$ &  $\Sigma$ & $\mathcal{P}_{\v}$ & Estimation error & \multicolumn{3}{c|}{Average coverage rate} & \multicolumn{3}{c}{Average length}\\
    &&& (standard error) & Plug-in  & Oracle & Fixed-$b$ & Plug-in  & Oracle & Fixed-$b$  \\
    \hline    &                       &{\ti}  & 0.037	(	0.011	)&	0.946	&	0.938	&	0.916	&	0.065	&	0.065	&	0.075	\\
    &                       &{\ts}&0.034	(	0.012	)&	0.958	&	0.962	&	0.916	&	0.065	&	0.065	&	0.076	\\
        \multirow{-3}{*}{5}& \multirow{-3}{*}{Identity}   &{\tg}&0.041	(	0.014	)&	0.962	&	0.960	&	0.932	&	0.077	&	0.076	&	0.089	\\ \hline
    &                      &{\ti}  &0.042	(	0.015	)&	0.934	&	0.932	&	0.908	&	0.073	&	0.073	&	0.085	\\
    &                       &{\ts}&0.039	(	0.013	)&	0.944	&	0.942	&	0.944	&	0.072	&	0.071	&	0.087	\\
    \multirow{-3}{*}{5}& \multirow{-3}{*}{Equicorr}   &{\tg} &0.047	(	0.016	)&	0.950	&	0.948	&	0.924	&	0.085	&	0.084	&	0.103	\\ \hline

    &                      &{\ti}  &0.152	(	0.025	)&	0.943	&	0.937	&	0.862	&	0.128	&	0.125	&	0.136	\\
    &                       &{\ts}&0.152	(	0.024	)&	0.932	&	0.928	&	0.874	&	0.127	&	0.125	&	0.136	\\
    \multirow{-3}{*}{20}& \multirow{-3}{*}{Identity}   &{\tg} &0.158	(	0.024	)&	0.940	&	0.933	&	0.873	&	0.134	&	0.131	&	0.139	\\ \hline
    &                      &{\ti}  &0.177	(	0.030	)&	0.939	&	0.935	&	0.848	&	0.154	&	0.150	&	0.158	\\
    &                       &{\ts}& 0.179	(	0.030	)&	0.942	&	0.937	&	0.835	&	0.152	&	0.149	&	0.155	\\
    \multirow{-3}{*}{20}& \multirow{-3}{*}{Equicorr}   &{\tg} &  0.187	(	0.030	)&	0.942	&	0.936	&	0.844	&	0.160	&	0.156	&	0.160	\\ \hline
    &                      &{\ti}  & 0.404	(	0.040	)&	0.914	&	0.912	&	0.688	&	0.199	&	0.197	&	0.140	\\
    &                       &{\ts}& 0.409	(	0.041	)&	0.912	&	0.910	&	0.677	&	0.199	&	0.197	&	0.136	\\
    \multirow{-3}{*}{50}& \multirow{-3}{*}{Identity}   &{\tg} &0.415	(	0.042	)&	0.914	&	0.912	&	0.677	&	0.203	&	0.201	&	0.140	\\ \hline
    &                      &{\ti}  & 0.495	(	0.051	)&	0.920	&	0.917	&	0.620	&	0.245	&	0.241	&	0.142	\\
    &                       &{\ts}& 0.490	(	0.048	)&	0.918	&	0.916	&	0.623	&	0.243	&	0.240	&	0.140	\\
    \multirow{-3}{*}{50}& \multirow{-3}{*}{Equicorr}   &{\tg} & 0.498	(	0.049	)&	0.920	&	0.918	&	0.621	&	0.249	&	0.245	&	0.141	\\  \bottomrule
  \end{tabular}
  \caption{Comparison among different direction distributions $\mathcal{P}_{\v}$ (Detailed specification of {\ti,\ts,\tg} can be referred to Section \ref{sec:choice_dir}). We consider the logistic regression model, and the {\akw} estimators are computed based on the case of two function queries ($m=1$). Corresponding standard errors are reported in the brackets.}
  \label{table:logisticchoices}
\end{table}

\subsection{Multi-query {\akw} estimator}

We further conduct experiments for the {\kw} algorithm with multiple function-value queries ($m>1$) and compare the performance of $m=5, 10, 20$ using different search directions with sampling schemes \hyperref[ti]{\tt{(I+WR)}}, \hyperref[ti]{\tt{(I+WOR)}}, and \hyperref[ts]{\tt{(S)}}. We note that \hyperref[ti]{\tt{(I+WR)}} and  \hyperref[ti]{\tt{(I+WOR)}} refer to the uniform sampling from natural basis with and without replacement, respectively; and \hyperref[ts]{\tt{(S)}} refers to the uniform sampling from the sphere.  We fix the sample size $n=10^5$ and dimension $d=20$ and report the results of the linear regression with Identity and Equicorr designs in Table \ref{table:linear4}. The specification of the stepsizes and spacing parameters in this subsection are the same as the experiments in Table \ref{table:inference} of Section \ref{sec:exp} except for $\eta_0$ is multiplied by $\sqrt{m}$.  

It can be clearly seen from the table that the {\kw} algorithm achieves similar performance in both estimation and inference. Among the three sampling schemes, the algorithm with {\tt{(I+WOR)}} achieves lower estimation error than the other two sampling schemes and constructs around 10\% to 30\% shorter confidence intervals on average while achieving comparable coverage rates, which validates our theoretical results in Section \ref{sec:multiple}.

\subsection{Computation Time}
\label{subsec:compp}
In this subsection, we provide comparisons of the computation time of the plug-in and fixed-$b$ HAR schemes. We report the computation time, the estimation error, and the average coverage rate and length of these candidates based on $100$ replications. The specifications in the experiments are the same as those in Table \ref{table:inference} of Section \ref{sec:exp}. The computation time is recorded in a simulation environment running Python 3.8 with compute nodes equipped with  dual CPU sockets of 24-core Intel Cascade Lake Platinum 8268 chips. As can be seen from Table~\ref{table:inference-p1}, the fixed-$b$ HAR approach gives the fastest execution while the plug-in method provides a higher coverage rate with shorter average lengths. In practice, we would recommend the fixed-$b$ HAR method for those computation-sensitive tasks, and the plug-in method  in less computation-sensitive tasks.

\begin{table}[t]
  \centering
  \begin{tabular}{lcc|c|ccc|ccc}
    \toprule
  $m$ &  $\Sigma$ & $\mathcal{P}_{\v}$ & Estimation error & \multicolumn{3}{c|}{Average coverage rate} & \multicolumn{3}{c}{Average length}\\
    &&& (standard error) & Plug-in  & Oracle & Fixed-$b$ & Plug-in  & Oracle & Fixed-$b$  \\
    \hline    &                       &{\tt{(I+WOR)}}  & 0.029	(	0.004	)&	0.946	&	0.943	&	0.922	&	0.025	&	0.025	&	0.031	\\
    &                       &{\tt{(I+WR)}}& 0.031	(	0.005	)&	0.953	&	0.952	&	0.934	&	0.028	&	0.027	&	0.033	\\
    \multirow{-3}{*}{5}&   \multirow{-3}{*}{Identity} &{\ts}&  0.031	(	0.005	)&	0.940	&	0.939	&	0.933	&	0.028	&	0.027	&	0.035	\\
    \hline    &                       &{\tt{(I+WOR)}}  & 0.034	(	0.005	)&	0.951	&	0.949	&	0.935	&	0.030	&	0.030	&	0.037	\\
    &                       &{\tt{(I+WR)}}& 0.037	(	0.006	)&	0.955	&	0.953	&	0.926	&	0.033	&	0.033	&	0.040	\\
    \multirow{-3}{*}{5}&   \multirow{-3}{*}{Equicorr} &{\ts}&  0.037	(	0.006	)&	0.947	&	0.945	&	0.924	&	0.033	&	0.033	&	0.041	\\
    \hline    &                       &{\tt{(I+WOR)}}  & 0.020	(	0.003	)&	0.945	&	0.945	&	0.930	&	0.018	&	0.018	&	0.022	\\
    &                       &{\tt{(I+WR)}}& 0.024	(	0.004	)&	0.944	&	0.943	&	0.927	&	0.021	&	0.021	&	0.026	\\
    \multirow{-3}{*}{10}&   \multirow{-3}{*}{Identity} &{\ts}&  0.024	(	0.004	)&	0.958	&	0.956	&	0.926	&	0.021	&	0.021	&	0.026	\\
    \hline    &                       &{\tt{(I+WOR)}}  &0.023	(	0.004	)&	0.949	&	0.949	&	0.927	&	0.021	&	0.021	&	0.025	\\
    &                       &{\tt{(I+WR)}}& 0.028	(	0.005	)&	0.956	&	0.954	&	0.926	&	0.025	&	0.025	&	0.031	\\
    \multirow{-3}{*}{10}&   \multirow{-3}{*}{Equicorr} &{\ts}& 0.029	(	0.005	)&	0.943	&	0.943	&	0.922	&	0.025	&	0.025	&	0.030	\\\hline        &                       &{\tt{(I+WOR)}}  &0.015	(	0.002	)&	0.942	&	0.942	&	0.920	&	0.013	&	0.013	&	0.016	\\
    &                       &{\tt{(I+WR)}}&0.020	(	0.003	)&	0.944	&	0.943	&	0.921	&	0.018	&	0.017	&	0.021	\\
    \multirow{-3}{*}{20}&   \multirow{-3}{*}{Identity} &{\ts}&0.020	(	0.003	)&	0.945	&	0.945	&	0.927	&	0.018	&	0.017	&	0.022	\\
    \hline    &                       &{\tt{(I+WOR)}}  & 0.016	(	0.003	)&	0.941	&	0.941	&	0.932	&	0.014	&	0.014	&	0.017	\\
    &                       &{\tt{(I+WR)}}& 0.023	(	0.004	)&	0.949	&	0.948	&	0.922	&	0.020	&	0.020	&	0.025	\\
    \multirow{-3}{*}{20}&   \multirow{-3}{*}{Equicorr} &{\ts}& 0.023	(	0.004	)&	0.946	&	0.946	&	0.926	&	0.020	&	0.020	&	0.025	\\
\bottomrule
  \end{tabular}
  \caption{Comparison among different sampling schemes for multi-query algorithms under linear regression model with dimension $d=20$ and $m=5, 10, 20$, respectively (Detailed specification of {\tt{(I+WOR)},\tt{(I+WR)},\tt{(S)}} can be referred to Sections \ref{sec:choice_dir} and \ref{sec:multiple}). Corresponding standard errors are reported in the brackets.}
  \label{table:linear4}
\end{table}

\begin{table}[t]
  \centering{}
  \small
  \begin{tabular}{ccc|cc|cc}
    \toprule
    Model & $d$  & Estimator & Comp. & Estimation & Coverage & Average \\
    &&&time& error (s.e.) & Rate & length\\
    \hline
    && Plug-in& 19.169   & 0.015	(	0.005	)   & 0.944   & 0.028    \\
    &\multirow{-2}{*}{5}& Fixed-$b$& 9.692   & 0.016	(	0.005	)  & 0.940   & 0.036    \\
     &      & Plug-in&36.468& 0.066	(	0.010	) & 0.943 & 0.058 \\
    &\multirow{-2}{*}{20}&Fixed-$b$& 9.702   &0.068	(	0.011	)  & 0.928  & 0.074    \\
    && Plug-in & 214.470   & 0.180	(	0.018	) & 0.947   & 0.097    \\
     \multirow{-6}{*}{Linear}& \multirow{-2}{*}{50}           & Fixed-$b$ &13.672 & 0.179	(	0.018	)& 0.881& 0.108\\
    \hline
    && Plug-in& 21.149   & 0.037	(	0.011	)   & 0.946   & 0.065    \\
    &\multirow{-2}{*}{5}& Fixed-$b$& 10.300   & 0.035	(	0.011	)   & 0.916   & 0.075    \\
     &      & Plug-in&44.976&0.152	(	0.025	) & 0.943 & 0.128\\
    &\multirow{-2}{*}{20}&Fixed-$b$& 12.155   & 0.147	(	0.026	)   & 0.862  & 0.136    \\
    && Plug-in & 224.102   & 0.404	(	0.040	)   & 0.914   &0.199 \\
     \multirow{-6}{*}{Logistic}& \multirow{-2}{*}{50}           & Fixed-$b$ & 14.265 &0.402	(	0.044	) & 0.688 & 0.140\\
    \bottomrule
  \end{tabular}
  \caption{Computation time, estimation errors, averaged coverage rates, and average lengths of the proposed algorithm with search direction {\ti} and two function queries ($m=1$). Sample size $n=10^5$, dimension $d=5,20,50$ under the linear and logistic regression model with identity design matrix. Corresponding standard errors are reported in the brackets. }
  \label{table:inference-p1}
\end{table}

\clearpage
\bibliographystyle{chicago-ff}
\bibliography{refs}

\clearpage

\bibliographystyle{chicago-ff}
\bibliography{refs}

\end{document}